\documentclass[12pt,a4paper,utf8x]{article}
\usepackage{color}
\usepackage{fancyhdr}
\usepackage{tikz}
\usepackage[vmargin=70pt, hmargin=35pt]{geometry}
\usetikzlibrary{decorations.pathmorphing}
\usepackage[T1]{fontenc}
\usepackage{ucs}
\usepackage[utf8x]{inputenc}
\usepackage{amsfonts}
\usepackage{listings}
\usepackage{amsmath}
\usepackage[mathscr]{euscript}
\usepackage{amssymb}
\usepackage{amsfonts}
\usepackage{dsfont}
\usepackage{amsthm}
\usepackage{fancybox}
\usepackage{enumitem}
\usepackage{dsfont}
\usepackage{algpseudocode}
\usepackage{algorithm}
\usepackage{graphicx}
\usepackage{caption}
\usepackage{subcaption}
\usepackage{float}
\usepackage{authblk}
\usepackage{relsize}
\usepackage{setspace}
\usepackage{parskip}

\DeclareMathOperator{\LL}{L}
\DeclareMathOperator{\Tr}{Tr}

\newcommand*{\DomA}{\mathcal{D}(A)}

\makeatletter
\newcommand*{\inlineequation}[2][]{%
  \begingroup
    \refstepcounter{equation}%
    \ifx\\#1\\%
    \else
      \label{#1}%
    \fi
    \relpenalty=10000 %
    \binoppenalty=10000 %
    \ensuremath{%
      #2%
    }%
    ~\@eqnnum
  \endgroup
}
\makeatother

\usepackage{hyperref}

\newenvironment{alphafootnotes}
  {\par\edef\savedfootnotenumber{\number\value{footnote}}
   
   \setcounter{footnote}{0}}
  {\par\setcounter{footnote}{\savedfootnotenumber}}

\newtheorem{mytheo}{Theorem}[section]
\newtheorem{myprop}{Proposition}[section]
\newtheorem{lemme}{Lemma}[section]

\newtheorem{remark}{Remark}[section]

\date{}

\begin{document}

\title{Recursive computation of the invariant distributions of Feller processes:\\
Revisited examples and new applications}
\author[1]{Gilles Pag\`es}
\author[1]{Cl\'ement Rey}
\affil[1]{Universit\'e Pierre et Marie Curie, LPMA, 4 Place Jussieu, 75005 Paris, France}
\maketitle
\begin{alphafootnotes}
\footnote{e-mails : gilles.pages@upmc.fr, clement.rey@upmc.fr This research benefited from the support of the "Chaire Risques Financiers''.}
\end{alphafootnotes}

\abstract{In this paper, we show that the abstract framework developed in \cite{Pages_Rey_2017} and inspired by \cite{Lamberton_Pages_2002} can be used to build invariant distributions for Brownian diffusion processes using the Milstein scheme and for diffusion processes with censored jump using the Euler scheme. Both studies rely on a weakly mean reverting setting for both cases. For the Milstein scheme we prove the convergence for test functions with polynomial (Wasserstein convergence) and exponential growth. For the Euler scheme of diffusion processes with censored jump we prove the convergence for test functions with polynomial growth.\\}


\section{Introduction}

In this paper, we apply a method developed in \cite{Pages_Rey_2017} for the recursive computation of the invariant distribution (denoted $\nu$) of a Brownian diffusion process $(X_t)_{t \geqslant 0}$. The main idea of this approach is to consider a non-homogeneous discrete Markov process which can be simulated using a family of transitions kernels $(Q_{\gamma})_{\gamma>0}$ and approximating $(X_t)_{t \geqslant 0}$. This paper aims to show that this method can be used in two non trivial situations under a weakly mean reverting setting. First, it can used for the computation of invariant distributions for Brownian diffusion processes using the Milstein
scheme. Then, it can be used for the computation of invariant distributions for diffusion processes with censored jump using the Euler scheme 

As suggested by the pointwise Birkhoff ergodic theorem, \cite{Pages_Rey_2017} shows that some sequence $( \nu_n)_{n \in \mathbb{N}^{\ast}}$ of random empirical measures $a.s.$ weakly converges toward $\nu$ under some appropriate mean-reverting and moment assumptions. An abstract framework is developed in \cite{Pages_Rey_2017} which can be used, among others, to obtain convergence of $\mbox{L}^p$-Wasserstein distance. Notice that for a given $f$, $\nu_n(f)$ can be recursively defined making its computation straightforward.\\

Invariant measures are crucial in the study of the long term behavior of stochastic differential systems. We invite the reader to refer to \cite{Hasminskii_1980} and \cite{Ethier_Kurtz_1986} for an overview of the subject. The computation of invariant distributions for stochastic systems has already been widely explored in the literature. In \cite{Soize_1994}, explicit exact expressions of the invariant density distribution for some solutions of Stochastic Differential Equations are given.\\

 However, in many cases there is no explicit formula for $\nu$. A first approach consists in studying the convergence, as $t$ tends to infinity, of the semigroup $(P_t)_{t \geqslant 0}$ of the Markov process $(X_t)_ {t\geqslant 0}$ with infinitesimal generator $A$ towards the invariant measure $\nu$. This is done $e.g.$ in \cite{Ganidis_Roynette_Simonot_1999} for the total variation topology which is thus adapted when the simulation of $P_T$ is possible for $T$ large enough. \\
 
 Whenever $(X_t)_{t \geqslant 0}$ can be simulated, we can use a Monte Carlo method to estimate $(P_t)_{t \geqslant 0}$, $i.e.$ $\mathbb{E}[f(X_t)]$, $t\geqslant 0$, producing a second term in the error analysis. When $(X_t)_{t \geqslant 0}$ cannot be simulated at a reasonable cost, a solution consists in simulating an approximation of $(X_t)_{t \geqslant 0}$, using a numerical scheme $(\overline{X}^{\gamma}_{\Gamma_n})_{n\in \mathbb{N}}$ built with transition functions $( \mathscr{Q}_{\gamma_n})_{n \in\mathbb{N}^{\ast}}$ (given a step sequence $(\gamma_n)_{n \in \mathbb{N}}$, $\Gamma_0=0$ and $\Gamma_n=\gamma_1+..+\gamma_n$). If the process $(\overline{X}^{\gamma}_{\Gamma_n})_{n\in \mathbb{N}}$ weakly converges towards $(X_t)_{t \geqslant 0}$, a natural construction relies on numerical homogeneous schemes ($(\gamma_n)_{n \in \mathbb{N}}$ is constant, $\gamma_n=\gamma_1>0$, for every $n \in \mathbb{N}^{\ast}$). This approach induces two more terms to control in the approximation of $\nu$ in addition to the error between $P_T$ and $\nu$ for a large enough fixed $T>0$, such that there exists $n(T) \in\mathbb{N}^{\ast}$,with $T=n(T) \gamma_1$: The first one is due to the weak approximation of $\mathbb{E}[f(X_T)]$ by $\mathbb{E}[f(\overline{X}^{\gamma_1}_T)]$ and the second one is due to the Monte Carlo error resulting from the computation of $\mathbb{E}[f(\overline{X}^{\gamma_1}_T]$.\\


Such an approach does not take advantage of the ergodic feature of $(X_t)_{t \geqslant 0}$. In fact, as investigated in \cite{Talay_1990} for Brownian diffusions, the ergodic (or positive recurrence) property of $(X_t)_{t \geqslant 0}$ is also satisfied by its approximation $(\overline{X}^{\gamma}_{\Gamma_n})_{n\in \mathbb{N}}$ at least for small enough time step $\gamma_n= \gamma_1, n \in \mathbb{N}^{\ast}$. Then $(\overline{X}^{\gamma}_{\Gamma_n})_{n\in \mathbb{N}}$ has an invariant distribution $\nu^{\gamma_1}$ (supposed to be unique for simplicity) and the empirical measures 
 \begin{align*}
 \nu^{\gamma_1}_n(dx)=\frac{1}{ \Gamma_n} \sum_{k=1}^n \gamma_k \delta_{\overline{X}^{\gamma_1}_{\Gamma_{k-1}}}(dx), \qquad \Gamma_n = n \gamma_1.
 \end{align*}
almost surely weakly converges to $\nu^{\gamma_1}$. Using this last result makes it is possible to compute by simulation arbitrarily accurate approximations of $\nu^{\gamma_1}(f)$ using only one simulated path of $(\overline{X}^{\gamma}_{\Gamma_n})_{n\in \mathbb{N}}$. It is an ergodic - or Langevin - simulation of $\nu^{\gamma_1}(f)$. However, it remains to establish at least that $\nu^{\gamma_1}(f)$ converges to $\nu(f)$ when $\gamma_1$ converges to zero and, if possible, at which rate.

 Another approach was proposed in \cite{Basak_Hu_Wei_1997}, still for Brownian diffusions, which avoids the asymptotic analysis between $\nu^{\gamma_1}$ and $\nu$. The authors directly prove that the discrete time Markov process $(\overline{X}^{\gamma}_{\Gamma_n})_{n\in \mathbb{N}}$, with step sequence $\gamma=(\gamma_n)_{n \in \mathbb{N}}$ vanishing to 0, weakly converges toward $\nu$. Therefore, the resulting error is made of two terms. The first one is due to this weak convergence and the second one to the Monte Carlo error involved in the computation of the law of $\overline{X}^{\gamma}_{\Gamma_n}$, for $n$ large enough. We also refer to \cite{Durmus_Moulines_2015} for the study of the total variation convergence for the Euler decreasing step of the over-damped Langevin diffusion. The reader may notice that in both approaches, strong ergodicity assumptions are required for the process with infinitesimal generator A. \\

In \cite{Lamberton_Pages_2002}, theses two ideas are combined to design a Langevin Euler Monte Carlo recursive algorithm with decreasing step which $a.s.$ weakly converges to the right target $\nu$. This paper treats the case where $ (\overline{X}^{\gamma}_{\Gamma_n})_{n\in \mathbb{N}}$ is a (inhomogeneous) Euler scheme with decreasing step associated to a strongly mean reverting Brownian diffusion process. The sequence $(\nu^{\gamma}_n)_{n \in \mathbb{N}^{\ast}}$ is defined as the weighted empirical measures of the path of $ (\overline{X}^{\gamma}_{\Gamma_n})_{n\in \mathbb{N}}$ (which is the procedure that is used in every work we mention from now on and which is also the one we use in this paper). In particular, the $a.s.$ weak convergence of
 \begin{align}
 \label{eq:def_weight_emp_meas_intro}
 \nu^{\gamma}_n(dx)=\frac{1}{\Gamma_n} \sum_{k=1}^n \gamma_k \delta_{\overline{X}^{\gamma}_{\Gamma_{k-1}}}(dx), \qquad \Gamma_n=\sum\limits_{k=1}^n \gamma_k,
 \end{align}
toward the (non-empty) set $\mathcal{V}$ of the invariant distributions of the underlying Brownian diffusion is established.
 Moreover, when the invariant measure $\nu$ is unique, it is proved that $\lim\limits_{n \to \infty} \nu^{\gamma}_n f=\nu f \; a.s.$ for a larger class of test functions which than $\mathcal{C}^0$ which contains $\nu-a.s.$ continuous functions with polynomial growth $i.e.$ convergence for the Wasserstein distance. In the spirit of \cite{Bhattacharya_1982} for the empirical measure of the underlying diffusion, they also obtained rates and limit gaussian laws for the convergence of $(\nu^{\gamma}_n(f))_{n \in \mathbb{N}^{\ast}}$ for test functions $f$ which can be written $f= A \varphi$. Note that, this approach does not require that the invariant measure $\nu$ is unique by contrast with the results obtained in \cite{Talay_1990}, \cite{Basak_Hu_Wei_1997} or \cite{Durmus_Moulines_2015} for instance. In this case, it is established that $a.s.$, every weak limiting distribution of $(\nu^{\gamma}_n)_{n \in \mathbb{N}^{\ast}}$ is an invariant distribution for the Brownian diffusion. \\
 This first paper gave rise to many generalizations and extensions. In \cite{Lamberton_Pages_2003}, the initial result is extended to the case of Euler scheme of Brownian diffusions with weakly mean reverting properties. Thereafter, in \cite{Lemaire_thesis_2005}, the class of test functions for which we have $\lim\limits_{n \to \infty} \nu^{\gamma}_n f =\nu f \; a.s.$ (when the invariant distribution is unique) is extended to include functions with exponential growth. Finally, in \cite{Panloup_2008}, the results concerning the polynomial case are shown to hold for the computation of invariant measures for weakly mean reverting Levy driven diffusion processes, still using the algorithm from \cite{Lamberton_Pages_2002}. This extension encourages relevant perspectives concerning not only the approximation of mean reverting Brownian diffusion stationary regimes but also to treat a larger class of processes. For a more complete overview of the studies concerning (\ref{eq:def_weight_emp_meas_intro}) for the Euler scheme, the reader can also refer to \cite{Pages_2001_ergo}, \cite{Lemaire_2007}, \cite{Panloup_2008_rate}, \cite{Pages_Panloup_2009}, \cite{Pages_Panloup_2012} or \cite{Mei_Yin_2015}.\\

Those results are extended in \cite{Pages_Rey_2017} and generalized to the case where $( \mathscr{Q}_{\gamma})_{\gamma >0}$ is not specified explicitly, to approximate invariant, not necessarily unique, distributions for general Feller processes. In \cite{Pages_Rey_2017}, an abstract framework, that can be used to prove every mentioned existing result, is developed which suggests various applications beyond the Euler scheme of Levy processes.\\

 This is the direction of this paper where we study the particular case of the Milstein scheme for Brownian diffusion processes and the Euler scheme of a diffusion process with censored jump (which is an extension of Levy processes).\\
  In particular, when $(\overline{X}^{\gamma}_{\Gamma_n})_{n\in \mathbb{N}}$ is the Milstein scheme of a Brownian diffusion process (respectively the Euler scheme of a diffusion process with censored jump), we establish the $a.s$ weak convergence of $(\nu^{\gamma}_n)_{n \in \mathbb{N}^{\ast}}$. Moreover, when the invariant distribution $\nu$ is unique we obtain $\lim\limits_{n \to \infty} \nu^{\gamma}_n f =\nu f \; a.s.$ when $f$ has polynomial growth in a first step (Wasserstein convergence) and when $f$ has exponential growth in a second step (resp. when $f$ has polynomial growth). Notice that for the Wasserstein convergence using the Milstein scheme, the simulation of Levy area is not necessary and then the approach we develop can also be applied to the antithetic Milstein scheme presented in \cite{Giles_Szpruch_2014}. \\
  Concerning the Euler scheme of a diffusion process with censored jump we establish the convergence of the empirical measures for the Wasserstein distance.\\

We begin by recalling the abstract results from \cite{Pages_Rey_2017} and then we focus on the specific applications.

%
%

\section{Convergence to invariant distributions - A general approach}
\label{section:convergence_inv_distrib_gnl}

In this section, we show that the empirical measures defined in the same way as in (\ref{eq:def_weight_emp_meas_intro}) and built from an approximation $(\overline{X}^{\gamma}_{\Gamma_n})_{n\in \mathbb{N}}$ of a Feller process $(X_t)_{t \geqslant 0}$ (which are not explicitly specified), where the step sequence $(\gamma_n)_{n \in \mathbb{N}^{\ast}} \underset{n \to \infty}{\to}0$, $a.s.$ weakly converges the set $\mathcal{V}$, of the invariant distributions of $(X_t)_{t \geqslant 0}$. This results are proved in \cite{Pages_Rey_2017} for generic approximation $(\overline{X}^{\gamma}_{\Gamma_n})_{n\in \mathbb{N}}$ and Feller processes $(X_t)_{t \geqslant 0}$. We will then apply those results to the case of the Milstein scheme of stochastic Brownian diffusion and also to the case of the Euler scheme of a diffusion process with censored jump.

To this end, we will provide as weak as possible mean reverting assumptions on the pseudo generator of $(\overline{X}^{\gamma}_{\Gamma_n})_{n\in \mathbb{N}}$ on the one hand and appropriate rate conditions on the step sequence $(\gamma_n)_{n \in \mathbb{N}^{\ast}}$ on the other hand.

\subsection{Presentation of the abstract framework}

\subsubsection{Notations}
Let $(E,\vert . \vert)$ be a locally compact separable metric space, we denote $\mathcal{C}(E)$ the set of continuous functions on $E$ and $\mathcal{C}_0(E)$ the set of continuous functions that vanish a infinity. We equip this space with the sup norm $\Vert f \Vert_{\infty}=\sup_{x \in E} \vert f(x) \vert$ so that $(\mathcal{C}_0(E),\Vert . \Vert_{\infty})$ is a Banach space. We will denote $\mathcal{B}(E)$ the $\sigma$-algebra of Borel subsets of $E$ and $\mathcal{P}(E)$ the family of Borel probability measures on $E$. We will denote by $\mathcal{K}_E$ the set of compact subsets of $E$.\\
Finally, for every Borel function $f:E \to \mathbb{R}$, and every $l_{\infty} \in \mathbb{R} \cup \{-\infty,+\infty\}$, $\lim\limits_{x\to \infty}f(x)= l_{\infty}$ if and only if for every $\epsilon >0$, there exists a compact $K_{\epsilon} \subset \mathcal{K}_E$ such that $\sup_{x \in K_{\epsilon}^c} \vert f(x)- l_{\infty} \vert < \epsilon$ if $l_{\infty} \in \mathbb{R} $, $\inf_{x \in K_{\epsilon}^c}  f(x)  > 1/\epsilon$ if $l_{\infty} =+\infty$, and $\sup\limits_{x \in K_{\epsilon}^c}  f(x)  < -1/\epsilon$ if $l_{\infty} =-\infty$ with $K_{\epsilon}^c=E \setminus K_{\epsilon}.$  \\

%
%


\subsubsection{Construction of the random measures}
Let $(\Omega,\mathcal{G}, \mathbb{P})$ be a probability space. We consider a Feller process $(X_t)_{t \geqslant 0}$ (see \cite{Feller_1952} for details) on $(\Omega,\mathcal{G}, \mathbb{P})$ taking values in a locally compact and separable metric space $E$. We denote by $(P_t)_{t \geqslant 0}$ the Feller semigroup (see \cite{Pazy_1992}) of this process. We recall that $(P_t)_{t \geqslant 0}$  is a family of linear operators from $\mathcal{C}_0(E)$ to itself such that $P_0 f=f$, $P_{t+s}f=P_tP_sf$, $t,s \geqslant 0$ (semigroup property) and $\lim\limits_{t \to 0} \Vert P_tf-f \Vert_{\infty}=0$ (Feller property). Using this semigroup, we can introduce the infinitesimal generator of $(X_t)_{t \geqslant 0}$ as a linear operator $A$ defined on a subspace $\DomA$ of $\mathcal{C}_0(E)$, satisfying: For every $f \in \DomA$,
\begin{align*}
Af= \lim\limits_{t \to 0} \frac{P_tf-f}{t}
\end{align*}
exists for the $\Vert . \Vert_{\infty}$-norm. The operator $A: \DomA \to \mathcal{C}_0(E)$ is thus well defined and $\DomA$ is called the domain of $A$. From the Echeverria Weiss theorem (see \cite{Ethier_Kurtz_1986} Theorem 9.17), the set of invariant distributions for $(X_t)_{t \geqslant 0}$ can be characterized in the following way: 
\begin{align*}
\mathcal{V}=\{ \nu \in \mathcal{P}(E), \forall t \geqslant 0, P_t \nu=  \nu \}=\{ \nu \in \mathcal{P}(E), \forall f \in \DomA, \nu(Af)=0 \}.
\end{align*}
The starting point of our reasoning is thus to consider an approximation of $A$. First, we introduce the family of transition kernels $(\mathscr{Q}_{\gamma})_{\gamma >0}$ from $\mathcal{C}_0(E)$ to itself. Now, let us define the family of linear operators $\widetilde{A} : = (\widetilde{A}_{\gamma} )_{ \gamma >0}$ from $\mathcal{C}_0(E)$ into itself, as follows
 \begin{equation}
 \label{eq:def_A_tilde}
 \forall f \in \mathcal{C}_0(E), \quad \gamma>0,  \qquad \widetilde{A}_{\gamma}f=\frac{\mathscr{Q}_{\gamma}f -f}{\gamma}.
 \end{equation}
The family $\widetilde{A}$ is usually called the pseudo-generator of the transition kernels $(\mathscr{Q}_{\gamma})_{\gamma>0}$ and is an approximation of $A$ as $\gamma$ tends to zero. From a practical viewpoint, the main interest of our approach is that we can consider that there exists $\overline{\gamma}>0$ such that for every $x \in E$ and every $\gamma \in [0, \overline{\gamma}]$, $\mathscr{Q}_{\gamma} (x,dy)$ is simulable at a reasonable computational cost. We use the family $(\mathscr{Q}_{\gamma})_{\gamma>0}$, to build $(\overline{X}_{\Gamma_n})_{n \in \mathbb{N}}$ (this notation replaces $(\overline{X}^{\gamma}_{\Gamma_n})_{n \in \mathbb{N}}$ from now for clarity in the writing) as the non-homogeneous Markov approximation of the Feller process $(X_t)_{t \geqslant 0}$. It is defined on the time grid $\{ \Gamma_n=\sum\limits_{k=1}^n \gamma_k, n \in \mathbb{N} \} $ with the sequence $\gamma:=(\gamma_n)_{n\in \mathbb{N}^{\ast} }$ of time step satisfying
\begin{align*}
\forall n \in \mathbb{N}^{\ast}, \quad 0 < \gamma_n  \leqslant \overline{\gamma}:= \sup_{n \in \mathbb{N}^{\ast}} \gamma_n< + \infty, \quad \lim\limits_{n \to + \infty} \gamma_n = 0 \quad \mbox{ and } \quad \lim\limits_{n \to + \infty}\Gamma_n=+ \infty.
\end{align*}
Its transition probability distributions are given by $\mathscr{Q}_{\gamma_n} (x,dy),n\in \mathbb{N}^{\ast}$, $x\in E$, $i.e. :$
\begin{align*}
 \mathbb{P}(\overline{X}_{\Gamma_{n+1}} \in dy \vert \overline{X}_{\Gamma_n})= \mathscr{Q}_{\gamma_{n+1}}(\overline{X}_{\Gamma_n},dy), \quad n \in \mathbb{N}.
\end{align*}
%
%
We can canonically extend $(\overline{X}_{\Gamma_n})_{n \in \mathbb{N}}$ into a \textit{càdlàg} process by setting $\overline{X}(t,\omega) =\overline{X}_{\Gamma_{n(t)}}(\omega)$ with $n(t)= \inf \{n \in \mathbb{N}, \Gamma_{n+1}>t \}$. Then $(\overline{X}_{\Gamma_n})_{n \in \mathbb{N}}$ is a simulable (as soon as $\overline{X}_0$ is) non-homogeneous Markov chain with transitions 
\begin{align*}
\forall m  \leqslant n, \qquad \overline{P}_{\Gamma_m,\Gamma_n}(x,dy)= \mathscr{Q}_{\gamma_{m+1}} \circ \cdots \circ \mathscr{Q}_{\gamma_n}(x,dy),
\end{align*}
and law
\begin{align*}
\mathcal{L}(\overline{X}_{\Gamma_n} \vert \overline{X}_{0}=x)=\overline{P}_{\Gamma_n}(x,dy)= \mathscr{Q}_{\gamma_1} \circ \cdots \circ \mathscr{Q}_{\gamma_n}(x,dy).
\end{align*}
We use $(\overline{X}_{\Gamma_n})_{n \in \mathbb{N}}$ to design a Langevin Monte Carlo algorithm. Notice that this approach is generic since the approximation transition kernels $(\mathscr{Q}_{\gamma})_{\gamma>0}$ are not explicitly specified and then, it can be used in many different configurations including among others, weak numerical schemes or exact simulation $i.e.$ $(\overline{X}_{\Gamma_n})_{n \in \mathbb{N}}=(X_{\Gamma_n})_{n \in \mathbb{N}}$. In particular, using high weak order schemes for $(X_t)_{t \geqslant 0}$ may lead to higher rates of convergence for the empirical measures. The approach we use to build the empirical measures is quite more general than in (\ref{eq:def_weight_emp_meas_intro}) as we consider some general weights which are not necessarily equal to the time steps. We define this weight sequence. Let $\eta:=(\eta_n)_{n \in \mathbb{N}^{\ast}}$ be such that
\begin{equation*}
\forall n \in \mathbb{N}^{\ast}, \quad \eta_n \geqslant 0, \quad \lim\limits_{n \to + \infty} H_n=+ \infty, \qquad \mbox{with} \qquad H_n= \sum\limits_{k=1}^n \eta_k.
\end{equation*}
Now we present our algorithm introduced in \cite{Pages_Rey_2017} and adapted from the one introduced in \cite{Lamberton_Pages_2002} designed with a Euler scheme with decreasing step $(\overline{X}_{\Gamma_n})_{n \in \mathbb{N}}$ of a Brownian diffusion process $(X_t)_{t \geqslant 0}$. For $x \in E$, let $\delta_x$ denote the Dirac mass at point $x$. For every $n \in \mathbb{N}^{\ast}$, we define the random weighted empirical random measures as follows
 \begin{equation}
 \label{eq:def_weight_emp_meas}
 \nu^{\eta}_n(dx)=\frac{1}{H_n} \sum_{k=1}^n \eta_k \delta_{\overline{X}_{\Gamma_{k-1}}}(dx).
 \end{equation}

This paper is dedicated to show that, when $(\overline{X}_{\Gamma_n})_{n \in \mathbb{N}}$ is the Milstein scheme of a Brownian diffusion process (respectively the Euler scheme of a censored jump diffusion) $(X_t)_{t \geqslant 0}$, then $a.s.$ every weak limiting distribution of $(\nu^{\eta}_n)_{n \in \mathbb{N}^{\ast}}$ belongs to $\mathcal{V}$. In particular when the invariant measure of $(X_t)_{t \geqslant 0}$ is unique, $i.e. \; \mathcal{V}=\{\nu\}$, we show that $\lim\limits_{n \to \infty} \nu^{\eta}_n f =\nu f  \; \mathbb{P}-a.s.$, for a generic class of continuous test functions $f$. The approach developed in \cite{Pages_Rey_2017} consists in two steps. First, we give a tightness property to obtain existence of a weak limiting distribution for $(\nu^{\eta}_n )_{n \in \mathbb{N}^{\ast}}$. Then, in a second step, we identify this limiting distribution with an invariant distribution of the Feller process $(X_t)_{t \geqslant 0}$. 

\subsubsection{Assumptions on the random measures}

In this part, we present the necessary assumptions on the pseudo-generator $\widetilde{A}  = (\widetilde{A}_{\gamma} )_{ \gamma >0}$ in order to prove the convergence of the empirical measures $(\nu^{\eta}_n)_{n \in \mathbb{N}^{\ast}}$.
\paragraph{Mean reverting recursive control \\}
 In our framework, we introduce a well suited assumption, referred to as the mean reverting recursive control of the pseudo-generator $\widetilde{A}$, that leads to a tightness property on $(\nu^{\eta}_n)_{n \in \mathbb{N}^{\ast}}$ from which follows the existence (in weak sense) of a limiting distribution for $(\nu^{\eta}_n)_{n \in \mathbb{N}^{\ast}}$. A supplementary interest of our approach is that it is designed to obtain the $a.s.$ convergence of $(\nu^{\eta}_n(f))_{n \in \mathbb{N}^{\ast}}$ for a generic class of continuous test functions $f$ which is larger then $\mathcal{C}_b(E)$. To do so, we introduce a Lyapunov function $V$ related to $(\overline{X}_{\Gamma_n})_{n \in \mathbb{N}}$. Assume that $V$ a Borel function such that
\begin{equation}
\label{hyp:Lyapunov}
\mbox{L}_{V}  \quad \equiv \qquad   V :(E \to [v_{\ast},+\infty), v_{\ast}> 0 \quad  \mbox{ and } \quad \lim\limits_{ x  \to \infty} V(x)=+ \infty. \\
\end{equation}
We now relate $V$ to $(\overline{X}_{\Gamma_n})_{n \in \mathbb{N}}$ introducing its mean reversion Lyapunov property. Let $\psi, \phi : [v_{\ast},\infty) \to (0,+\infty) $ some Borel functions such that $\widetilde{A}_{\gamma}\psi \circ V$ exists for every $\gamma \in (0, \overline{\gamma}]$. Let $\alpha>0$ and $\beta \in \mathbb{R}$. We assume  
%
%
 \begin{eqnarray}
\label{hyp:incr_sg_Lyapunov}
&\mathcal{RC}_{Q,V} (\psi,\phi,\alpha,\beta) \quad \equiv  \nonumber\\
& \quad  \left\{
    \begin{array}{l}
  (i) \; \quad \exists n_0 \in \mathbb{N}^{\ast},   \forall n   \geqslant n_0, x \in E,  \quad\widetilde{A}_{\gamma_n}\psi \circ V(x)\leqslant  \frac{\psi \circ V(x)}{V(x)}(\beta - \alpha \phi \circ V(x)). \\
  (ii) \quad     \liminf\limits_{y \to + \infty} \phi(y)> \beta / \alpha .
    \end{array}
\right.
\end{eqnarray}
$\mathcal{RC}_{Q,V} (\psi,\phi,\alpha,\beta)$ is called the weakly mean reverting recursive control assumption of the pseudo generator for Lyapunov function $V$. \\

Lyapunov functions are usually used to show the existence and sometimes the uniqueness of the invariant measure of Feller processes. In particular, when $p=1$, the condition $\mathcal{RC}_{Q,V}(I_d,I_d,\alpha,\beta) (i)$ appears as the discrete version of $AV \leqslant \beta-\alpha V$, which is used in that interest for instance in \cite{Hasminskii_1980}, \cite{Ethier_Kurtz_1986}, \cite{Basak_Hu_Wei_1997} or\cite{Pages_2001_ergo}. \\

 The condition $\mathcal{RC}_{Q,V}(V^p,I_d,\alpha,\beta) (i)$, $p \geqslant 1$, is studied in the seminal paper \cite{Lamberton_Pages_2002} (and then in \cite{Lamberton_Pages_2003} with $\phi(y)=y^a,a\in (0,1]$,$y \in [v_{\ast},\infty)$) concerning the Wasserstein convergence of the weighted empirical measures of the Euler scheme with decreasing step of a Brownian diffusions. When $\phi=I_d$, the Euler scheme is also studied for markov switching Brownian diffusions in \cite{Mei_Yin_2015}. Notice also that $\mathcal{RC}_{Q,V}(I_d,\phi,\alpha,\beta) (i)$ with $\phi$ concave appears in \cite{DFMS_2004} to prove sub-geometrical ergodicity of Markov chains. In \cite{Lemaire_thesis_2005}, a similar hypothesis to $\mathcal{RC}_{Q,V}(I_d,\phi,\alpha,\beta) (i)$ (with $\phi$ not necessarily concave and $\widetilde{A}_{\gamma_n}$ replaced by $A$), is also used  to study the Wasserstein but also exponential convergence of the weighted empirical measures (\ref{eq:def_weight_emp_meas}) for the Euler scheme of a Brownian diffusions. Finally in \cite{Panloup_2008} similar properties as $\mathcal{RC}_{Q,V}(V^p,V^a,\alpha,\beta) (i)$, $a\in (0,1]$, $p >0$, are developped in the study of the Euler scheme for Levy processes.\\
 
On the one hand, the function $\phi$ controls the mean reverting property. In particular, we call strongly mean reverting property when $\phi=I_d$ and weakly mean reverting property when $\lim\limits_{y \to +\infty} \phi(y)/y=0$, for instance $\phi(y)=y^a$, $a \in (0,1)$ for every $y \in [v_{\ast},\infty)$. On the other hand, the function $\psi$ is closely related to the identification of the set of test functions $f$ for which we have $\lim\limits_{n \to +\infty} \nu^{\eta}_n(f)=\nu(f) \; a.s.$, when $\nu$ is the unique invariant distribution of the underlying Feller process.\\
 
  To this end, for $s \geqslant 1$, which is related to step weight assumption, we introduce the sets of test functions for which we will show the $a.s.$ convergence of the weighted empirical measures (\ref{eq:def_weight_emp_meas}):
\begin{align}
\label{def:espace_test_function_cv}
\mathcal{C}_{\tilde{V}_{\psi,\phi,s}}(E)=& \big\{ f \in \mathcal{C}(E), \vert f(x) \vert=\underset{  x \to \infty}{o}( \tilde{V}_{\psi,\phi,s} (x) ) \big\}, \\
&\mbox{with} \quad \tilde{V}_{\psi,\phi,s}:E \to \mathbb{R}_+, x \mapsto\tilde{V}_{\psi,\phi,s}(x): =\frac{\phi\circ V(x)\psi \circ V(x)^{1/s}}{V(x)}. \nonumber
\end{align}
Notice that our approach benefits from providing generic results because we consider general Feller processes and approximations but also because the functions $\phi$ and $\psi$ are not specified explicitly.
\paragraph{Infinitesimal generator approximation \\}
This section presents the assumption that enables to characterize the limiting distributions of the $a.s.$ tight  sequence $(\nu^{\eta}_n(dx, \omega))_{n \in \mathbb{N}^{\ast}}$.
 We aim to estimate the distance between $\mathcal{V}$ and $\nu^{\eta}_n$ (see (\ref{eq:def_weight_emp_meas})) for $n$ large enough. We thus introduce an hypothesis concerning the distance between $(\widetilde{A}_{\gamma} )_{ \gamma>0}$, the pseudo-generator of $(\mathscr{Q}_{\gamma} )_{ \gamma>0}$, and $A$, the infinitesimal generator of $(P_t)_{t \geqslant 0}$. We assume that there exists $\DomA_0 \subset \DomA$ with $\DomA_0 $ dense in $\mathcal{C}_0(E)$ such that:
%
%
 \begin{align}
\mathcal{E}(\widetilde{A},A,\DomA_0) \quad \equiv \qquad   \forall \gamma \in (0, \overline{\gamma}], \forall f \in \DomA_0, \forall x \in E, \quad  \vert \widetilde{A}_{\gamma} f(x) -Af(x)\vert \leqslant  \Lambda_f(x,\gamma),
 \label{hyp:erreur_tems_cours_fonction_test_reg}
\end{align}
where $\Lambda_{f}:E \times \mathbb{R}_+ \to \mathbb{R}_+$ can be represented in the following way: Let $(\tilde{\Omega},\tilde{\mathcal{G}},\tilde{\mathbb{P}})$ be a probability space. Let $g :E\to \mathbb{R}_+^{q}$, $q \in \mathbb{N}$, be a locally bounded Borel measurable function and let $\tilde{\Lambda}_{f}:(E\times \mathbb{R}_+ \times \tilde{\Omega}, \mathcal{B}(E) \otimes \mathcal{B}(\mathbb{R}_+) \otimes \tilde{\mathcal{G}}) \to \mathbb{R}_+^{q}$ be a measurable function such that  
\begin{align*}
\sup_{i \in \{1,\ldots,q\} } \tilde{\mathbb{E}}[ \sup_{x \in E} \sup_{\gamma \in (0,\overline{\gamma}] } \tilde{\Lambda}_{f,i}(x,\gamma, \tilde{\omega}) ]< + \infty
\end{align*}
and that we have the following representation
%
%
%
\begin{align*}
\forall x \in E , \forall \gamma \in (0,\overline{\gamma}], \qquad \Lambda_f(x,\gamma)= \langle g (  x ) ,\tilde{\mathbb{E}} [\tilde{\Lambda}_{f}(x,\gamma, \tilde{\omega})]  \rangle_{\mathbb{R}^q}
\end{align*}
%
%
%
%
%
%
%
Moreover, we assume that for every $i \in \{1,\ldots,q\}$, $\sup_{n \in \mathbb{N}^{\ast}} \nu_n^{\eta}( g_i ,\omega )< + \infty, \; \mathbb{P}(d\omega)-a.s.$, and that $\tilde{\Lambda}_{f,i}$ satisfies one of the following two properties:\\
There exists a measurable function $\underline{\gamma}:(\tilde{\Omega}, \tilde{\mathcal{G}}) \to((0, \overline{\gamma}],\mathcal{B}((0, \overline{\gamma}]) )$ such that:
\begin{enumerate}[label=\textbf{\Roman*)}]
\item \label{hyp:erreur_tems_cours_fonction_test_reg_Lambda_representation_1} 
\inlineequation[hyp:erreur_temps_cours_fonction_test_reg_Lambda_representation_2_1]{
 \tilde{\mathbb{P}}(d\tilde{\omega})-a.s \qquad \left\{
    \begin{array}{l}
  (i)   \quad \; \;  \forall K \in \mathcal{K}_E ,   \quad  \lim\limits_{\gamma \to 0} \sup\limits_{x \in K} \tilde{\Lambda}_{f,i}(x, \gamma,\tilde{\omega})=0, \\
  (ii) \quad     \lim\limits_{x \to \infty}  \sup\limits_{\gamma \in (0,\underline{\gamma}(\tilde{\omega}) ]} \tilde{\Lambda}_{f,i}(x, \gamma,\tilde{\omega})=0,  \qquad \qquad \qquad \qquad \qquad \qquad \qquad \qquad \qquad
    \end{array}
\right.
}
\item \label{hyp:erreur_temps_cours_fonction_test_reg_Lambda_representation_2}\inlineequation[hyp:erreur_temps_cours_fonction_test_reg_Lambda_representation_2_2]{
 \tilde{\mathbb{P}}(d\tilde{\omega})-a.s \qquad  \lim\limits_{\gamma \to 0} \sup\limits_{x \in E} \tilde{\Lambda}_{f,i}(x, \gamma,\tilde{\omega}) g_i(x) =0  . \qquad \qquad \qquad \qquad \qquad \qquad \qquad \qquad \qquad \qquad \qquad \; \;}
\end{enumerate}
\begin{remark}
\label{rmk:representation_mesure_infinie}
Let $(F,\mathcal{F},\lambda)$ be a measurable space. Using the exact same approach, the results we obtain hold when we replace the probability space $(\tilde{\Omega},\tilde{\mathcal{G}},\tilde{\mathbb{P}})$ by the product measurable space $(\tilde{\Omega} \times F,\tilde{\mathcal{G}} \otimes \mathcal{F},\tilde{\mathbb{P}}\otimes \lambda)$ in the representation of $\Lambda_f$ and in (\ref{hyp:erreur_temps_cours_fonction_test_reg_Lambda_representation_2_1}) and (\ref{hyp:erreur_temps_cours_fonction_test_reg_Lambda_representation_2_2}) but we restrict to that case for sake of clarity in the writing. This observation can be useful when we study jump process where $\lambda$ can stand for the jump intensity.
\end{remark}
This representation assumption benefits from the fact that the transition functions $(\mathscr{Q}_{\gamma} (x,dy))_{\gamma \in (0, \overline{\gamma}]}$, $x \in E$, can be represented using distributions of random variables which are involved in the computation of $(\overline{X}_{\Gamma_n})_{n \in \mathbb{N}^{\ast}}$. In particular, this approach is well adapted to stochastic approximations associated to a time grid such as numerical schemes for stochastic differential equations with a Brownian part or/and a jump part. 
\paragraph{Growth control and Step Weight assumptions \\}
We conclude with hypothesis concerning the control of the martingale part of one step of our approximation. Let $\rho \in [1,2]$ and let $\epsilon_{\mathcal{I}} : \mathbb{R}_+ \to \mathbb{R}_+$ an increasing function. For $F \subset \{f,f:(E, \mathcal{B}(E)) \to (\mathbb{R}, \mathcal{B}(\mathbb{R}) ) \}$ and $g:E \to \mathbb{R}_+$ a Borel function, we assume that, for every $n \in \mathbb{N}$,
 \begin{align}
\label{hyp:incr_X_Lyapunov}
\mathcal{GC}_{Q}(F,g,\rho,\epsilon_{\mathcal{I}}) \; \equiv \quad  \mathbb{P}-a.s.& \quad  \forall f \in F, \\  \nonumber
&   \mathbb{E}[ \vert  f  ( \overline{X}_{\Gamma_{n+1}})- \mathscr{Q}_{\gamma_{n+1}}f(\overline{X}_{\Gamma_n}) \vert^{\rho}\vert \overline{X}_{\Gamma_n} ]  \leqslant   C_f \epsilon_{\mathcal{I}}(\gamma_{n+1})  g (\overline{X}_{\Gamma_n}) ,
\end{align}
with $C_f>0$ a finite constant which may depend on $f$. 
\begin{remark}\label{rmrk:Accroiss_mes} The reader may notice that $\mathcal{GC}_{Q}(F,g,\rho,\epsilon_{\mathcal{I}}) $ holds as soon as (\ref{hyp:incr_X_Lyapunov}) is satisfied with $\mathscr{Q}_{\gamma_{n+1}}f(\overline{X}_{\Gamma_n})$, $n  \in \mathbb{N}^{\ast} $, replaced by a $\mathcal{F}^{\overline{X}}_n:=\sigma(\overline{X}_{\Gamma_k},k \leqslant n)$- progressively  measurable process $(\mathfrak{X}_n)_{n \in \mathbb{N}^{\ast}}$ since we have $\mathscr{Q}_{\gamma_{n+1}}f(\overline{X}_{\Gamma_n}) =\mathbb{E}[f(\overline{X}_{\Gamma_{n+1}}) \vert \overline{X}_{\Gamma_n}]$ and $\mathbb{E}[ \vert  f  ( \overline{X}_{\Gamma_{n+1}})- \mathscr{Q}_{\gamma_{n+1}}f(\overline{X}_{\Gamma_n}) \vert^{\rho}\vert \overline{X}_{\Gamma_n} ]  \leqslant 2^{\rho} \mathbb{E}[ \vert  f  ( \overline{X}_{\Gamma_{n+1}})- \mathfrak{X}_n \vert^{\rho}\vert \overline{X}_{\Gamma_n} ]$ for every $\mathfrak{X}_n \in \LL^2(\mathcal{F}^{\overline{X}}_n)$.
\end{remark}

We will combine this assumption with the following step weight related ones:
\begin{equation}
 \label{hyp:step_weight_I_gen_chow}
\mathcal{S}\mathcal{W}_{\mathcal{I}, \gamma,\eta}(g, \rho , \epsilon_{\mathcal{I}}) \quad \equiv \qquad   \mathbb{P}-a.s. \quad   \sum_{n=1}^{\infty} \Big \vert \frac{\eta_n }{H_n \gamma_n } \Big \vert^{\rho} \epsilon_{\mathcal{I}}(\gamma_n)  g(\overline{X}_{\Gamma_n})  < + \infty,
 \end{equation}
and
 \begin{equation}
 \label{hyp:step_weight_I_gen_tens}
\mathcal{S}\mathcal{W}_{\mathcal{II},\gamma,\eta}(F) \quad \equiv \qquad \mathbb{P}-a.s. \quad \forall f \in F, \quad   \sum_{n=0}^{\infty} \frac{(\eta_{n+1} /\gamma_{n+1}-\eta_n /\gamma_n)_+ }{H_{n+1} }  \vert f(\overline{X}_{\Gamma_n}) \vert < + \infty,
 \end{equation}
with the convention $\eta_0/\gamma_0=1$. Notice that this last assumption holds as soon as the sequence $(\eta_n / \gamma_n)_{n \in \mathbb{N}^{\ast} }$ is non-increasing. \\

\noindent At this point we can focus now on the main results concerning this general approach.
\subsection{Convergence}

We give abstract results which are proved in \cite{Pages_Rey_2017}.

\subsubsection{Almost sure tightness}
From the recursive control assumption, the following Theorem establish the $a.s.$ tightness of the sequence $(\nu^{\eta}_n)_{n \in \mathbb{N}^{\ast}}$ and also provides a uniform control of $(\nu^{\eta}_n)_{n \in \mathbb{N}^{\ast}}$ on a generic class of test functions.

\begin{mytheo}
\label{th:tightness}
Let $s \geqslant 1$, $\rho \in[1,2]$, $v_{\ast}>0$, and let us consider the Borel functions $V :E \to [v_{\ast},\infty)$, $g:E \to \mathbb{R}_+$, $\psi : [v_{\ast},\infty) \to \mathbb{R}_+ $ and $\epsilon_{\mathcal{I}} : \mathbb{R}_+ \to \mathbb{R}_+$ an increasing function. We have the following properties:
\begin{enumerate}[label=\textbf{\Alph*.}]
\item\label{th:tightness_point_A}  Assume that $\widetilde{A}_{\gamma_n}(\psi \circ V)^{1/s}$ exists for every $n \in \mathbb{N}^{\ast}$, and that $\mathcal{GC}_{Q}((\psi \circ V)^{1/s},g ,\rho ,\epsilon_{\mathcal{I}}) $ (see (\ref{hyp:incr_X_Lyapunov})), $\mathcal{S}\mathcal{W}_{\mathcal{I}, \gamma,\eta}( g,\rho,\epsilon_{\mathcal{I}}) $ (see (\ref{hyp:step_weight_I_gen_chow})) and $\mathcal{S}\mathcal{W}_{\mathcal{II},\gamma,\eta}((\psi \circ V)^{1/s}) $ (see (\ref{hyp:step_weight_I_gen_tens}) hold. Then
\begin{equation}
\label{eq:invariance_mes_emp_Lyap_gen}
\mathbb{P} \mbox{-a.s.} \quad  \sup_{n\in \mathbb{N}^{\ast} } - \frac{1}{H_n} \sum_{k=1}^n \eta_k \widetilde{A}_{\gamma_k} (\psi \circ V)^{1/s}  (\overline{X}_{\Gamma_{k-1}})< + \infty.
\end{equation}
\item\label{th:tightness_point_B}
Let $\alpha>0$ and $\beta \in \mathbb{R}$. Let $\phi:[v_{\ast},\infty )\to \mathbb{R}_+^{\ast}$ be a continuous function such that $C_{\phi}:= \sup_{y \in [v_{\ast},\infty )}\phi(y)/y< \infty$. Assume that (\ref{eq:invariance_mes_emp_Lyap_gen}) holds and
\begin{enumerate}[label=\textbf{\roman*.}]
\item $\mathcal{RC}_{Q,V}(\psi,\phi,\alpha,\beta)$ (see (\ref{hyp:incr_sg_Lyapunov})) holds.
\item $\mbox{L}_{V}$ (see (\ref{hyp:Lyapunov})) holds and $\lim\limits_{y \to +\infty}  \frac{\phi(y) \psi (y)^{1/s}}{y}=+\infty$.
\end{enumerate}
Then,
 \begin{equation*}
\mathbb{P} \mbox{-a.s.} \quad \sup_{n \in \mathbb{N}^{\ast}} \nu_n^{\eta}( \tilde{V}_{\psi,\phi,s} ) < + \infty .
\end{equation*}
with $\tilde{V}_{\psi,\phi,s}$ defined in (\ref{def:espace_test_function_cv}). Therefore, the sequence $(\nu^{\eta}_n)_{n \in \mathbb{N}^{\ast}}$ is $\mathbb{P}-a.s.$ tight. 
 \end{enumerate}
\end{mytheo}

\subsubsection{Identification of the limit}
In Theorem \ref{th:tightness}, we obtained the tightness of $(\nu_n^{\eta})_{n \in \mathbb{N}^{\ast}}$. It remains to show that every limiting point of this sequence is an invariant distribution of the Feller process with infinitesimal generator $A$. This is the interest of the following Theorem which relies on the infinitesimal generator approximation.
\begin{mytheo}
\label{th:identification_limit}
Let $\rho \in [1,2]$. We have the following properties:
\begin{enumerate}[label=\textbf{\Alph*.}]
\item\label{th:identification_limit_A} 
Let $\DomA_0 \subset \DomA$, with  $\DomA_0$ dense in $\mathcal{C}_0(E)$. We assume that $\widetilde{A}_{\gamma_n}f$ exists for every $f \in \DomA_0$ and every $n \in \mathbb{N}^{\ast}$. Also assume that there exists $g:E \to \mathbb{R}_+$ a Borel function and $\epsilon_{\mathcal{I}} : \mathbb{R}_+ \to \mathbb{R}_+$ an increasing function such that $\mathcal{GC}_{Q}(\DomA_0,g,\rho ,\epsilon_{\mathcal{I}}) $ (see (\ref{hyp:incr_X_Lyapunov})) and $\mathcal{S}\mathcal{W}_{\mathcal{I}, \gamma,\eta}( g,\rho,\epsilon_{\mathcal{I}}) $ (see (\ref{hyp:step_weight_I_gen_chow})) hold and that
 \begin{equation}
 \label{hyp:accroiss_sw_series_2}  \lim\limits_{n \to + \infty} \frac{1}{H_n} \sum_{k =1}^{n} \vert \eta_{k+1}/\gamma_{k+1}-\eta_k /\gamma_k \vert = 0.
\end{equation} Then
\begin{equation}
\label{hyp:identification_limit}
\mathbb{P} \mbox{-a.s.} \quad \forall f \in \DomA_0, \qquad   \lim\limits_{n \to + \infty}  \frac{1}{H_n} \sum_{k=1}^n \eta_k \widetilde{A}_{\gamma_k}f (\overline{X}_{\Gamma_{k-1}})=0.
\end{equation}
\item \label{th:identification_limit_B} 
We assume that (\ref{hyp:identification_limit}) and $\mathcal{E}(\widetilde{A},A,\DomA_0) $ (see (\ref{hyp:erreur_tems_cours_fonction_test_reg})) hold. Then
\begin{align*}
\mathbb{P} \mbox{-a.s.} \quad \forall f \in \DomA_0, \qquad     \lim\limits_{n \to + \infty} \nu_n^{\eta}( Af )=0.
\end{align*}
It follows that, $\mathbb{P}-a.s.$, every weak limiting distribution $\nu^{\eta}_{\infty}$ of the sequence $(\nu_n^{\eta})_{n \in \mathbb{N}^{\ast}}$ belongs to $\mathcal{V}$, the set of the invariant distributions of $(X_t)_{t \geqslant 0}$. Finally, if the hypothesis from Theorem \ref{th:tightness} point \ref{th:tightness_point_B} hold and $(X_t)_{t \geqslant 0}$ has a unique invariant distribution, $i.e.$ $\mathcal{V}=\{\nu\}$, then 
\begin{align}
\label{eq:test_function_gen_cv}
\mathbb{P} \mbox{-a.s.} \quad \forall f \in \mathcal{C}_{\tilde{V}_{\psi,\phi,s}}(E), \quad \lim\limits_{n \to + \infty} \nu_n^{\eta}(f)=\nu(f),
\end{align}
 with $\mathcal{C}_{\tilde{V}_{\psi,\phi,s}}(E)$ defined in (\ref{def:espace_test_function_cv}).
\end{enumerate}
\end{mytheo}
In the particular case where the function $\psi$ is polynomial, (\ref{eq:test_function_gen_cv}) also reads as the $a.s.$ convergence of the empirical measures for some $\mbox{L}^p$-Wasserstein distances, $p>0$, that we will study further in this paper for some numerical schemes of some diffusion processes. From the liberty granted by the choice of $\psi$ in this abstract framework, where only a recursive control with mean reverting is required, we will also propose an application for functions $\psi$ with exponential growth.

\subsection{About Growth control and Step Weight assumptions}
We present other useful abstract results from \cite{Pages_Rey_2017}. The following Lemma presents a $\mbox{L}_1$-finiteness property that we can obtain under recursive control hypothesis and strongly mean reverting assumptions ($\phi=I_d$). This result is thus useful to prove $\mathcal{S}\mathcal{W}_{\mathcal{I}, \gamma,\eta}(g, \rho , \epsilon_{\mathcal{I}})$ (see (\ref{hyp:step_weight_I_gen_chow})) or $\mathcal{S}\mathcal{W}_{\mathcal{II},\gamma,\eta}(F)  $ (see (\ref{hyp:step_weight_I_gen_tens})) for well chosen $F$ and $g$ in this specific situation.
\begin{lemme}
\label{lemme:mom_psi_V}
Let $v_{\ast}>0$, $V:E \to [v_{\ast},\infty) $, $\psi:[v_{\ast},\infty) \to \mathbb{R}_+, $ such that $\widetilde{A}_{\gamma_n}\psi \circ V$ exists for every $n \in \mathbb{N}^{\ast}$. Let $\alpha>0$ and $\beta \in \mathbb{R}$. We assume that $\mathcal{RC}_{Q,V}(\psi,I_d,\alpha,\beta)$ (see (\ref{hyp:incr_sg_Lyapunov})) holds and that $\mathbb{E}[\psi\circ V (\overline{X}_{\Gamma_{n_0}})]< + \infty$ for every $n_0 \in \mathbb{N}^{\ast}$. Then
\begin{align}
\label{eq:mom_psi_V}
\sup_{n \in \mathbb{N}} \mathbb{E}[\psi \circ V(\overline{X}_{\Gamma_n})] < + \infty
\end{align}
In particular, let $\rho \in [1,2]$ and $\epsilon_{\mathcal{I}} : \mathbb{R}_+ \to \mathbb{R}_+$, an increasing function. It follows that if $\sum_{n=1}^{\infty} \Big \vert \frac{\eta_n }{H_n \gamma_n } \Big \vert^{\rho} \epsilon_{\mathcal{I}}(\gamma_n)   < + \infty$, then $\mathcal{S}\mathcal{W}_{\mathcal{I}, \gamma,\eta}(\psi \circ V, \rho , \epsilon_{\mathcal{I}}) $ holds and if $ \sum_{n=0}^{\infty} \frac{(\eta_{n+1} /\gamma_{n+1}-\eta_n /\gamma_n)_+ }{H_{n+1} } < + \infty$, then $\mathcal{S}\mathcal{W}_{\mathcal{II},\gamma,\eta}(\psi \circ V) $ is satisfied
\end{lemme}
Now, we provide a general way to obtain $\mathcal{S}\mathcal{W}_{\mathcal{I}, \gamma,\eta}(g,\rho, \epsilon_{\mathcal{I}})$ and $\mathcal{S}\mathcal{W}_{\mathcal{I}\mathcal{I},\gamma,\eta}(F)$  for some specific $g$ and $F$ as soon as a recursive control with weakly mean reversion assumption holds.
\begin{lemme}
\label{lemme:mom_V}
Let $v_{\ast}>0$, $V:E \to [v_{\ast},\infty) $, $\psi, \phi :[v_{\ast},\infty) \to \mathbb{R}_+, $ such that $\widetilde{A}_{\gamma_n}\psi \circ V$ exists for every $n \in \mathbb{N}^{\ast}$. Let $\alpha>0$ and $\beta \in \mathbb{R}$. We also introduce the non-increasing sequence $(\theta_n)_{n \in \mathbb{N}^{\ast}}$ such that $\sum_{n \geqslant 1} \theta_n \gamma_n < + \infty$. We assume that $\mathcal{RC}_{Q,V}(\psi,\phi,\alpha,\beta)$ (see (\ref{hyp:incr_sg_Lyapunov})) holds and that $\mathbb{E}[\psi\circ V (\overline{X}_{\Gamma_{n_0}})]< + \infty$ for every $n_0 \in \mathbb{N}^{\ast}$. Then
\begin{equation*}
\sum_{n=1}^{\infty} \theta_n \gamma_n \mathbb{E} [\tilde{V}_{\psi,\phi,1}(\overline{X}_{\Gamma_{n-1}}) ] < + \infty
\end{equation*}
 with $ \tilde{V}_{\psi,\phi,1}$ defined in (\ref{def:espace_test_function_cv}). In particular, let $\rho \in [1,2]$ and $\epsilon_{\mathcal{I}} : \mathbb{R}_+ \to \mathbb{R}_+$, an increasing function. If we also assume 
 \begin{equation}
 \label{hyp:step_weight_I}
\mathcal{S}\mathcal{W}_{\mathcal{I}, \gamma,\eta}(\rho, \epsilon_{\mathcal{I}})  \qquad \Big( \gamma_n^{-1} \epsilon_{\mathcal{I}}(\gamma_n) \big( \frac{\eta_n }{H_n \gamma_n } \big)^{\rho} \Big)_{n \in \mathbb{N}^{\ast}} \mbox{ is non-increasing and } \sum_{n=1}^{\infty} \Big( \frac{\eta_n }{H_n \gamma_n } \Big)^{\rho} \epsilon_{\mathcal{I}}(\gamma_n) < + \infty,
 \end{equation}
 then we have $\mathcal{S}\mathcal{W}_{\mathcal{I}, \gamma,\eta}(\tilde{V}_{\psi,\phi,1},\rho,\epsilon_{\mathcal{I}}) $ (see (\ref{hyp:step_weight_I_gen_chow})). Finally,if
  \begin{equation}
 \label{hyp:step_weight_II}
\mathcal{S}\mathcal{W}_{\mathcal{II},\gamma,\eta}  \qquad \Big( \frac{ \frac{\eta_{n+1} }{(\gamma_{n+1}}-\frac{\eta_n}{\gamma_n})_+ }{ \gamma_n H_n }  \Big)_{n \in \mathbb{N}^{\ast}} \mbox{ is non-increasing and } \sum_{n=1}^{\infty}  \frac{(\eta_{n+1} /\gamma_{n+1}-\eta_n /\gamma_n)_+ }{H_n }  < + \infty,
 \end{equation}
  then we have $\mathcal{S}\mathcal{W}_{\mathcal{II},\gamma,\eta}( \tilde{V}_{\psi,\phi,1}) $ (see (\ref{hyp:step_weight_I_gen_tens})).
\end{lemme}
This result concludes the general approach in a generic framework to prove convergence. The next part of this paper is dedicated to various applications.

\section{Applications}

\subsection{Notations}

First, for $\alpha \in (0,1]$ and $f$ a $\alpha$-H\"older function we denote $[f]_{\alpha}=\sup_{x \neq y}\vert f(y)-f(x) \vert / \vert y -x \vert^{\alpha}$.\\

 Now, let $d \in \mathbb{N}$. For any $ \mathbb{R}^{d \times d}$-valued symmetric matrix $S$, we define $\lambda_S:= \sup\{\lambda_{S,1},.. \lambda_{S,d},0 \}$, with $\lambda_{S,i}$ the $i$-th eigenvalue of $S$.\\
 
We also recall the Burkholder--Davies--Gundy (BDG) inequality for martingales. Let $p \geqslant 1/2$ and $(\widetilde{M_t})_{t \geqslant 0}$ a $\mathbb{R}^d$-valued martingale with respect to the filtration $\mathcal{F}^{\widetilde{M}}_t := \sigma(\widetilde{M}_t,s\leqslant t)$. 
Then, there exists $C_p \geqslant 0$ such that
\begin{align}
\label{eq:BDG_inegalite}
 \forall t \geqslant 0, \qquad \mathbb{E}[ \sup_{s \in [0,t]}\vert \widetilde{M_s} \vert^{2p}] \leqslant C_p \mathbb{E}[  \langle \widetilde{M}\rangle_t ^{p}].
\end{align}

%

\subsection{The Milstein scheme}

%
In this part, we treat the case of a Milstein scheme (introduced in \cite{Milstein_1987}) with decreasing steps for a Brownian diffusion process. We propose two approaches under weakly mean reverting assumptions. The first one relies on polynomial test functions and the second one relies on exponential test functions. More particularly we propose a setting with functions $\psi$ such that $\psi(y)=y^p$, $ p \geqslant 0$ for every $y \in [v_{\ast},\infty)$. The other setting is based on functions $\psi(y)=\exp( \lambda y^p)$, $ p \in [0,1/2]$, $\lambda \geqslant 0$, for every $y \in [v_{\ast},\infty)$.\\

\subsubsection{Presentation and main result}
 We consider a $d$-dimensional Brownian motion $(W_t)_{t \geqslant 0}$. We are interested in the solution of the $d$-dimensional stochastic equation
\begin{align*}
X_t= x+\int\limits_0^t b(X_{s})ds + \int\limits_0^t & \sigma(X_{s})  dW_s
\end{align*}
where $b: \mathbb{R}^d \to \mathbb{R}^d$ and $\sigma, \partial_{x_l}\sigma:\mathbb{R}^{d } \to \mathbb{R}^{d \times d}, l \in \{1, \ldots d \},$ are locally bounded functions. The infinitesimal generator of this process is given by
\begin{align*}
Af(x) =\langle b(x) , \nabla f(x) \rangle + \frac{1}{2}\sum_{i,j=1}^d (\sigma \sigma^{\ast})_{i,j}(x) \frac{\partial^2f}{\partial x_i \partial x_j}(x)  
\end{align*} 

and its domain $\DomA$ contains $\DomA_0 =\mathcal{C}^2_K(\mathbb{R}^d)$. Notice that $\DomA_0 $ is dense in $\mathcal{C}_0(E)$. Now, we introduce the Milstein scheme for $(X_t)_{t \geqslant 0}$ defined for every $n \in \mathbb{N}$ and $t \in [\Gamma_n, \Gamma_{n+1}]$, by
%

\begin{align*}
\overline{X}_{t}  =& \overline{X}_{\Gamma_n} + (t-\Gamma_{n}) b(\overline{X}_{\Gamma_{n}}) + \sigma(\overline{X}_{\Gamma_{n}})(W_t-W_{\Gamma_n}) + \sum_{i,j,l=1}^d \partial_{x_l} \sigma_i ( \overline{X}_{\Gamma_n} )  \sigma_{l,j}( \overline{X}_{\Gamma_n} )   \int\limits_{\Gamma_n}^{t} \int\limits_{\Gamma_n}^s dW^j_u dW^i_s
\end{align*}
with $\sigma_i:\mathbb{R}^d \to \mathbb{R}^d, x \mapsto \sigma_i(x)= (\sigma_{1,i}(x), \ldots , \sigma_{d,i}(x) )$. We introduce the notations:
\begin{align}
\label{def:incr_milstein}
\Delta \overline{X}^1_{n+1} = &  \gamma_{n+1}b(\overline{X}_{\Gamma_{n}})  , \quad \Delta \overline{X}^{3}_{n+1} =  \sum_{i,j=1}^d \sum_{l=1}^d \partial_{x_l} \sigma_i ( \overline{X}_{\Gamma_n} )  \sigma_{l,j}( \overline{X}_{\Gamma_n} )   \int\limits_{\Gamma_n}^{\Gamma_{n+1}} \int\limits_{\Gamma_n}^s dW^j_u dW^i_s \nonumber , \\
\Delta \overline{X}^{2}_{n+1} = &  \sigma (\overline{X}_{\Gamma_{n}}) (W_{\Gamma_{n+1}}-W_{\Gamma_n} ).
\end{align}
and $\overline{X}_{\Gamma_{n+1}}^i=\overline{X}_{\Gamma_n}+ \sum_{j=1}^i \Delta \overline{X}^i_{n+1}$. In the sequel we will use the notation $U_{n+1}=\gamma_{n+1}^{-1/2}(W_{\Gamma_{n+1}}-W_{\Gamma_n}) $ and $\mathcal{W}_{n+1}=(\mathcal{W}^{i,j}_{n+1})_{i,j \in \{1, \ldots, d \} }$ with $ \mathcal{W}^{i,j}_{n+1}=\gamma_{n+1}^{-1} \int\limits_{\Gamma_n}^{\Gamma_{n+1}} \int\limits_{\Gamma_n}^s dW^j_u dW^i_s $.

%
%
%

Now, we assume that the Lyapunov function $V: \mathbb{R}^d \to [v_{\ast}, \infty)$, $v_{\ast}>0$, satisfies $\mbox{L}_V$ (see (\ref{hyp:Lyapunov})) and is essentially quadratic:
\begin{align}
\label{hyp:Lyapunov_control_milstein}
\vert \nabla V \vert^2 \leqslant C_V V, \qquad \sup_{x \in \mathbb{R}^d} \vert D^2 V(x) \vert < + \infty
\end{align}
 We also define
\begin{align}
\label{def:lambda_psi_milstein}
\forall x \in \mathbb{R}^d , \quad \lambda_{\psi}(x):= \lambda_{D^2V(x)+2\nabla V(x)^{\otimes 2} \psi''\circ V(x) \psi'\circ V(x)^{-1}}  .
\end{align}
When $\psi(y)=\psi_p(y)=y^{p}$, we will also use the notation $\lambda_p$ instead of $\lambda_{\psi}$. Now, let $\phi:[v_{\ast}, + \infty) \to \mathbb{R}_+$, and assume that
\begin{align}
\label{hyp:controle_coefficients_milstein}
\mathfrak{B}(\phi) \quad \equiv \qquad \forall x \in \mathbb{R}^d,  \quad  \vert b(x) \vert^2 + \Tr[ \sigma \sigma^{\ast}(x) ] +\sum_{i,j,l=1}^d \vert \partial_{x_l} \sigma_i ( x )  \sigma_{l,j}(x)  \vert^2   \leqslant C  \phi \circ V (x).
\end{align}

\paragraph{Polynomial case. \\}
In case of Wasserstein convergence, we introduce a weaker assumption than Gaussian distribution for the sequence $(U_n)_{n \in \mathbb{N}^{\ast}} $. Let $q \in \mathbb{N}^{\ast}$, $p \geqslant 0$. We suppose that $(U_n)_{n \in \mathbb{N}^{\ast}} $ is a sequence of independent and identically distributed random variables such that
\begin{align}
\label{hyp:matching_normal_moment_ordre_q_va_schema_milstein}
M_{\mathcal{N},q}(U) \quad \equiv \qquad \forall n \in \mathbb{N}^{\ast} ,\forall \tilde{q} \in \{1, \ldots, q\} , \quad \mathbb{E}[(U_n)^{\otimes \tilde{q}}]=\mathbb{E}[(\mathcal{N}(0,I_d))^{\otimes \tilde{q}}],
\end{align}
and
\begin{align}
\label{hyp:moment_ordre_p_va_schema_milstein}
M_p(U) \qquad \sup_{n \in \mathbb{N}^{\ast}} \mathbb{E}[\vert U_n \vert^{2p} ] < + \infty.
\end{align}

Moreover, we will also assume that $(\mathcal{W}_n)_{n \in \mathbb{N}^{\ast} }$ is a sequence of independent and centered random variables such that $M_{p}(\mathcal{W})$ holds for some $p$ we will precise further on.\\

We are now able to introduce the mean-reverting property of $V$. Let $p \geqslant 0$. Let $\beta \in \mathbb{R}$, $\alpha>0$. We assume that $\liminf\limits_{y \to \infty} \phi(y)>\beta/\alpha$ and 
\begin{align}
\label{hyp:recursive_control_param_milstein}
\mathcal{R}_p(\alpha,\beta, \phi, V) \quad \equiv \qquad \forall x \in \mathbb{R}^d,  \quad  \langle \nabla V(x), b(x) \rangle+   \frac{1}{2} \chi_{p}(x) \leqslant \beta - \alpha \phi \circ V (x),
\end{align}
with

\begin{equation}
\label{hyp:recursive_control_param_terme_ordre_sup_milstein}
\chi_{p}(x) =  \left\{
      \begin{aligned}
        & \Vert \lambda_{1} \Vert_{\infty}  \mbox{Tr}[\sigma \sigma^{\ast}(x)]  & & \quad \mbox{if } p \leqslant 1\\
        & \Vert \lambda_{p} \Vert_{\infty} 2^{(2p-3)_+} \mbox{Tr}[\sigma \sigma^{\ast}(x)]  &   & \quad \mbox{if } p >1.
      \end{aligned}
    \right.
\end{equation}

\begin{mytheo}
\label{th:cv_was_milstein}
Let $p >0,a \in (0,1]$, $s \geqslant 1, \rho \in [1,2]$ and, $\psi_p(y)=y^p$, $\phi(y)=y^a$ and $\epsilon_{\mathcal{I}}(\gamma)=\gamma^{\rho/2}$. Let $\alpha>0$ and $\beta \in \mathbb{R}$. \\

Assume that the sequence $(U_n)_{n \in \mathbb{N}^{\ast}}$ satisfies $M_{\mathcal{N},2}(U)$ (see (\ref{hyp:matching_normal_moment_ordre_q_va_schema_milstein})) and $M_{p \vee (p \rho/s) \vee 1}(U)$ (see (\ref{hyp:moment_ordre_p_va_schema_milstein})). Moreover, assume that $(\mathcal{W}_n)_{n \in \mathbb{N}^{\ast} }$ is a sequence of independent and centered random variables such that $M_{p \vee (p \rho/s) \vee 1}(\mathcal{W})$ (see (\ref{hyp:moment_ordre_p_va_schema_milstein})) holds.\\
Also assume that (\ref{hyp:Lyapunov_control_milstein}), $\mathfrak{B}(\phi)$ (see (\ref{hyp:controle_coefficients_milstein})), $\mathcal{R}_p(\alpha,\beta, \phi, V)$ (see (\ref{hyp:recursive_control_param_milstein})), $\mbox{L}_{V}$ (see (\ref{hyp:Lyapunov}), $\mathcal{S}\mathcal{W}_{\mathcal{I}, \gamma,\eta}(\rho, \epsilon_{\mathcal{I}})$ (see (\ref{hyp:step_weight_I})), $\mathcal{S}\mathcal{W}_{\mathcal{II},\gamma,\eta}(V^{p/s}) $ (see (\ref{hyp:step_weight_I_gen_tens})) and (\ref{hyp:accroiss_sw_series_2}) also hold and that $ap\rho/s \leqslant p+a-1$. \\

Then, if $p/s+a-1>0$, $(\nu_n^{\eta})_{n \in \mathbb{N}^{\ast}}$ is $\mathbb{P}-a.s.$ tight and
\begin{align}
 \label{eq:tightness_milstein}
\mathbb{P} \mbox{-a.s.} \quad \sup_{n \in \mathbb{N}^{\ast}} \nu_n^{\eta}( V^{p/s+a-1} ) < + \infty .
\end{align}

%
%
%
%
%

Moreover, assume also that $b$, $\sigma$ and $\sum_{i,j,l=1}^d  \vert  \partial_{x_l} \sigma_i  \sigma_{l,j} \vert  $ have sublinear growth and that $g_{\sigma}\leqslant C  V^{p/s+a-1}$, with $g_{\sigma}=\Tr[ \sigma \sigma^{\ast} ]+ \sum_{i,j,l=1}^d  \vert  \partial_{x_l} \sigma_i  \sigma_{l,j} \vert  $. Then, every weak limiting distribution $\nu$ of $(\nu_n^{\eta})_{n \in \mathbb{N}^{\ast}}$ is an invariant distribution of $(X_t)_{t \geqslant 0}$ and when $\nu$ is unique, we have
\begin{align}
\label{eq:cv_was_milstein}
\mathbb{P} \mbox{-a.s.} \quad  \forall f \in \mathcal{C}_{\tilde{V}_{\psi_p,\phi,s}}(\mathbb{R}^d), \quad \lim\limits_{n \to + \infty} \nu_n^{\eta}(f)=\nu(f),
\end{align}
 with $\mathcal{C}_{\tilde{V}_{\psi_p,\phi,s}}(\mathbb{R}^d)$ defined in (\ref{def:espace_test_function_cv}). Notice that when $p/s \leqslant p\vee 1 +a-1$, the assumption $\mathcal{S}\mathcal{W}_{\mathcal{II},\gamma,\eta}(V^{p/s}) $ (see (\ref{hyp:step_weight_I_gen_tens})) can be replaced by $\mathcal{S}\mathcal{W}_{\mathcal{II},\gamma,\eta} $ (see (\ref{hyp:step_weight_II})).
\end{mytheo}

 \paragraph{Exponential case. \\} 
 For the exponential case we modify this assumption in the following way. Let $p \leqslant 1/2$ and let $\alpha>0$ and $\beta \in \mathbb{R}$. We assume that $\liminf\limits_{y \to + \infty} \phi(y)>\beta_+/\alpha$, $\beta_+=0 \vee \beta$, and 
\begin{align}
\label{hyp:recursive_control_param_milstein_expo}
\mathcal{R}_{p, \lambda}(\alpha,\beta,\phi,V) \quad \equiv \qquad \forall x \in \mathbb{R}^d,  \quad   \langle \nabla V(x), b(x) + \kappa_p(x) \rangle+   \frac{1}{2}\chi_{p}(x) \leqslant \beta - \alpha \phi \circ V (x),
\end{align}
with

\begin{align*}
\kappa_p(x)=\frac{1}{2} \sum_{i=1}^d \sum_{l=1}^d \partial_{x_l} \sigma_i ( x )  \sigma_{l,i}( x ) + \lambda p \frac{V^{p-1}(x) }{\phi \circ V(x)}\sigma \sigma^{\ast} (x) \nabla V (x) 
\end{align*}
and

\begin{equation*}
\chi_{p}(x) =   -\frac{V^{1-p} (x)}{ \phi \circ V(x)}C_{\sigma}(x)^{-1} \ln (\det(\Sigma(x) ) 
\end{equation*}

with $\Sigma: \mathbb{R}^d \to \mathcal{S}^d_{+,\ast}$, $\mathcal{S}^d_{+,\ast}$ being the set of a positive definite matrix, defined by 

\begin{align*}
\Sigma(x):=I_d \Big(1-2 C_{\sigma}(x)[\sqrt{V}]_1 v_{\ast}^{p-1/2}  \frac{1}{2} \sum_{i,j,l=1}^d  \vert \partial_{x_l} \sigma_i ( x )  \sigma_{l,j}(x) \vert \Big)- \Vert D^2 V  \Vert_{\infty}  C_{\sigma}(x)V^{p-1}(x)  \sigma^{\ast} \sigma(x)
\end{align*}

with $C_{\sigma}: \mathbb{R}^d\to \mathbb{R}_+^{\ast}$ satisfying
$\inf_{x \in \mathbb{R}^d} C_{\sigma}(x)>0$.

%

\begin{mytheo}
\label{th:cv_exp_milstein}
 Let $p\in [0,1/2], \lambda \geqslant 0$, $s \geqslant 1$, $\rho \in [1,2]$ and, let $\phi:[v_{\ast},\infty )\to \mathbb{R}_+$ be a continuous function such that $C_{\phi}:= \sup_{y \in [v_{\ast},\infty )}\phi(y)/y< +\infty$ and $\liminf\limits_{y \to +\infty} \phi(y)=+\infty$, let $\psi(y)=\exp ( \lambda y^p )$, $y\in \mathbb{R}_+$ and let $\epsilon_{\mathcal{I}}(\gamma)=\gamma^{\rho/2}$ and $\tilde{\epsilon}_{\mathcal{I}}(\gamma)=\gamma^{p\rho}$. Let $\alpha>0$ and $\beta \in \mathbb{R}$.\\
 
 Assume that $\rho<s$, $\mathfrak{B}(\phi)$ (see (\ref{hyp:controle_coefficients_milstein})), $\mathcal{R}_{p, \lambda}(\alpha,\beta,\phi,V)$ (see (\ref{hyp:recursive_control_param_milstein_expo})) and $\mbox{L}_{V}$ (see (\ref{hyp:Lyapunov}))  hold.
 Also suppose that $\mathcal{S}\mathcal{W}_{\mathcal{I}, \gamma,\eta}(\rho, \epsilon_{\mathcal{I}})$, $\mathcal{S}\mathcal{W}_{\mathcal{I}, \gamma,\eta}(\rho, \tilde{\epsilon}_{\mathcal{I}})$ (see (\ref{hyp:step_weight_I})), $\mathcal{S}\mathcal{W}_{\mathcal{II},\gamma,\eta} $ (see (\ref{hyp:step_weight_II})), (\ref{hyp:accroiss_sw_series_2}) and (\ref{hyp:dom_recurs_milstein}) hold. \\
 
 Then $(\nu_n^{\eta})_{n \in \mathbb{N}^{\ast}}$ is $\mathbb{P}-a.s.$ tight and
\begin{align}
\label{eq:tightness_milstein_expo}
\mathbb{P} \mbox{-a.s.} \quad \sup_{n \in \mathbb{N}^{\ast}} \nu_n^{\eta} \Big( \frac{\phi \circ V }{V} \exp \big(\lambda/s V^{p}) \Big) < + \infty .
\end{align}

Moreover, assume also that $b$, $\sigma$ and $\sum_{i,j,l=1}^d  \vert  \partial_{x_l} \sigma_i  \sigma_{l,j} \vert  $ have sublinear growth. Then, every weak limiting distribution $\nu$ of $(\nu_n^{\eta})_{n \in \mathbb{N}^{\ast}}$ is an invariant distribution of $(X_t)_{t \geqslant 0}$ and when $\nu$ is unique, we have
\begin{align}
\label{eq:cv_expo_milstein}
 \mathbb{P} \mbox{-a.s.} \quad \forall f \in \mathcal{C}_{\tilde{V}_{\psi,\phi,s}}(\mathbb{R}^d), \quad \mathbb{P} \mbox{-a.s.}\lim\limits_{n \to \infty} \nu_n^{\eta}(f)=\nu(f),
\end{align}
 with $\mathcal{C}_{\tilde{V}_{\psi,\phi,s}}(\mathbb{R}^d)$ defined in (\ref{def:espace_test_function_cv}).

\end{mytheo}

\subsubsection{Recursive control}
\paragraph{Polynomial case}

\begin{myprop}
\label{prop:recursive_control_milstein}
Let $v_{\ast}>0$, and let $\phi:[v_{\ast},\infty )\to \mathbb{R}_+$ be a continuous function such that $C_{\phi}:= \sup_{y \in [v_{\ast},\infty )}\phi(y)/y<+ \infty$.  Now let $p >0$ and define $\psi_p(y)=y^p$. Let $\alpha>0$ and $\beta \in \mathbb{R}$. \\

Assume that $(U_n)_{n \in \mathbb{N}^{\ast}} $ is a sequence of independent random variables such that $U$ satisfies $M_{\mathcal{N},2}(U) $ (see (\ref{hyp:matching_normal_moment_ordre_q_va_schema_milstein})) and $M_{p \vee 1} (U)$ (see (\ref{hyp:moment_ordre_p_va_schema_milstein})). Moreover, assume that $(\mathcal{W}_n)_{n \in \mathbb{N}^{\ast} }$ is a sequence of independent and centered random variables such that $M_{p \vee 1}(\mathcal{W})$ (see (\ref{hyp:moment_ordre_p_va_schema_milstein})) holds. \\
Also assume that (\ref{hyp:Lyapunov_control_milstein}), $\mathfrak{B}(\phi)$ (see (\ref{hyp:controle_coefficients_milstein})), $\mathcal{R}_p(\alpha,\beta, \phi, V)$ (see (\ref{hyp:recursive_control_param_milstein})), are satisfied. \\

Then, for every $\tilde{\alpha}\in (0, \alpha)$, there exists $n_0 \in \mathbb{N}^{\ast}$, such that 

\begin{align}
\label{eq:recursive_control_milsteint_fonction_pol}
 \forall n   \geqslant n_0, \forall x \in \mathbb{R}^d,  \quad\tilde{A}_{\gamma_n} \psi \circ V(x)\leqslant \frac{ \psi_p \circ V(x)}{V(x)}p(\beta - \tilde{\alpha} \phi\circ V(x)). 
\end{align}

Then $\mathcal{RC}_{Q,V}(\psi,\phi,p\tilde{\alpha},p\beta)$ (see (\ref{hyp:incr_sg_Lyapunov})) holds for every $\tilde{\alpha}\in (0, \alpha)$ such that $\liminf\limits_{y \to + \infty} \phi(y) > \beta / \tilde{\alpha}$. Moreover, when $\phi=Id$ we have

\begin{align}
\label{eq:mom_pol_milstein}
\sup_{n \in \mathbb{N}} \mathbb{E}[V^p(\overline{X}_{\Gamma_{n}})] < + \infty.
\end{align}
\end{myprop}

\begin{proof} 
We distinguish the cases $p \geqslant 1$ and $p \in(0,1)$.
\paragraph{Case $p\geqslant 1$. }
 First ,we focus on the case $p \geqslant 1$.  From the Taylor's formula and the definition of $\lambda_{\psi_p}=\lambda_p$ (see (\ref{def:lambda_psi_milstein})), we have

\begin{align}
\label{eq:taylor_preuve_RC_pol_milstein}
\psi_{p} \circ V(\overline{X}_{\Gamma_{n+1}})=& \psi_{p} \circ V(\overline{X}_{\Gamma_n})+ \langle \overline{X}_{\Gamma_{n+1}}-\overline{X}_{\Gamma_n}, \nabla V(\overline{X}_{\Gamma_n}) \rangle \psi_{p}'\circ V(\overline{X}_{\Gamma_n}) \nonumber \\
&+ \frac{1}{2} ( D^2 V(\Upsilon_{n+1}) \psi_{p}' \circ  V(\Upsilon_{n+1})+\nabla V (\Upsilon_{n+1} )^{\otimes 2} \psi_{p}'' \circ V(\Upsilon_{n+1}) )
 ( \overline{X}_{\Gamma_{n+1}}-\overline{X}_{\Gamma_n} )^{\otimes 2} \nonumber \\
 \leqslant & \psi_{p} \circ V(\overline{X}_{\Gamma_n})+ \langle \overline{X}_{\Gamma_{n+1}}-\overline{X}_{\Gamma_n}, \nabla V(\overline{X}_{\Gamma_n}) \rangle \psi_{p}'\circ V(\overline{X}_{\Gamma_n})  \nonumber \\
&+ \frac{1}{2} \lambda_{p} (\Upsilon_{n+1} )  \psi_{p}'\circ V(\Upsilon_{n+1}) \vert  \overline{X}_{\Gamma_{n+1}}-\overline{X}_{\Gamma_n}  \vert^{2}. 
\end{align}

%

with $\Upsilon_{n+1}  \in (\overline{X}_{\Gamma_n}, \overline{X}_{\Gamma_{n+1}})$. First, from (\ref{hyp:Lyapunov_control_milstein}), we have $\sup_{x \in \mathbb{R}^d} \lambda_{p} (x)  < + \infty$.
%
 
 Since $\mathcal{W}$ is made of centered random variables, we deduce from $M_{\mathcal{N},2}(U) $ (see (\ref{hyp:matching_normal_moment_ordre_q_va_schema_milstein})), $M_{ 2} (U)$ (see (\ref{hyp:moment_ordre_p_va_schema_milstein})) and $M_{2}(\mathcal{W})$ (see (\ref{hyp:moment_ordre_p_va_schema_milstein})), that

\begin{align*}
& \mathbb{E}[\overline{X}_{\Gamma_{n+1}}-\overline{X}_{\Gamma_n} \vert \overline{X}_{\Gamma_n} ]= \gamma_{n+1}  b(\overline{X}_{\Gamma_n} )\\
& \mathbb{E}[ \vert  \overline{X}_{\Gamma_{n+1}}-\overline{X}_{\Gamma_n} \vert^{2} \vert \overline{X}_{\Gamma_n} ]\leqslant \gamma_{n+1}\mbox{Tr}[\sigma \sigma^{\ast}(\overline{X}_{\Gamma_n} )]  + \gamma_{n+1}^2 \vert b(\overline{X}_{\Gamma_n} ) \vert^{2}+c_d  \gamma_{n+1}^2\sum_{i,j,l=1}^d \vert \partial_{x_l} \sigma_i ( \overline{X}_{\Gamma_n}  )  \sigma_{l,j}(\overline{X}_{\Gamma_n} )  \vert^2  \\
& \qquad \qquad \qquad \qquad \qquad  \qquad +c_d  \gamma_{n+1}^{3/2}\sum_{i,j,l=1}^d \vert \partial_{x_l} \sigma_i ( \overline{X}_{\Gamma_n}  )  \sigma_{l,j}(\overline{X}_{\Gamma_n} )  \vert \mbox{Tr}[\sigma \sigma^{\ast}(\overline{X}_{\Gamma_n} )]^{1/2}
\end{align*}

with $c_d$ a positive constant. Assume first that $p=1$. Using $\mathfrak{B}(\phi)$ (see (\ref{hyp:controle_coefficients_milstein})), for every $\tilde{\alpha} \in (0, \alpha)$, there exists $n_0(\tilde{\alpha})$ such that for every $n\geqslant n_0(\tilde{\alpha})$,
\begin{align}
\label{eq:recursive_control_remaider_Id_milstein}
\frac{1}{2} \Vert \lambda_{1} \Vert_{\infty}    \gamma_{n+1}^2 ( & \vert b(\overline{X}_{\Gamma_n} ) \vert^2 +  c_d\vert  \sum_{i,j,l=1}^d \partial_{x_l} \sigma_i ( \overline{X}_{\Gamma_n} )  \sigma_{l,j}( \overline{X}_{\Gamma_n} )  \vert^2 ) \\
+ & \frac{1}{2} \Vert \lambda_{1} \Vert_{\infty} c_d  \gamma_{n+1}^{3/2}\sum_{i,j,l=1}^d \vert \partial_{x_l} \sigma_i ( \overline{X}_{\Gamma_n}  )  \sigma_{l,j}(\overline{X}_{\Gamma_n} )  \vert \mbox{Tr}[\sigma \sigma^{\ast}(\overline{X}_{\Gamma_n} )]^{1/2} \leqslant \gamma_{n+1}(\alpha- \tilde{\alpha})\phi \circ V(\overline{X}_{\Gamma_n} ).  \nonumber 
\end{align}
From assumption $\mathcal{R}_p(\alpha,\beta, \phi, V)$ (see (\ref{hyp:recursive_control_param_milstein}) and (\ref{hyp:recursive_control_param_terme_ordre_sup_milstein})), we conclude that
\begin{align*}
\tilde{A}_{\gamma_n} \psi \circ V(x) \leqslant \beta - \tilde{\alpha} \phi \circ V(x)
\end{align*}
Assume now that $p>1$.Since $\vert \nabla V \vert \leqslant C_V V$ (see (\ref{hyp:Lyapunov_control_milstein})), then $\sqrt{V}$ is Lipschitz. Now, we use the following inequality: Let $l \in \mathbb{N}^{\ast}$. We have

\begin{align}
\label{eq:puisance_somme_n_terme}
\forall \alpha >0 ,\forall u_i \in \mathbb{R}^d, i=1,\ldots,l, \qquad  \big \vert  \sum_{i=1}^l u_i \big \vert^{\alpha} \leqslant l^{(\alpha-1)_+}  \sum_{i=1}^l \vert u_i  \vert^{\alpha} .
\end{align}
\begin{align*}
V^{p-1} (\Upsilon_{n+1} )  \leqslant & \big( \sqrt{V}(\overline{X}_{\Gamma_n})+[\sqrt{V}]_1 \vert \overline{X}_{\Gamma_{n+1}}-\overline{X}_{\Gamma_n} \vert \big)^{2p-2} \\
\leqslant & 2^{(2p-3)_+} (V^{p-1}(\overline{X}_{\Gamma_n}) + [\sqrt{V}]_1^{2p-2} \vert \overline{X}_{\Gamma_{n+1}}-\overline{X}_{\Gamma_n} \vert^{2p-2})
\end{align*}
To study the `remainder' of (\ref{eq:taylor_preuve_RC_pol_milstein}), we multiply the above inequality by $\vert \overline{X}_{\Gamma_{n+1}}-\overline{X}_{\Gamma_n} \vert^{2}$. First, we study the second term which appears in the $r.h.s.$ and using $\mathfrak{B}(\phi)$ (see (\ref{hyp:controle_coefficients_milstein})), for every $p \geqslant 1$,
\begin{align*}
\vert  \overline{X}_{\Gamma_{n+1}}- \overline{X}_{\Gamma_n} \vert^{2p} \leqslant C \gamma_{n+1}^p \phi \circ V(\overline{X}_{\Gamma_{n}})^p(1+\vert U_{n+1} \vert^{2p} +\vert \mathcal{W}_{n+1} \vert^{2p}).
\end{align*}

 Let $\hat{\alpha} \in (0,\alpha)$. Then, we deduce from $M_{2p } (U)$ (see (\ref{hyp:moment_ordre_p_va_schema_milstein})), $M_{2p }(\mathcal{W})$ (see (\ref{hyp:moment_ordre_p_va_schema_milstein})), that there exists $n_0(\hat{\alpha}) \in \mathbb{N}$ such that for any $n \geqslant n_0(\hat{\alpha})$, we have
\begin{align*}
 \mathbb{E}[ \vert \overline{X}_{\Gamma_{n+1}}- \overline{X}_{\Gamma_n} \vert^{2p} \vert \overline{X}_{\Gamma_n} ] \leqslant \gamma_{n+1} \phi \circ V (\overline{X}_{\Gamma_{n}})^{p} \frac{\alpha- \hat{\alpha} }{\Vert \phi /I_d \Vert_{\infty}^{p-1} \Vert \lambda_{p} \Vert_{\infty} 2^{(2p-3)_+} [\sqrt{V}]_1^{2p-2} } 
\end{align*}

To treat the other term of the `remainder' of (\ref{eq:taylor_preuve_RC_pol_milstein}) we proceed as in (\ref{eq:recursive_control_remaider_Id_milstein}) with $\Vert \lambda_{1} \Vert_{\infty}$ replaced by $\Vert \lambda_{p} \Vert_{\infty} 2^{2p-3} [\sqrt{V}]_1^{2p-2}  $, $\alpha$ replace by $ \hat{\alpha}$ and $\tilde{\alpha} \in (0, \hat{\alpha})$. We gather all the terms together and using (\ref{hyp:recursive_control_param_terme_ordre_sup_milstein}), for every $n \geqslant n_0(\tilde{\alpha}) \vee n_0(\hat{\alpha}) $, we obtain

\begin{align*}
\mathbb{E}[V^p (\overline{X}_{\Gamma_{n+1}})- V^p(\overline{X}_{\Gamma_{n}}) \vert \overline{X}_{\Gamma_{n}}] \leqslant & \gamma_{n+1}p V^{p-1}(\overline{X}_{\Gamma_{n}})( \beta - \alpha \phi \circ V (\overline{X}_{\Gamma_{n}})  ) \\
+ &\gamma_{n+1}p V^{p-1}(\overline{X}_{\Gamma_{n}}) \big(  \phi \circ V (\overline{X}_{\Gamma_{n}}) (\hat{\alpha} -\tilde{\alpha}  ) + (\alpha-\hat{\alpha})  \frac{ V^{1-p}(\overline{X}_{\Gamma_{n}})  \phi \circ V (\overline{X}_{\Gamma_{n}})^{p} }{\Vert \phi /I_d \Vert_{\infty}^{p-1}}    \big) \\
\leqslant&\gamma_{n+1} V^{p-1}(\overline{X}_{\Gamma_{n}})( \beta p- \tilde{\alpha} p  \phi \circ V (\overline{X}_{\Gamma_{n}})  ). \\
\end{align*}
which is exactly the recursive control for $p>1$. 

\paragraph{Case $p\in(0,1)$. }
Now, let $p \in (0,1)$ so that $x \mapsto x^{p}$ is concave. it follows that
\begin{align*}
V^{p}(\overline{X}_{\Gamma_{n+1}} ) - V^{p}(\overline{X}_{\Gamma_{n}} ) \leqslant pV^{p-1}(\overline{X}_{\Gamma_{n}}  ) (V(\overline{X}_{\Gamma_{n+1}} ) - V(\overline{X}_{\Gamma_{n}} ) )
\end{align*}
We have just proved that we have the recursive control $\mathcal{RC}_{Q,V}(\psi , \phi, \alpha, \beta)$ holds for $\psi=I_d$ (with $\alpha$ replaced by $\tilde{\alpha} >0$), and since $V$ takes positive values, we obtain

\begin{align*}
\mathbb{E}[V^{p}(\overline{X}_{\Gamma_{n+1}} ) - V^{p}(\overline{X}_{\Gamma_{n}} )\vert \overline{X}_{\Gamma_{n}} ] \leqslant  & pV^{p-1}(\overline{X}_{\Gamma_{n}}  ) \mathbb{E}[V(\overline{X}_{\Gamma_{n+1}} ) - V(\overline{X}_{\Gamma_{n}} ) \vert \overline{X}_{\Gamma_{n}} ] \\
 \leqslant & V^{p-1}(\overline{X}_{\Gamma_{n}}  ) (p\beta - p\tilde{\alpha} \phi \circ V (\overline{X}_{\Gamma_{n}}  ) ),
\end{align*}
which completes the proof of (\ref{eq:recursive_control_milsteint_fonction_pol}). The proof of (\ref{eq:mom_pol_milstein}) is an immediate application of Lemma \ref{lemme:mom_psi_V} as soon as we notice that the increments of the Milstein scheme have finite polynomial moments which implies (\ref{eq:mom_psi_V}).

\end{proof}

\paragraph{Exponential case \\}

In this section we will not relax the assumption on the Gaussian structure of the increment as we do in the polynomial case with hypothesis  (see (\ref{hyp:matching_normal_moment_ordre_q_va_schema_milstein}) and (\ref{hyp:moment_ordre_p_va_schema_milstein})). In order to obtain our result, we introduce a supplementary assumption in order to express the iterated stochastic integrals in terms of products of the increments of the Brownian motion. The so called commutative noise assumption is the following
\begin{align*}
\forall x \in \mathbb{R}^d, \forall i,j \in \{1,\ldots,d \}, \qquad  \sum_{l=1}^d \partial_{x_l} \sigma_i ( x)  \sigma_{l,j}( x )  =\sum_{l=1}^d \partial_{x_l} \sigma_j ( x )  \sigma_{l,i}( x )  .
\end{align*}

In this case, with the notation from (\ref{def:incr_milstein}), we have
\begin{align*}
\Delta \overline{X}_{n+1}^{3}= &\frac{1}{2} \sum_{i,j=1}^d \sum_{l=1}^d \partial_{x_l} \sigma_i ( \overline{X}_{\Gamma_n} )  \sigma_{l,j}( \overline{X}_{\Gamma_n} )   (W^j_{\Gamma_{n+1}}- W^j_{\Gamma_{n}} )(W^i_{\Gamma_{n+1}}-W^i_{\Gamma_{n}}) \\
- & \frac{1}{2} \gamma_{n+1} \sum_{i=1}^d \sum_{l=1}^d \partial_{x_l} \sigma_i ( \overline{X}_{\Gamma_n} )  \sigma_{l,i}( \overline{X}_{\Gamma_n} ).    \nonumber
\end{align*}
In the sequel we will adopt the following notation

\begin{align}
\label{def:incr_milstein_expo}
\tilde{\Delta} \overline{X}^1_{n+1} = &  \gamma_{n+1}b(\overline{X}_{\Gamma_{n}}) -  \frac{1}{2} \gamma_{n+1} \sum_{i=1}^d \sum_{l=1}^d \partial_{x_l} \sigma_i ( \overline{X}_{\Gamma_n} )  \sigma_{l,i}( \overline{X}_{\Gamma_n} ).    , \nonumber \\
\tilde{\Delta} \overline{X}^{3}_{n+1} = &   \frac{1}{2} \sum_{i,j=1}^d \sum_{l=1}^d \partial_{x_l} \sigma_i ( \overline{X}_{\Gamma_n} )  \sigma_{l,j}( \overline{X}_{\Gamma_n} )   (W^j_{\Gamma_{n+1}}- W^j_{\Gamma_{n}} )(W^i_{\Gamma_{n+1}}-W^i_{\Gamma_{n}})
\end{align}

\begin{lemme}
\label{lemme:transformee_laplace_prod_gauss}
Let $\Lambda \in \mathbb{R}^{d \times d}$ and $U \sim \mathcal{N}(0,I_d)$. We define $\Sigma \in \mathbb{R}^{d \times d}$ by $\Sigma=I_d-2 \Lambda^{\ast} \Lambda$. Assume that $\Sigma \in \mathcal{S}_{+,\ast}^d$. Then,  for every $h \in (0,1)$,

\begin{align}
\label{eq:transformee_laplace_jointe}
\forall v \in \mathbb{R}^d , \qquad \mathbb{E}\Big[\exp \Big( \sqrt{h}  \langle v , U \rangle    + h \vert \Lambda U \vert^2 \Big) \Big]  \leqslant   \exp\Big( \frac{h}{2(1-h) } \vert v \vert^2    \Big) \det( \Sigma)^{-h/2}.
\end{align}


\end{lemme}
\begin{proof}
A direct computation yields
\begin{align*}
\mathbb{E}[\exp ( \vert \Lambda U\vert^2  )] = & \int_{\mathbb{R}^d} (2 \pi)^{-d/2}\exp \Big( -\frac{1}{2} \langle -2\Lambda^{\ast} \Lambda u+u,u \rangle \Big) du = \det( \Sigma)^{-1/2}.
\end{align*}
Now, (\ref{eq:transformee_laplace_jointe}) follows from the H\" older inequality since
\begin{align*}
 \mathbb{E}[\exp (  \sqrt{h}  \langle v , U \rangle    + h \vert \Lambda U\vert^2  )]   \leqslant   \mathbb{E}[\exp ( \frac{\sqrt{h}}{1-h}  \langle v , U \rangle )] ^{1-h}   \mathbb{E}[\exp (  \vert \Lambda U\vert^2 )]^h = \exp\Big( \frac{h}{2(1-h) } \vert v \vert^2    \Big) \det( \Sigma)^{-h/2}.
\end{align*}
\end{proof}

Using this result, we deduce the recursive control for exponential test functions. 
\begin{myprop}
\label{prop:recursive_control_milstein_exp}
Let $v_{\ast}>0$, and let $\phi:[v_{\ast},\infty )\to \mathbb{R}_+$ be a continuous function such that $C_{\phi}:= \sup_{y \in [v_{\ast},\infty )}\phi(y)/y< +\infty$.  Now let $p \in (0,1/2]$, $\lambda \geqslant 0$ and define $\psi:[v_{\ast},\infty )\to \mathbb{R}_+$ such that $\psi(y)=\exp ( \lambda y^p )$. \\

Assume that (\ref{hyp:Lyapunov_control_milstein}), $\mathfrak{B}(\phi)$ (see (\ref{hyp:controle_coefficients_milstein})) and $\mathcal{R}_{p, \lambda}(\alpha,\beta,\phi,V)$ (see (\ref{hyp:recursive_control_param_milstein_expo})), are satisfied.Also assume that
\begin{align}
\label{hyp:dom_recurs_milstein}
\forall x \in \mathbb{R}^d , \quad \Tr[\sigma \sigma^{\ast} (x) ] \vert b(x) \vert ( \sqrt{V}(x)+ \vert b (x) \vert ) \leqslant C V^{1-p}(x) \phi \circ V (x).
\end{align}

 Then, for every $\tilde{\alpha} \in (0,\alpha)$, there exists $\tilde{\beta} \in \mathbb{R}_+$ and $n_0 \in \mathbb{N}^{\ast}$, such that
\begin{align}
\label{eq:recursive_control_milstein_fonction_expo}
 \forall n   \geqslant n_0, x \in \mathbb{R}^d,  \quad\tilde{A}_{\gamma_n} \psi \circ V(x)\leqslant \frac{ \psi \circ V(x)}{V(x)}p(\tilde{\beta} - \tilde{\alpha} \phi\circ V(x)). 
\end{align}
Then, $\mathcal{RC}_{Q,V}(\psi,\phi,p\tilde{\alpha},p\tilde{\beta})$ (see (\ref{hyp:incr_sg_Lyapunov})) holds as soon as $\liminf\limits_{y \to +\infty} \phi(y) =+\infty$.
Moreover, when $\phi=Id$ we have

\begin{align}
\label{eq:mom_exp_milstein}
\sup_{n \in \mathbb{N}} \mathbb{E}[\psi \circ V (\overline{X}_{\Gamma_{n}})] < + \infty.
\end{align}

\end{myprop}

\begin{proof}
 First, with notations (\ref{def:incr_milstein_expo}), we rewrite

\begin{align*}
V^p(\overline{X}_{\Gamma_{n+1}})-V^p(\overline{X}_{\Gamma_{n}})=& V^p(\overline{X}_{\Gamma_{n}}+\tilde{\Delta} \overline{X}^1_{n+1} +\Delta \overline{X}^{2}_{n+1} ) - V^p(\overline{X}_{\Gamma_{n}}) \\
& + V^p(\overline{X}_{\Gamma_{n+1}})  - V^p(\overline{X}_{\Gamma_{n}}+\tilde{\Delta} \overline{X}^1_{n+1} +\Delta \overline{X}^{2}_{n+1} ) 
\end{align*}

an we study each term separately. Since $p \leqslant 1$, the function defined on $[v_{\ast}, \infty)$ by  $y \mapsto y^p$ is concave. Using then the Taylor expansion of order 2 of the function $V$, for every $x,y \in \mathbb{R}^d$, 

\begin{align*}
V^p(y) - V^p(x) \leqslant & pV^{p-1}(x) \big(V(y)-V(x) \big) \\
\leqslant & pV^{p-1}(x) \big( \langle \nabla V(x),y-x \rangle +\frac{1}{2} \Vert D^2V \Vert_{\infty} \vert y - x \vert^2 \big).
\end{align*}

Using this inequality with $x=\overline{X}_{\Gamma_{n}}$ and $y=\overline{X}_{\Gamma_{n+1}}=\overline{X}_{\Gamma_{n}}+\tilde{\Delta} \overline{X}^1_{n+1} +\Delta \overline{X}^2_{n+1} $, with notations (\ref{def:incr_milstein}) and (\ref{def:incr_milstein_expo})
\begin{align*}
V^p(\overline{X}_{\Gamma_{n}}+\tilde{\Delta} \overline{X}^1_{n+1} +\Delta \overline{X}^{2}_{n+1} ) - V^p(\overline{X}_{\Gamma_{n}}) \leqslant & pV^{p-1} \overline{X}_{\Gamma_{n}} )( \langle \nabla V(\overline{X}_{\Gamma_{n}}),\tilde{\Delta} \overline{X}^1_{n+1} +\Delta \overline{X}^{2}_{n+1} \rangle \\
&+ \frac{1}{2}  pV^{p-1}(\overline{X}_{\Gamma_{n}} ) \Vert D^2V \Vert_{\infty} \vert \tilde{\Delta} \overline{X}^1_{n+1} +\Delta \overline{X}^{2}_{n+1} \vert^2.
\end{align*}

%

Now, we study the other term. Since $p \leqslant 1/2$, then the function defined on $[v_{\ast}, \infty)$ by $y \mapsto y^{2p}$ is concave and we obtain
\begin{align*}
V^{p}(\overline{X}_{\Gamma_{n+1}} ) -& V^{p}(\overline{X}_{\Gamma_{n}} +\tilde{\Delta} \overline{X}^1_{n+1} +\Delta \overline{X}^{2}_{n+1}) \\
\leqslant & pV^{p-1/2}(\tilde{\Delta} \overline{X}^1_{n+1} +\Delta \overline{X}^{2}_{n+1} ) (\sqrt{V}(\overline{X}_{\Gamma_{n+1}} ) - \sqrt{V}(\overline{X}_{\Gamma_{n}} + \tilde{\Delta} \overline{X}^1_{n+1} +\Delta \overline{X}^{2}_{n+1}) ) \\
\leqslant & p [\sqrt{V}]_1 v_{\ast}^{p-1/2} \vert \tilde{\Delta}\overline{X}^{3}_{n+1}\vert 
\end{align*}

In the sequel, we will use the notation
\begin{align*}
\forall x \in \mathbb{R}^d, \qquad \tilde{b}(x)= b(x) +\frac{1}{2} \sum_{i=1}^d \sum_{l=1}^d \partial_{x_l} \sigma_i ( x )  \sigma_{l,i}( x ).
\end{align*}

It follows that

\begin{align*}
\mathbb{E}[\exp(\lambda V^p(\overline{X}_{\Gamma_{n+1}})) \vert \overline{X}_{\Gamma_{n}}] \leqslant H_{\gamma_{n+1}}(\overline{X}_{\Gamma_{n}})L_{\gamma_{n+1}}(\overline{X}_{\Gamma_{n}})
\end{align*}
with, for every $x \in \mathbb{R}^d$, $\gamma \in \mathbb{R}_+^{\ast}$,
\begin{align*}
H_{\gamma}(x)=\exp(\lambda V^p(x)+\gamma \lambda pV^{p-1}(x) \langle \nabla V(x), \tilde{b}(x) \rangle +\gamma^2  \frac{1}{2} \lambda   p \Vert D^2 V \Vert_{\infty} V^{p-1}(x)\vert \tilde{b} (x) \vert^2 )
\end{align*}
and
%
%

\begin{align*}
L_{\gamma}(x)= &\mathbb{E}[\exp( \sqrt{\gamma}  \lambda p V^{p-1}(x) \langle \nabla V(x)+ \gamma \Vert D^2 V \Vert_{\infty}   \tilde{b}(x) , \sigma(x)  U \rangle+ \gamma \frac{1}{2}\lambda  p  \Vert D^2 V  \Vert_{\infty}   V^{p-1}(x) \vert \sigma(x) U \vert^2  \\
&+ \gamma \lambda  p [\sqrt{V}]_1 v_{\ast}^{p-1/2}  \frac{1}{2} \sum_{i,j=1}^d \sum_{l=1}^d  \vert \partial_{x_l} \sigma_i ( x )  \sigma_{l,j}(x) \vert  \vert U\vert^2 
\end{align*}

%


where $U=(U_1, \ldots U_d)$, with $U_i$, $i \in \{1, \ldots ,d \}$, some independent and identically distributed standard normal random variables. In order to compute $L_{\gamma}(x)$, we use Lemma \ref{lemme:transformee_laplace_prod_gauss} (see (\ref{eq:transformee_laplace_jointe})) with $h=C_{\sigma}(x)^{-1} \gamma \lambda p $, $v=\sqrt{C_{\sigma}(x)\lambda p }V^{p-1}(x)\sigma^{\ast}(x)(  \nabla V(x)+ \gamma    \Vert D^2 V \Vert_{\infty}   \tilde{b}(x) )$ and the positive definite matrix 
\begin{align*}
\Sigma(x):=I_d \Big(1-2 C_{\sigma}(x)[\sqrt{V}]_1 v_{\ast}^{p-1/2}  \frac{1}{2} \sum_{i,j,l=1}^d  \vert \partial_{x_l} \sigma_i ( x )  \sigma_{l,j}(x) \vert \Big)- \Vert D^2 V  \Vert_{\infty}  C_{\sigma}(x,z)V^{p-1}(x)  \sigma^{\ast} \sigma(x,z)
\end{align*}
 where $\inf_{x \in \mathbb{R}^d} C_{\sigma}(x)>0$ and $\Sigma(x,z) \in \in \mathcal{S}_{+, \ast}^d)$. We apply Lemma \ref{lemme:transformee_laplace_prod_gauss} and it follows that for $\gamma \leqslant \inf_{x \in \mathbb{R}^d}C_{\sigma}(x)/(2 \lambda p)$

\begin{align*}
L_\gamma(x) \leqslant &\exp \Big(\frac{\gamma \lambda p  C_{\sigma}(x)^{-1}}{2(1 -\gamma \lambda p  C_{\sigma}(x)^{-1} )} \vert v \vert^2- \frac{1}{2} \gamma \lambda p  C_{\sigma}(x)^{-1} \ln ({\det(\Sigma(x))})  \Big) \\
\leqslant & \exp \Big(\gamma  \lambda p  C_{\sigma}(x)^{-1} \vert v \vert^2- \frac{1}{2} \gamma \lambda p  C_{\sigma}(x)^{-1} \ln ({\det(\Sigma(x))})  \Big)
\end{align*}

At this point, we focus on the first term inside the exponential. We have
\begin{align*}
\vert v \vert^2 \leqslant & C_{\sigma}(x ) \lambda p  V^{2p-2}(x)   \big( \langle \sigma \sigma^{\ast} (x )\nabla V(x),\nabla V(x)\rangle \\
&+\Tr[\sigma \sigma^{\ast}(x )]( \gamma   \Vert D^2 V \Vert_{\infty} 2  \langle \nabla V(x), b(x )  \rangle + \gamma^2  \Vert D^2 V \Vert_{\infty}^2 \vert   b(x ) \vert^2 ) \big)
\end{align*}
Using $\mathfrak{B}(\phi)$ (see (\ref{hyp:controle_coefficients_milstein})), (\ref{hyp:dom_recurs_milstein}) and  $\mathcal{R}_{p, \lambda}(\alpha,\beta,\phi,V)$ (see (\ref{hyp:recursive_control_param_milstein_expo})), it follows that there exists $\overline{C}>0$ such that
\begin{align*}
H_{\gamma}(x )L_{\gamma}(x ) \leqslant & \exp \big(\lambda V^p(x)+\gamma \lambda p  V^{p-1}(x)(\beta - \alpha \phi \circ V(x))+\overline{C} \gamma^2 V^{p-1}(x) \phi \circ V(x) \big) 
\end{align*}
which can be rewritten
\begin{align*}
H_{\gamma}(x )L_{\gamma}(x ) \leqslant & \exp \Big( \Big(1-\gamma p \alpha \frac{\phi \circ V(x)}{V(x)} \Big) \lambda V^{p}(x) + \gamma p  \alpha \frac{\phi \circ V(x)}{V(x)}  V^p(x) \Big( \frac{ \lambda  \beta }{ \alpha \phi \circ V(x)}+\gamma \overline{C}/(\alpha p) \Big) \Big ) .
\end{align*} 
Using the convexity of the exponential function, we have for every $\gamma p \alpha C_{\phi}<1$,
\begin{align*}
H_{\gamma}(x )L_{\gamma}(x ) \leqslant & \exp \big(\lambda V^p(x)\big) -\gamma p \alpha \frac{\phi \circ V(x)}{V(x)}  \exp\big(\lambda V^p(x)\big) \\
&+ \gamma p \alpha \frac{\phi \circ V(x)}{V(x)} \exp \Big(V^p(x)\Big( \frac{\lambda  \beta }{ \alpha \phi \circ V(x)} +\gamma \overline{C}/(\alpha p) \Big) \Big).
\end{align*} 

It remains to study the last term of the $r.h.s$ of the above inequality. The function defined on $[v_{\ast},+\infty)$ by $y \mapsto \exp(y^p(\frac{\lambda \beta }{\alpha\phi (y)}+\gamma \overline{C}/(\alpha p) ))$ is continuous and locally bounded. Moreover, by $\mathcal{R}_{p, \lambda}(\alpha,\beta,\phi,V)$ (see (\ref{hyp:recursive_control_param_milstein_expo})), we have $\liminf\limits_{y \to + \infty} \phi(y)>\beta_+/\alpha$. Hence, there exists $\zeta \in (0,1)$ and $y_{\zeta}\geqslant v_{\ast}$ such that $\phi(y) \geqslant \beta_+/(\alpha \zeta)$ for every $y \geqslant y_{\zeta}$. Consequently, as soon as $\gamma <\zeta \lambda \alpha p/ \overline{C}$,  for every $\tilde{\alpha}\in (0, \alpha)$ there exists $\tilde{\beta} \geqslant 0$ such that
\begin{align*}
\frac{\phi \circ V(x)}{V(x)} \exp \Big(V^p(x)\Big( \frac{\lambda  \beta }{ \alpha \phi \circ V(x)} +\gamma \overline{C}/(\alpha p) \Big) \Big)\leqslant \frac{\tilde{\beta}}{\alpha} \frac{\exp(\lambda V^p(x))}{V(x)} +\frac{\alpha -\tilde{\alpha}}{\alpha}  \frac{\phi \circ V(x)}{V(x)} \exp(\lambda V^p(x))
\end{align*}

and the proof of the recursive control (\ref{eq:recursive_control_milstein_fonction_expo}) is completed. Finally (\ref{eq:mom_exp_milstein}) follows from (\ref{eq:mom_psi_V}), which follow from the equation above, and Lemma \ref{lemme:mom_psi_V}.

\end{proof}

\subsubsection{Proof of the infinitesimal estimation}

\begin{myprop}
\label{prop:milstein_infinitesimal_approx}
Assume that the sequence $(U_n)_{n \in \mathbb{N}^{\ast}}$ satisfies $M_{\mathcal{N},2}(U)$ (see (\ref{hyp:matching_normal_moment_ordre_q_va_schema_milstein})) and that the sequence $(\mathcal{W}_n)_{n \in \mathbb{N}^{\ast}}$ is centered and satisfies $M_{1/2}(\mathcal{W})$ (see (\ref{hyp:moment_ordre_p_va_schema_milstein})).\\
Also assume that $b$, $\sigma$ and $\sum_{i,j,l=1}^d  \vert  \partial_{x_l} \sigma_i  \sigma_{l,j} \vert  $ have sublinear growth and that we have $\sup_{n \in \mathbb{N}^{\ast}} \nu_n^{\eta}( \Tr[ \sigma \sigma^{\ast}]  )< + \infty$ and $\sup_{n \in \mathbb{N}^{\ast}} \nu_n^{\eta}(  \sum_{i,j,l=1}^d\vert \partial_{x_l} \sigma_i  \sigma_{l,j} \vert)< + \infty$. \\

Then, $\mathcal{E}(\tilde{A},A,\DomA_0) $ (see (\ref{hyp:erreur_tems_cours_fonction_test_reg})) holds.

\end{myprop}

%
%

\begin{proof}
First, we recall that $\DomA_0= \mathcal{C}^2_K(E)$. The proof consists in studying successively the three terms of the following decomposition:

\begin{align*}
 \mathbb{E}[ f(\overline{X}_{\Gamma_{n+1}} - f(\overline{X}_{\Gamma_{n}} )  \vert \overline{X}_{\Gamma_{n}}=x] =&  \mathbb{E}[  f(\overline{X}^1_{\Gamma_{n+1}}) -  f(\overline{X}_{\Gamma_{n}} )  \vert \overline{X}_{\Gamma_{n}}=x] +   \mathbb{E}[  f(\overline{X}^2_{\Gamma_{n+1}}) -   f(\overline{X}^1_{\Gamma_{n+1}})  \vert \overline{X}_{\Gamma_{n}}=x] \\
+&   \mathbb{E}[  f(\overline{X}^3_{\Gamma_{n+1}}) -   f(\overline{X}^2_{\Gamma_{n+1}})  \vert \overline{X}_{\Gamma_{n}}=x] 
\end{align*}

 Using the Taylor expansion, we have
\begin{align*}
\gamma_{n+1}^{-1} \mathbb{E}[  f(\overline{X}^1_{\Gamma_{n+1}}) - & f(\overline{X}_{\Gamma_{n}} )  \vert \overline{X}_{\Gamma_{n}}=x] - \langle \nabla_xf(x  ),b(x) \rangle  \leqslant \int\limits_0^1 \vert  \nabla f(x+  \theta b(x) \gamma_{n+1} ) - \nabla f(x) \vert  \vert b(x)   \vert d \theta 
\end{align*}

 Using a similar reasonning as in the proof of Proposition 3.3 in \cite{Pages_Rey_2017}, one can show that there exists $R>0$ such that $\mathcal{E}(\tilde{A},A,\DomA_0)$ \ref{hyp:erreur_tems_cours_fonction_test_reg_Lambda_representation_1} holds for $(\tilde{\Lambda}_{f,1} ,\vert b(x) \vert \mathds{1}_{\vert x \vert \leqslant R})$ with

\begin{align*}
\begin{array}{crcl}
\tilde{\mathcal{R}}_{f,1} & :  \mathbb{R}^d \times \mathbb{R}_+   \times [0,1] & \to & \mathbb{R}_+ \\
 &( x, \gamma , \theta ) & \mapsto &   \vert \nabla_xf(x+  \theta b(x) \gamma ) -\nabla_xf(x  ) \vert,
\end{array}
\end{align*}

and $\tilde{\Lambda}_{f,1}(x,t)=\tilde{\mathcal{R}}_{f,1}(x,t, \Theta)$ with $\Theta \sim \mathcal{U}_{[0,1]}$ under $\tilde{\mathbb{P}}(d \tilde{\omega})$.

We focus on the second term. We define $\Lambda_{f,2}(x,z, \gamma)= g_2 (x,z) \tilde{\mathbb{E}}[\tilde{\Lambda}_{f,2}(x,z,\gamma)]$ with $\tilde{\Lambda}^{i,j}_{f,2}(x,\gamma)=\tilde{\mathcal{R}}^{i,j}_{f,2} (x,z,\gamma,U, \Theta)$,
\begin{align*}
\begin{array}{crcl}
\tilde{\mathcal{R}}_{f,2} & :  \mathbb{R}^d \times \mathbb{R}_+ \times \mathbb{R}^{d } \times [0,1] & \to & \mathbb{R}_+ \\
 &( x, \gamma , v,\theta) & \mapsto &  \vert v\vert^2 \vert D^2f(x+\gamma b(x)+\sqrt{\gamma} \theta \sigma(x)v) -D^2f(x)\vert ,
\end{array}
\end{align*}

 with $U \sim p_{U}$, $\Theta \sim \mathcal{U}_{[0,1]}$ under $\tilde{\mathbb{P}}(d \tilde{\omega})$ and $g_{2}(x)=\Tr[ \sigma \sigma^{\ast}(x)]   $.\\
 
%
%
%

  We are going to prove that $\mathcal{E}(\tilde{A},A,\DomA_0)$ \ref{hyp:erreur_tems_cours_fonction_test_reg_Lambda_representation_1} (see (\ref{hyp:erreur_temps_cours_fonction_test_reg_Lambda_representation_2_1})) holds for the couple $(\tilde{\Lambda}_{f,2} ,g_2)$. We fix $v \in \mathbb{R}^{d }$ and $\theta \in [0,1]$. \\
  
 Since the functions $b$ and $ \sigma $, $i,j,l \in \{1, \ldots,d \}$ have sublinear growth, there exists $C_{b, \sigma} \geqslant 0$ such that $\vert b(x)\vert +  \Tr [\sigma \sigma^{\ast} (x )  ]^{1/2}  \leqslant C_{b, \sigma}(1+ \vert x \vert ) $ for every $x \in \mathbb{R}^d$. Therefore, since $f$ has compact support, there exists $t_0(v,\theta)>0$ and $R>0$ such that $\sup_{\vert x \vert >R } \sup_{\gamma \leqslant t_0(v,\theta)} \vert\tilde{\mathcal{R}}_{f,2}(x,\gamma,v,\theta) \vert=0$. It follows that $\mathcal{E}(\tilde{A},A,\DomA_0)$ \ref{hyp:erreur_tems_cours_fonction_test_reg_Lambda_representation_1} (ii) holds.\\
     Moreover since $\nabla f$ is continuous and $b$ and $\sigma$, are locally bounded functions,  it is immediate that $\mathcal{E}(\tilde{A},A,\DomA_0)$ \ref{hyp:erreur_tems_cours_fonction_test_reg_Lambda_representation_1} (i) is also satisfied.\\
   We recall that $\sup_{n \in \mathbb{N}^{\ast}} \nu_n^{\eta}(  \Tr[ \sigma \sigma^{\ast}])< + \infty, \; a.s.$ and $U$ is bounded in $\mbox{L}_2$ and then $\mathcal{E}(\tilde{A},A,\DomA_0)$ \ref{hyp:erreur_tems_cours_fonction_test_reg_Lambda_representation_1} holds for $(\tilde{\Lambda}_{f,2} ,g_2)$.\\

Finally, we notice that from Taylor's formula with $M_{\mathcal{N},2}(U)$ that
\begin{align*}
\mathbb{E}[f(\overline{X}^{2}_{\Gamma_{n+1}})-f(\overline{X}^{1}_{\Gamma_{n+1}}) \vert& \overline{X}_{ \Gamma_n} =x]-\frac{1}{2} \Tr[\sigma \sigma^{\ast}(x) D^2f(x)]\leqslant \gamma_{n+1}  \Tr[\sigma \sigma^{\ast}(x)] \int_{\tilde{\Omega}} \tilde{\Lambda}_{f,2}(x, \gamma, \omega) \tilde{\mathbb{P}}(d \tilde{\omega})
\end{align*}

%

To study the last term, we define $\Lambda^{i,j}_{f,3}(x,z, \gamma)= g^{i,j}_3 (x,z)  \tilde{\mathbb{E}}[\tilde{\Lambda}^{i,j}_{f,3}(x,z,\gamma)]$ with $\tilde{\Lambda}^{i,j}_{f,3}(x,\gamma)=\tilde{\mathcal{R}}^{i,j}_{f,3} (x,z,\gamma,U,\mathcal{W} ,\Theta)$,

\begin{align*}
\begin{array}{crcl}
\tilde{\mathcal{R}}^{i,j}_{f,3} & :  \mathbb{R}^d \times \mathbb{R}_+ \times \mathbb{R}^{d}  \times \mathbb{R}^{d \times d}\times [0,1]& \to & \mathbb{R}_+ \\
 &( x, \gamma , v,w, \theta ) & \mapsto &    \vert  w_{i,j} \vert  \vert \nabla f(x+\gamma b(x)+\sqrt{\gamma} \sigma (x) v+\gamma \theta  \sum\limits_{i,j,l=1}^d \partial_{x_l} \sigma_i (x )  \sigma_{l,j}(x )w_{i,j} \\
 &&& \qquad  \qquad  \qquad  \qquad  \qquad - \nabla f(x)\vert.
\end{array}
\end{align*}.

and $(U,\mathcal{W}) \sim p_{(U,\mathcal{W})}$, $\Theta \sim \mathcal{U}_{[0,1]}$ under $\tilde{\mathbb{P}}(d \tilde{\omega})$ and $g^{i,j}_{3}(x)=\sum_{l=1}^d  \vert \partial_{x_l} \sigma_i ( x)  \sigma_{l,j}(x) \vert    $. \\

We are going to prove that $\mathcal{E}(\tilde{A},A,\DomA_0)$ \ref{hyp:erreur_tems_cours_fonction_test_reg_Lambda_representation_1} (see (\ref{hyp:erreur_temps_cours_fonction_test_reg_Lambda_representation_2_1})) holds for every couple $(\tilde{\Lambda}_{f,3}^{i,j} ,g^{i,j}_3)$, $i,j \in \{1,..,d \}$. We fix $v \in \mathbb{R}^{d \times d}$ and $\theta \in [0,1]$. \\

Since the functions $b$, $\sigma$ and $ \partial_{x_l} \sigma_i  \sigma_{l,j} $, $i,j,l \in \{1, \ldots,d \}$ have sublinear growth, there exists $\overline{C}_{b, \sigma} \geqslant 0$ such that $\vert b(x)\vert +\Tr [\sigma \sigma^{\ast} (x )  ]^{1/2}  +  \sum_{i,j,l=1}^d  \vert  \partial_{x_l} \sigma_i (x )  \sigma_{l,j}(x )\vert  \leqslant \overline{C}_{b, \sigma}(1+ \vert x \vert ) $ for every $x \in \mathbb{R}^d$. Therefore, since $f$ has compact support, there exists $\gamma_0(v,\theta)>0$ and $R>0$ such that $\sup_{\vert x \vert >R } \sup_{\gamma \leqslant \gamma_0(v,\theta)} \vert\tilde{\mathcal{R}}^{i,j}_{f,3}(x,\gamma,v,\theta) \vert=0$. It follows that $\mathcal{E}(\tilde{A},A,\DomA_0)$ \ref{hyp:erreur_tems_cours_fonction_test_reg_Lambda_representation_1} (ii) holds.\\
 Moreover since $\nabla f$ is continuous and $b$, $\sigma$ and $ \partial_{x_l} \sigma_i  \sigma_{l,j} $, $i,j,l \in \{1, \ldots,d \}$ are locally bounded functions,  it is immediate that $\mathcal{E}(\tilde{A},A,\DomA_0)$ \ref{hyp:erreur_tems_cours_fonction_test_reg_Lambda_representation_1} (i) is also satisfied. \\
 We recall that $\sup_{n \in \mathbb{N}^{\ast}} \nu_n^{\eta}(  \sum_{l=1}^d\vert \partial_{x_l} \sigma_i  \sigma_{l,j} \vert)< + \infty, \; a.s.$ and $(\mathcal{W}^{i,j})_{1\leqslant i,j, \leqslant d}$ is bounded in $\mbox{L}_1$ and then $\mathcal{E}(\tilde{A},A,\DomA_0)$ \ref{hyp:erreur_tems_cours_fonction_test_reg_Lambda_representation_1} holds for $(\tilde{\Lambda}^{i,j}_{f,3} ,g^{i,j}_3)$.\\

Finally, it follows from the fact that $\mathcal{W}$ is centered and bounded in $\mbox{L}_1$ and from Taylor's formula,

\begin{align*}
\mathbb{E}[f(\overline{X}^{3}_{\Gamma_{n+1}})-f(\overline{X}^{2}_{\Gamma_{n+1}}) \vert \overline{X}_{ \Gamma_n} =x]=&   
\mathbb{E}[f(\overline{X}^{3}_{\Gamma_{n+1}})-f(\overline{X}^{2}_{\Gamma_{n+1}}) - \gamma_{n+1}\langle \nabla f(x), \Delta \overline{X}^3_{n+1} \rangle \vert \overline{X}_{ \Gamma_n} =x] \\
\leqslant  &\gamma_{n+1} \sum_{i,j=1}^d  g^{i,j}_3(x) \tilde{\mathbb{E}}[\tilde{\Lambda}^{i,j}_{f,3}(x, \gamma)]
\end{align*}

We gather all the terms together noticing that $\tilde{\Lambda}_{f,q}= \tilde{\Lambda}_{-f,q}$, $q \in \{1,\ldots , 3\}$, and the proof is completed.

%
%
%
%
%
\end{proof}

\subsubsection{Proof of Growth control and Step Weight assumptions}

\paragraph{Polynomial case}

\begin{lemme}
\label{lemme:incr_lyapunov_X_milstein}
 Let $p > 0,a \in (0,1]$, $s \geqslant 1$, $\rho \in [1,2]$ and, $\psi(y)=y^p$ and $\phi(y)=y^a$. We suppose that the sequence $(U_n)_{n \in \mathbb{N}^{\ast}}$ satisfies $M_{(\rho/2) \vee (p \rho /s)}(U)$ (see (\ref{hyp:moment_ordre_p_va_schema_milstein})) and that the sequence $(\mathcal{W}_n)_{n \in \mathbb{N}^{\ast}}$ satisfies $M_{(\rho/2) \vee (p \rho /s)}(\mathcal{W})$ (see (\ref{hyp:moment_ordre_p_va_schema_milstein})).
Then, for every $n \in \mathbb{N}$, we have: for every $f \in \DomA_0$,
\begin{align}
\label{eq:incr_lyapunov_X_milstein_f_DomA}
 \mathbb{E}[  \vert f(\overline{X}_{\Gamma_{n+1}})- f( \overline{X}^1_{\Gamma_{n}} ) \vert^{\rho}\vert \overline{X}_{\Gamma_{n}}  ]   \leqslant  &  C_f \gamma_{n+1}^{\rho/2}  \Tr [\sigma \sigma^{\ast} (\overline{X}_{\Gamma_n} )  ]^{\rho/2} \\
 & + C_f \gamma_{n+1}^{\rho}  \sum_{i=1}^d \sum_{l=1}^d \vert \partial_{x_l} \sigma_i ( x )  \sigma_{l,i}(x)  \vert^{\rho} . \nonumber
\end{align}
with $\DomA_0 =\mathcal{C}^2_K (\mathbb{R}^d )$. In other words, we have $\mathcal{GC}_{Q}(\DomA_0,  g_{\sigma},\rho,\epsilon_{\mathcal{I}}) $ (see (\ref{hyp:incr_X_Lyapunov})) with $g_{\sigma}=\Tr[ \sigma \sigma^{\ast} ]^{ \rho/2}+ \sum_{i=1}^d \sum_{l=1}^d \vert \partial_{x_l} \sigma_i  \sigma_{l,i}  \vert^{\rho} $ and $\epsilon_{\mathcal{I}}(\gamma)=\gamma^{\rho/2}$ for every $\gamma \in \mathbb{R}_+$. \\


Moreover, if (\ref{hyp:Lyapunov_control_milstein}) and $\mathfrak{B}(\phi)$ (see (\ref{hyp:controle_coefficients_milstein})) hold and 


\begin{align}
\label{hyp:control_step_weight_pol_milstein}
\mathcal{S} \mathcal{W}_{pol}(p,a,s,\rho) \qquad a\,p \rho/s \leqslant p+a-1.
\end{align}

Then, for every $n \in \mathbb{N}$, we have
\begin{align}
\label{eq:incr_lyapunov_X_milstein_f_tens}
 \mathbb{E}[\vert  V^{p/s}(\overline{X}_{\Gamma_{n+1}})-V^{p/s}(\overline{X}_{\Gamma_{n}}) \vert^{\rho}\vert \overline{X}_{\Gamma_{n}}]  
\leqslant  C \gamma_{n+1}^{\rho/2} V^{p+a-1}(\overline{X}_{\Gamma_{n}}).
\end{align}
In other words, we have $\mathcal{GC}_{Q}(V^{p/s},V^{p+a-1},\rho,\epsilon_{\mathcal{I}}) $ (see (\ref{hyp:incr_X_Lyapunov})) with and $\epsilon_{\mathcal{I}}(\gamma)=\gamma^{\rho/2}$ for every $\gamma \in \mathbb{R}_+$.

\end{lemme}

\begin{proof}
We begin by noticing that
\begin{align*}
 \vert \overline{X}_{\Gamma_{n+1}} -\overline{X}^1_{\Gamma_n} \vert\leqslant C  \gamma_{n+1}^{1/2} \Tr [\sigma \sigma^{\ast} (\overline{X}_{\Gamma_n} )  ] ^{1/2} \vert U_{n+1} \vert +C\gamma_{n+1} \vert \sum_{i=1}^d \sum_{l=1}^d \vert \partial_{x_l} \sigma_i ( \overline{X}_{\Gamma_n} )  \sigma_{l,i}(\overline{X}_{\Gamma_n})  \vert^2  \vert^{1/2} \vert \mathcal{W}_{n+1} \vert
\end{align*} 

Let $f \in \DomA$. Then $f$ is Lipschitz and the previous inequality gives (\ref{eq:incr_lyapunov_X_milstein_f_DomA}).\\

We focus now on the proof of (\ref{eq:incr_lyapunov_X_milstein_f_tens}).
We first notice that $\mathfrak{B}(\phi)$ (see (\ref{hyp:controle_coefficients_milstein}))implies that for any $n \in \mathbb{N}$,
\begin{align*}
 \vert \overline{X}_{\Gamma_{n+1}} -\overline{X}_{\Gamma_n} \vert\leqslant C  \gamma_{n+1}^{1/2}\sqrt{\phi \circ V (\overline{X}_{\Gamma_n})}(1+\vert U_{n+1}\vert +\vert \mathcal{W}_{n+1} \vert \vert )
\end{align*}  

 \paragraph{Case $2p\leqslant s$.} 
 
 We notice that  $V^{p/s}$ is $\alpha$-H\"older for any $\alpha\! \in [2p/s,1]$ (see Lemma 3. in \cite{Panloup_2008}) and then $V^{p/s}$ is $2p/s$-H\"older. We deduce that
\begin{align*}
\mathbb{E} [ \vert V^{p/s}( \overline{X}_{\Gamma_{n+1}} )-& V^{p/s}(\overline{X}_{\Gamma_n}) \vert^{\rho} \vert \overline{X}_{q,\Gamma_n} ] \leqslant  C[V^{p/s}]^{\rho}_{2p/s} \gamma_{n+1}^{ \rho/2}  V^{a \rho /2}(\overline{X}_{q,\Gamma_{n}}).
\end{align*}
In order to obtain (\ref{eq:incr_lyapunov_X_milstein_f_tens}), it remains to use $a\,p \rho /s \leqslant a+p-1$.\\

%
%

\paragraph{Case $2p\geqslant s$.}  Using  the following inequality
 \begin{align}
 \label{eq:puisssance_sup_1}
\forall u,v \in \mathbb{R}_+,\forall \alpha \geqslant 1, \qquad \vert u^{\alpha} -v^{\alpha} \vert \leqslant &  \alpha 2^{\alpha-1} ( v^{\alpha-1} \vert u -v \vert + \vert u -v \vert ^{\alpha} ),
 \end{align}
with $\alpha=2p/s$, and since $\sqrt{V}$ is Lipschitz, we have

\begin{align*}
\big \vert V^{p/s}( \overline{X}_{\Gamma_{n+1}} )-V^{p/s}(\overline{X}_{\Gamma_n}) \big \vert \leqslant &  2^{2p/s}p/s ( V^{p/s-1/2}(\overline{X}_{\Gamma_n}) \vert \sqrt{V}( \overline{X}_{\Gamma_{n+1}} )  - \sqrt{V} (\overline{X}_{\Gamma_n} ) \vert  \\
 & + \vert  \sqrt{V}( \overline{X}_{\Gamma_{n+1}} )  - \sqrt{V}(\overline{X}_{\Gamma_n} ) \vert^{2p/s}  ) \\
\leqslant &  2^{2p/s}p/s ( [ \sqrt{V}]_1 V^{p/s-1/2}(\overline{X}_{\Gamma_n} ) \vert  \overline{X}_{\Gamma_{n+1}}- \overline{X}_{\Gamma_n} \vert \\
& + [ \sqrt{V}]_1^{2p/s}  \vert   \overline{X}_{\Gamma_{n+1}} -\overline{X}_{\Gamma_n}\vert^{2p/s}   ).
\end{align*}
In order to obtain (\ref{eq:incr_lyapunov_X_milstein_f_tens}), it remains to use the assumptions $\mathfrak{B}(\phi)$ (see (\ref{hyp:controle_coefficients_milstein}))  and then $a\,p\rho/s \leqslant p+a-1$.

\end{proof}

\paragraph{Exponential case}

\begin{lemme}
\label{lemme:incr_lyapunov_X_milstein_expo}
 Let $p\in [0,1/2], \lambda \geqslant 0$, $s \geqslant 1$, $\rho \in [1,2]$ and,let $\phi:[v_{\ast},\infty )\to \mathbb{R}_+$ be a continuous function such that $C_{\phi}:= \sup_{y \in [v_{\ast},\infty )}\phi(y)/y< \infty$ and $\liminf\limits_{y \to +\infty} \phi(y)=+\infty$, let $\psi(y)=\exp ( \lambda y^p )$, $y\in \mathbb{R}_+$.  Assume that (\ref{hyp:Lyapunov_control_milstein}) and $\mathfrak{B}(\phi)$ (see (\ref{hyp:controle_coefficients_milstein})) hold, that $\rho<s$, and that

\begin{align}
\label{hyp:incr_lyapunov_X_milstein_expo}
 \forall \tilde{\lambda} \leqslant \lambda, \exists C\geqslant 0, \forall n \in \mathbb{N}, \qquad\mathbb{E}[  \exp( \tilde{\lambda} V^p(\overline{X}_{\Gamma_{n+1}}))  \vert \overline{X}_{\Gamma_n}] \leqslant & C \exp(\tilde{\lambda} V^p(\overline{X}_{\Gamma_{n}}))   .
\end{align}

Then, for every $n \in \mathbb{N}$, we have
\begin{align}
\label{eq:incr_lyapunov_X_milstein_f_tens_expo}
\mathbb{E}[\vert \exp( & \lambda/s V^p(\overline{X}_{\Gamma_{n+1}}))- \exp(\lambda/s V^p(\overline{X}_{\Gamma_{n}}))  \vert^{\rho} \vert \overline{X}_{\Gamma_{n}}] 
\leqslant  C \gamma_{n+1}^{p \rho} \frac{\phi \circ V (\overline{X}_{\Gamma_{n}}) }{V(\overline{X}_{\Gamma_{n}}) }\exp(\lambda V^p(\overline{X}_{\Gamma_{n}})) ,
\end{align}
In other words, we have $\mathcal{GC}_{Q}(\exp(\lambda/s  V^{p}),V^{-1}  \phi \circ V \exp(\lambda  V^{p}),\rho,\epsilon_{\mathcal{I}}) $ (see (\ref{hyp:incr_X_Lyapunov})) and $\epsilon_{\mathcal{I}}(\gamma)=\gamma^{p \rho}$ for every $\gamma \in \mathbb{R}_+$.

\end{lemme}

%
%
%
\begin{proof}
When $p=0$ the result is straightforward. Before we prove the result, we notice that $\mathfrak{B}(\phi)$ (see (\ref{hyp:controle_coefficients_milstein})) implies that for any $n \in \mathbb{N}$,,
\begin{align*}
 \vert \overline{X}_{\Gamma_{n+1}} -\overline{X}_{\Gamma_n} \vert\leqslant C  \gamma_n^{1/2}  \sqrt{\phi \circ V (\overline{X}_{\Gamma_n})}(1+\vert U_{n+1} \vert^2 + \vert \mathcal{W}_{n+1} \vert^2 ).
\end{align*} 

Let $x,y \in \mathbb{R}^d$. From Taylor expansion at order one, we derive,
\begin{align*}
\big\vert\exp(\lambda/s V^p(y))-\exp(\lambda/s V^p(x)) \big\vert \leqslant  \frac{\lambda}{s}\big(\exp(\lambda/s V^p(y))+\exp(\lambda/s V^p(x))\big) \big\vert  V^p(y) -V^p(x) \big\vert .
\end{align*}

Since $p \leqslant 1/2$, we notice that the function $V^p$ is $\alpha$-H\"older for every $\alpha \in [2p,1]$ (see Lemma 3. in \cite{Panloup_2008}) and then $V^{p}$ is $2p$-H\"older that is
\begin{align*}
\vert V^{p}( y )-V^{p}(x)\vert \leqslant & [\sqrt{V}]_{2p}   \vert y-x \vert^{2p}.
\end{align*}
Combining both above inequalities, we derive

\begin{align*}
\mathbb{E}\big[\vert \exp(  \lambda/s V^p(\overline{X}_{\Gamma_{n+1}}))- &\exp(\lambda/s V^p(\overline{X}_{\Gamma_{n}}))  \vert^{\rho} \vert \overline{X}_{\Gamma_n} \big] \\
\leqslant & C \exp(\lambda \rho /s V^p(\overline{X}_{\Gamma_{n}}))  \mathbb{E}\big[ \vert \overline{X}_{\Gamma_{n+1}} -\overline{X}_{\Gamma_n} \vert^{2p\rho} \vert \overline{X}_{\Gamma_n}\big] \\
&+ C \mathbb{E}\Big[  \exp(\lambda \rho /s V^p(\overline{X}_{\Gamma_{n+1}})) \vert \overline{X}_{\Gamma_{n+1}} -\overline{X}_{\Gamma_n} \vert^{2p\rho} \vert \overline{X}_{\Gamma_n}\Big] \\
\leqslant &  C \exp(\lambda \rho /s V^p(\overline{X}_{\Gamma_{n}}))  \mathbb{E}\big[ \vert \overline{X}_{\Gamma_{n+1}} -\overline{X}_{\Gamma_n} \vert^{2p\rho} \vert \overline{X}_{\Gamma_n} \big] \\
&+ C  \mathbb{E}\big[  \exp(\lambda \rho \theta /s V^p(\overline{X}_{\Gamma_{n+1}}))  \vert \overline{X}_{\Gamma_n}]^{1/\theta} \mathbb{E}[ \vert \overline{X}_{\Gamma_{n+1}} -\overline{X}_{\Gamma_n} \vert^{2p \rho \theta/(\theta-1)} \vert \overline{X}_{\Gamma_n} \big]^{(\theta-1)/\theta},
\end{align*}
for every $\theta>1$. From (\ref{hyp:incr_lyapunov_X_milstein_expo}) and since $\rho<s$, we take $\theta \in (1,\rho/s]$ and we get
\begin{align*}
\mathbb{E}\big[  \exp(\lambda \rho \theta /s V^p(\overline{X}_{\Gamma_{n+1}})  \vert \overline{X}_{\Gamma_n}\big] \leqslant & C \exp( \lambda \theta \rho  /s V^p(\overline{X}_{\Gamma_{n}}))   .
\end{align*}

Rearranging the terms and  since $\rho < s$, we conclude from $\mathfrak{B}(\phi)$ (see (\ref{hyp:controle_coefficients_milstein})) that
\begin{align*}
\mathbb{E}[\vert \exp(   \lambda/s V^p(\overline{X}_{\Gamma_{n+1}}))-  \exp(\lambda/s V^p(\overline{X}_{\Gamma_{n+1}}))  \vert^{\rho} \vert \overline{X}_{\Gamma_{n+1}}]  \leqslant & C  \gamma_n^{p \rho} \frac{ \phi \circ V  (\overline{X}_{\Gamma_{n}} )}{V(\overline{X}_{\Gamma_{n}})}  \exp(\lambda V^p(\overline{X}_{\Gamma_{n}}))   ,
\end{align*}
and the proof is completed.

\end{proof}

\subsubsection{Proof of Theorem \ref{th:cv_was_milstein}}

This result follows from Theorem \ref{th:tightness} and Theorem \ref{th:identification_limit}. The proof consists in showing that the assumptions from those theorems are satisfied.

\paragraph{Step 1. Mean reverting recursive control}
First, we show that $\mathcal{RC}_{Q,V}(\psi_p,\phi,p\tilde{\alpha},p\beta)$ and $\mathcal{RC}_{Q,V}(\psi_{1},\phi,\tilde{\alpha},\beta)$ (see (\ref{hyp:incr_sg_Lyapunov})) is satisfied for every $\tilde{\alpha} \in (0,\alpha)$.\\

 Since (\ref{hyp:Lyapunov_control_milstein}), $\mathfrak{B}(\phi)$ (see (\ref{hyp:controle_coefficients_milstein})) and $\mathcal{R}_p(\alpha,\beta,\phi,V)$ (see (\ref{hyp:recursive_control_param_milstein})) hold, it follows from Proposition \ref{prop:recursive_control_milstein} that $\mathcal{RC}_{Q,V}(\psi_p,\phi,p\tilde{\alpha},p\beta)$ (see (\ref{hyp:incr_sg_Lyapunov})) is satisfied for every $\tilde{\alpha} \in (0,\alpha)$ since $\liminf_{y \to + \infty} \phi(y) > \beta / \tilde{\alpha}$. Moreover let us notice that for every $p \leqslant 1$ then $\mathcal{R}_p(\alpha,\beta, \phi, V)$ (see (\ref{hyp:recursive_control_param_milstein})) is similar to $\mathcal{R}_1(\alpha,\beta, \phi, V)$ and then $\mathcal{RC}_{Q,V}(\psi_1,\phi,\tilde{\alpha},\beta)$ (see (\ref{hyp:incr_sg_Lyapunov})) is satisfied for every $\tilde{\alpha} \in (0,\alpha)$

\paragraph{Step 2. Step weight assumption} 
Now, we show that $\mathcal{S}\mathcal{W}_{\mathcal{I}, \gamma,\eta}(V^{p \vee 1+a-1} ,\rho,\epsilon_{\mathcal{I}}) $ (see (\ref{hyp:step_weight_I_gen_chow})) and $\mathcal{S}\mathcal{W}_{\mathcal{II},\gamma,\eta}(V^{p \vee 1 +a-1}) $ (see (\ref{hyp:step_weight_I_gen_tens})) hold. \\

First we noticel that from Step1. the assumption $\mathcal{RC}_{Q,V}(\psi_{p \vee 1},\phi,(p \vee 1)\tilde{\alpha},(p \vee 1)\beta)$ (see (\ref{hyp:incr_sg_Lyapunov})) is satisfied for every $\tilde{\alpha} \in (0,\alpha)$. Then, using $\mathcal{S}\mathcal{W}_{\mathcal{I}, \gamma,\eta}(\rho, \epsilon_{\mathcal{I}})$ (see (\ref{hyp:step_weight_I})) with Lemma \ref{lemme:mom_V} gives $\mathcal{S}\mathcal{W}_{\mathcal{I}, \gamma,\eta}( V^{p \vee 1 +a-1},\rho,\epsilon_{\mathcal{I}}) $ (see (\ref{hyp:step_weight_I_gen_chow})). Similarly, $\mathcal{S}\mathcal{W}_{\mathcal{II},\gamma,\eta}(V^{p \vee 1+a-1}) $ (see (\ref{hyp:step_weight_I_gen_tens}) follows from $\mathcal{S}\mathcal{W}_{\mathcal{II},\gamma,\eta} $ (see (\ref{hyp:step_weight_II})) and  Lemma \ref{lemme:mom_V}. 

\paragraph{Step 3. Growth control assumption}
Now, we prove $\mathcal{GC}_{Q}(F,V^{p \vee 1+a-1},\rho,\epsilon_{\mathcal{I}}) $ (see (\ref{hyp:incr_X_Lyapunov})) for $F= \DomA_0$ and $F=\{V^{p/s}\}$ .\\

This is a consequence of Lemma \ref{lemme:incr_lyapunov_X_milstein}. We notice that $\rho/2\leqslant 1$. Consequently $M_{(\rho/2)\vee(p\rho /s) }(U)$ (see (\ref{hyp:moment_ordre_p_va_schema_milstein})) and $M_{(\rho/2) \vee (p \rho /s)}(\mathcal{W})$ (see (\ref{hyp:moment_ordre_p_va_schema_milstein})) hold.  Now, we notice that from $\mathfrak{B}(\phi)$ (see (\ref{hyp:controle_coefficients_milstein})), we have$\Tr[ \sigma \sigma^{\ast} ]^{\rho/2} +  \sum_{i=1}^d \sum_{l=1}^d \vert \partial_{x_l} \sigma_i ( x )  \sigma_{l,i}  \vert^{\rho} \leqslant CV^{ \rho a/2} $ with $a \rho/2 \leqslant p+a-1$ since $\mathcal{S} \mathcal{W}_{pol}(p,a,s,\rho)$ (see (\ref{hyp:control_step_weight_pol_milstein})) holds. Then Lemma \ref{lemme:incr_lyapunov_X_milstein} implies that for $F= \DomA_0$ and $F=\{V^{p/s}\}$, then $\mathcal{GC}_{Q}(F,V^{p \vee 1 + a -1},\rho,\epsilon_{\mathcal{I}}) $ (see (\ref{hyp:incr_X_Lyapunov})) holds

\paragraph{Step 4. Conclusion}
\begin{enumerate}[label=\textbf{\roman*.}]
\item
The first part of Theorem \ref{th:cv_was_milstein} (see (\ref{eq:tightness_milstein})) is a consequence of Theorem \ref{th:tightness}. Let us observe that assumptions from Theorem \ref{th:tightness} indeed hold. \\

On the one hand, we observe that from Step 2. and Step 3. the assumptions $\mathcal{GC}_{Q}(V^{p/s},V^{p \vee 1 +a-1},\rho,\epsilon_{\mathcal{I}}) $ (see (\ref{hyp:incr_X_Lyapunov})), $\mathcal{S}\mathcal{W}_{\mathcal{I}, \gamma,\eta}( V^{p \vee 1 +a-1},\rho,\epsilon_{\mathcal{I}}) $ (see (\ref{hyp:step_weight_I_gen_chow})) and $\mathcal{S}\mathcal{W}_{\mathcal{II},\gamma,\eta}(V^{p \vee 1 +a-1}) $ (see (\ref{hyp:step_weight_I_gen_tens})) hold which are the hypothesis from Theorem \ref{th:tightness} point \ref{th:tightness_point_A} with $g=V^{p \vee 1 +a-1}$.\\

On the other hand, form Step 1. the assumption$\mathcal{RC}_{Q,V}(\psi_p,\phi,p\tilde{\alpha},p\beta)$ (see (\ref{hyp:incr_sg_Lyapunov})) is satisfied for every $\tilde{\alpha} \in (0,\alpha)$. Moreover, since $\mbox{L}_{V}$ (see (\ref{hyp:Lyapunov})) holds and  that $p/s+a-1>0$, then the hypothesis  from Theorem \ref{th:tightness} point \ref{th:tightness_point_B} are satisfied. \\

We thus conclude from Theorem \ref{th:tightness} that $(\nu_n^{\eta})_{n \in \mathbb{N}^{\ast}}$ is $\mathbb{P}-a.s.$ tight and (\ref{eq:tightness_milstein}) holds which concludes the proof of the first part of Theorem \ref{th:cv_was_milstein}. \\

\item Let us now prove the second part of Theorem \ref{th:cv_was_milstein} (see (\ref{eq:cv_was_milstein})) which is a consequence of Theorem \ref{th:identification_limit}.\\

On the one hand,we observe that from Step 2. and Step 3. the assumptions $\mathcal{GC}_{Q}(\DomA_0,V^{p \vee 1 +a-1},\rho,\epsilon_{\mathcal{I}}) $ (see (\ref{hyp:incr_X_Lyapunov})) and $\mathcal{S}\mathcal{W}_{\mathcal{I}, \gamma,\eta}( V^{p \vee 1 +a-1},\rho,\epsilon_{\mathcal{I}}) $ (see (\ref{hyp:step_weight_I_gen_chow})) hold which are the hypothesis from Theorem \ref{th:identification_limit} point \ref{th:identification_limit_A} with $g=V^{p \vee 1 +a-1}$.\\

On the other hand, since $b$, $\sigma$ and $\sum_{i,j,l=1}^d  \vert  \partial_{x_l} \sigma_i  \sigma_{l,j} \vert  $ have sublinear growth and that $g_{\sigma}\leqslant C  V^{p/s+a-1}$, with $g_{\sigma}=\Tr[ \sigma \sigma^{\ast} ]+ \sum_{i,j,l=1}^d  \vert  \partial_{x_l} \sigma_i  \sigma_{l,j} \vert  $, so that $\mathbb{P} \mbox{-a.s.} \;\sup_{n \in \mathbb{N}^{\ast}} \nu_n^{\eta}( g_{\sigma}) < + \infty $,  it follows from Proposition \ref{prop:milstein_infinitesimal_approx} that $\mathcal{E}(\widetilde{A},A,\DomA_0) $ (see (\ref{hyp:erreur_tems_cours_fonction_test_reg})) is satisfied. Then, the hypothesis from Theorem \ref{th:identification_limit} point \ref{th:identification_limit_B} hold and (\ref{eq:cv_was_milstein}) follows from (\ref{eq:test_function_gen_cv}).

\end{enumerate}

\subsubsection{Proof of Theorem \ref{th:cv_exp_milstein}}

This result follows from Theorem \ref{th:tightness} and Theorem \ref{th:identification_limit}. The proof consists in showing that the assumptions from those theorems are satisfied.

\paragraph{Step 1. Mean reverting recursive control}

First, we show that for every $\tilde{\alpha} \in (0,\alpha)$, there exists $\tilde{\beta} \in \mathbb{R}_+$ such that $\mathcal{RC}_{Q,V}(\tilde{\psi},\phi,p\tilde{\alpha},p\tilde{\beta})$ (see (\ref{hyp:incr_sg_Lyapunov})) is satisfied for every function $\tilde{\psi}: [v_{\ast},\infty) \to \mathbb{R}_+$ such that $\tilde{\psi}(y)=  \exp( \tilde{\lambda} V^p)$ with $\tilde{\lambda} \leqslant \lambda$. Notice that this property and the fact that $\phi$ has sublinear growth imply (\ref{hyp:incr_lyapunov_X_milstein_expo}).\\

We begin by noticing that $\mathcal{R}_{p, \lambda}(\alpha,\beta,\phi,V)$ (see (\ref{hyp:recursive_control_param_milstein_expo})) implies $\mathcal{R}_{p,\tilde{ \lambda}}(\alpha,\beta,\phi,V)$ for every $\tilde{\lambda} \leqslant \lambda$. Since (\ref{hyp:Lyapunov_control_milstein}), $\mathfrak{B}(\phi)$ (see (\ref{hyp:controle_coefficients_milstein})), $\mathcal{R}_{p, \lambda}(\alpha,\beta,\phi,V)$  (see (\ref{hyp:recursive_control_param_milstein_expo})) and (\ref{hyp:dom_recurs_milstein}) hold, it follows from Proposition \ref{prop:recursive_control_milstein_exp} with $\lim_{y \to +\infty} \phi(y)=+\infty$, that that for every $\tilde{\alpha} \in (0,\alpha)$, there exists $\tilde{\beta} \in \mathbb{R}_+$ such that $\mathcal{RC}_{Q,V}(\tilde{\psi},\phi,p\tilde{\alpha},p\tilde{\beta})$ (see (\ref{hyp:incr_sg_Lyapunov})) is satisfied for every function $\tilde{\psi}: [v_{\ast},\infty) \to \mathbb{R}_+$ such that $\tilde{\psi}(y)=  \exp( \tilde{\lambda} V^p)$ with $\tilde{\lambda} \leqslant \lambda$. 

\paragraph{Step 2. Step weight assumption} 
Now, we show that $\mathcal{S}\mathcal{W}_{\mathcal{I}, \gamma,\eta}(V^{-1}. \phi \circ V  .\exp(\lambda  V^{p}) ,\rho,\tilde{\epsilon}_{\mathcal{I}}) $, $\mathcal{S}\mathcal{W}_{\mathcal{I}, \gamma,\eta}(V^{-1}. \phi \circ V  .\exp(\lambda  V^{p}) ,\rho,\epsilon_{\mathcal{I}}) $ (see (\ref{hyp:step_weight_I_gen_chow})) and $\mathcal{S}\mathcal{W}_{\mathcal{II},\gamma,\eta}(\exp(\lambda /s V^{p})) $ (see (\ref{hyp:step_weight_I_gen_tens})) hold. \\

First we recall that that there exists $\tilde{\alpha} \in (0,\alpha)$ and $\tilde{\beta} \in \mathbb{R}_+$ such that $\mathcal{RC}_{Q,V}(\psi,\phi,\tilde{\alpha},\tilde{\beta})$ (see (\ref{hyp:incr_sg_Lyapunov})) is satisfied. Then, using $\mathcal{S}\mathcal{W}_{\mathcal{I}, \gamma,\eta}(\rho, \tilde{\epsilon}_{\mathcal{I}})$ and $\mathcal{S}\mathcal{W}_{\mathcal{I}, \gamma,\eta}(\rho, \epsilon_{\mathcal{I}})$ (see (\ref{hyp:step_weight_I})) with Lemma \ref{lemme:mom_V} gives $\mathcal{S}\mathcal{W}_{\mathcal{I}, \gamma,\eta}( V^{-1}. \phi \circ V  .\exp(\lambda  V^{p}),\rho,\tilde{\epsilon}_{\mathcal{I}}) $ and $\mathcal{S}\mathcal{W}_{\mathcal{I}, \gamma,\eta}( V^{-1}. \phi \circ V  .\exp(\lambda  V^{p}),\rho,\epsilon_{\mathcal{I}}) $ (see (\ref{hyp:step_weight_I_gen_chow})). Similarly, $\mathcal{S}\mathcal{W}_{\mathcal{II},\gamma,\eta}(V^{-1}. \phi \circ V  .\exp(\lambda  V^{p})) $ (see (\ref{hyp:step_weight_I_gen_tens}) follows from $\mathcal{S}\mathcal{W}_{\mathcal{II},\gamma,\eta} $ (see (\ref{hyp:step_weight_II})) and  Lemma \ref{lemme:mom_V}. 

\paragraph{Step 3. Growth control assumption}
Now, we prove $\mathcal{GC}_{Q}(F,V^{-1}. \phi \circ V  .\exp(\lambda  V^{p}),\rho,\epsilon_{\mathcal{I}}) $ (see (\ref{hyp:incr_X_Lyapunov})) for $F= \DomA_0$ and $F=\{\exp(\lambda /s V^{p})\}$ .\\

This is a consequence of Lemma \ref{lemme:incr_lyapunov_X_milstein} and Lemma \ref{lemme:incr_lyapunov_X_milstein_expo}. We notice indeed that $\mathfrak{B}(\phi)$ (see (\ref{hyp:controle_coefficients_milstein})) gives $\Tr[ \sigma \sigma^{\ast} ]^{\rho/2}  \leqslant (\phi \circ V)^{ \rho} $. Moreover, we have already shown that (\ref{hyp:incr_lyapunov_X_milstein_expo}) is satisfied in Step 1. These observations combined with (\ref{eq:incr_lyapunov_X_milstein_f_tens_expo}) imply that $\mathcal{GC}_{Q}(\DomA_0,V^{-1}  \phi \circ V \exp(\lambda V^p) ,\rho,\epsilon_{\mathcal{I}}) $  and $\mathcal{GC}_{Q}(\exp(\lambda /s V^{p}),V^{-1}.  \phi \circ V . \exp(\lambda V^p) ,\rho,\tilde{\epsilon}_{\mathcal{I}}) $ (see (\ref{hyp:incr_X_Lyapunov})) hold.

\paragraph{Step 4. Conclusion}
\begin{enumerate}[label=\textbf{\roman*.}]
\item
The first part of Theorem \ref{th:cv_exp_milstein} (see (\ref{eq:tightness_milstein_expo})) is a consequence of Theorem \ref{th:tightness}. Let us observe that assumptions from Theorem \ref{th:tightness} indeed hold. \\

On the one hand, we observe that from Step 2. and Step 3. the assumptions $\mathcal{GC}_{Q}(\exp(\lambda /s V^{p}),V^{-1}  \phi \circ V \exp(\lambda V^p) ,\rho,\tilde{\epsilon}_{\mathcal{I}}) $ (see (\ref{hyp:incr_X_Lyapunov})), $\mathcal{S}\mathcal{W}_{\mathcal{I}, \gamma,\eta}( V^{-1}  \phi \circ V \exp(\lambda V^p) ,\rho,\tilde{\epsilon}_{\mathcal{I}}) $ (see (\ref{hyp:step_weight_I_gen_chow})) and $\mathcal{S}\mathcal{W}_{\mathcal{II},\gamma,\eta}( V^{-1}  \phi \circ V \exp(\lambda V^p)) $ (see (\ref{hyp:step_weight_I_gen_tens})) hold which are the hypothesis from Theorem \ref{th:tightness} point \ref{th:tightness_point_A} with $g=V^{-1}  \phi \circ V \exp(\lambda V^p)$.\\

On the other hand, form Step 1. for every $\tilde{\alpha} \in (0,\alpha)$, there exists $\tilde{\beta} \in \mathbb{R}_+$ such that $\mathcal{RC}_{Q,V}(\psi,\phi,p\tilde{\alpha},p\tilde{\beta})$ (see (\ref{hyp:incr_sg_Lyapunov})) is satisfied. Moreover, since $\mbox{L}_{V}$ (see (\ref{hyp:Lyapunov})) holds, then the hypothesis  from Theorem \ref{th:tightness} point \ref{th:tightness_point_B} are satisfied. \\

We thus conclude from Theorem \ref{th:tightness} that $(\nu_n^{\eta})_{n \in \mathbb{N}^{\ast}}$ is $\mathbb{P}-a.s.$ tight and (\ref{eq:tightness_milstein_expo}) holds
which concludes the proof of the first part of Theorem \ref{th:cv_exp_milstein}. \\

\item Let us now prove the second part of Theorem \ref{th:cv_exp_milstein} (see (\ref{eq:cv_expo_milstein})) which is a consequence of Theorem \ref{th:identification_limit}.\\

On the one hand,we observe that from Step 2. and Step 3. the assumptions $\mathcal{GC}_{Q}(\DomA_0,V^{-1}  \phi \circ V \exp(\lambda V^p),\rho,\epsilon_{\mathcal{I}}) $ (see (\ref{hyp:incr_X_Lyapunov})) and $\mathcal{S}\mathcal{W}_{\mathcal{I}, \gamma,\eta}( V^{-1}  \phi \circ V \exp(\lambda V^p),\rho,\epsilon_{\mathcal{I}}) $ (see (\ref{hyp:step_weight_I_gen_chow})) hold which are the hypothesis from Theorem \ref{th:identification_limit} point \ref{th:identification_limit_A} with $g=V^{-1}  \phi \circ V \exp(\lambda V^p)$.\\

On the other hand, since $b$, $\sigma$ and $\sum_{i,j,l=1}^d  \vert  \partial_{x_l} \sigma_i  \sigma_{l,j} \vert  $ have sublinear growth and that $g_{\sigma}\leqslant V^{-1}  \phi \circ V \exp(\lambda/s V^p)$, with $g_{\sigma}=\Tr[ \sigma \sigma^{\ast} ]+ \sum_{i,j,l=1}^d  \vert  \partial_{x_l} \sigma_i  \sigma_{l,j} \vert  $, so that $\mathbb{P} \mbox{-a.s.} \;\sup_{n \in \mathbb{N}^{\ast}} \nu_n^{\eta}( g_{\sigma} ) < + \infty $,  it follows from Proposition \ref{prop:milstein_infinitesimal_approx} that $\mathcal{E}(\widetilde{A},A,\DomA_0) $ (see (\ref{hyp:erreur_tems_cours_fonction_test_reg})) is satisfied. Then, the hypothesis from Theorem \ref{th:identification_limit} point \ref{th:identification_limit_B} hold and (\ref{eq:cv_expo_milstein}) follows from (\ref{eq:test_function_gen_cv}).

\end{enumerate}

\subsection{Application to censored jump processes}
In this section, applying results from Section \ref{section:convergence_inv_distrib_gnl}, we build invariant distributions for censored jump processes which are not necessarily Levy processes. Our approach extends the one made in \cite{Panloup_2008}, and inspired by \cite{Lamberton_Pages_2002}, for Levy processes in a weakly mean reverting setting, namely $\phi(y)=y^a$, $a \in (0,1]$ for every $y \in [v_{\ast},\infty)$. Like in \cite{Panloup_2008}, we consider polynomial test functions, $i.e.$ $\psi_p(y)=y^p$, with $ p \geqslant 0$ for every $y \in [v_{\ast},\infty)$.\\

\noindent Now, we present the censored jump process, its decreasing step Euler approximation and the hypothesis necessary to obtain the convergence of $(\nu^{\eta}_n)_{n \in\mathbb{N}^{\ast}}$.  We consider a Poisson point process $\mathfrak{p}$ with state space $(\hat{F};\mathcal{B}(\hat{F}))$ where $\hat{F} = F\times  \mathbb{R}_+$ with $F$ an open set. We refer to \cite{Ikeda_Watanabe_1989} for more details. We denote by $N$ the counting measure associated to $\mathfrak{p}$. We have $ N([0, t) \times A) = \# \{ 0 \leqslant   s < t; \mathfrak{p}_s \in A \}$ for $t \geqslant 0$ and $A \in \mathcal{B}(\hat{F})$. We assume that the associated intensity measure is given by $ \hat{N} (dt, dz, dv) = dt \times  \pi( dz) \times \mathds{1}_{[0, \infty)}(v)dv$ where $(z, v) \in  \hat{F} = F\times \mathbb{R}_+$ and $\pi$ is a positive measure with $\pi(F) \in \mathbb{R}_+ \cup \{+\infty\}$. We will use the notation $\widetilde{N}=N -\hat{N}$. We also consider a $d$-dimension Brownian motion $(W_t)_{t \geqslant 0}$ independent from $N$. We are interested in the strong solution - assumed to exist and to be unique - of the d dimensional stochastic equation
\begin{align*}
X_t= x+\int_0^t b(X_{s^-})ds + \int_0^t \int_{\hat{F}} c(z,X_{s^-}) \mathds{1}_{v \leqslant \zeta(z, X_{s^-}) } N(ds,dz,dv)  .
\end{align*}
where $b: \mathbb{R}^d \to \mathbb{R}^d$ and $c(z,.):\mathbb{R}^d \to \mathbb{R}^d$, $z \in F$ are locally bounded functions and $\zeta: F \times \mathbb{R}^d \to \mathbb{R}_+$ is bounded. The infinitesimal generator of this process reads
\begin{align}
\label{eq:PDMP_generator}
A f(x) =\langle b(x) , \nabla f(x) \rangle + \int_{F} \big(f(x+c(z,x))-f(x) \big) \zeta(z,x)\pi(dz).  
\end{align}
and its domain $\DomA$ contains $\DomA_0 =\mathcal{C}^2_K(\mathbb{R}^d)$. Notice that $\DomA_0 $ is dense in $\mathcal{C}_0(E)$. In this paper, we do not discuss existence or unicity of such processes. The main difference with Levy processes is that the intensity of jump $\zeta(x,z)\pi(dz)$ may depend on the position of the process. The studies concerning these processes were initiated in \cite{Fournier_2002} where the focus is made on the existence of an absolutely continuous (with respect to the Lebesgue measure) density. In the PhD thesis \cite{Rabiet_2015}, the author extends existence and uniqueness results for SDE with non zero Brownian component and establish ergodicity properties. Notice that our results can be easily extended to the case of a non null Brownian part using the same approach as the one we present now. Notice that in this case, we can recover the results from \cite{Panloup_2008} as a particular case of our study.  \\
 
\noindent We now introduce an Euler scheme for this process. Since $\pi(F)$ may take an infinite value, we introduce the family $(F_{\gamma})_{\gamma \geqslant 0}$, with $F_{\tilde{\gamma}} \subset F_{\gamma} \subset F_0= F$ for every $\tilde{\gamma} \geqslant \gamma \geqslant 0$ and such that $\cup_{\gamma >0} F_{\gamma}=F$. When $\pi(F)<+\infty$, we suppose that $F_{\gamma} = F$ for every $\gamma>0$. 
First, let $q>0$ and define $\tilde{b}_{q}(x)=b(x) +\kappa_{q}(x) $ with
\begin{align}
\label{hyp:recursive_control_param_saut_terme_ordre_un}
\forall x \in \mathbb{R}^d,  \quad  \kappa_{q}(x)= \big( \mathds{1}_{q \in (1/2,1]} +\mathds{1}_{\pi(F)=+\infty}\mathds{1}_{q \in (1,+\infty)} \big)\int_{F} c(z,x ) \zeta (z,x ) \pi(dz) .
\end{align}
Now, for $x \in \mathbb{R}^d$ and $\gamma \geqslant 0$, we introduce the following quantities - supposed to be well defined \textit{càdlàg} processes: For every $t \geqslant 0$,
\begin{align*}
  M^{\gamma}_t(x):=&\int_0^{t} \int_{\hat{F}} c(z,x) \mathds{1}_{v \leqslant \zeta(z,x) } \mathds{1}_{F_{\gamma}}(z) N(ds,dz,dv),  \quad \mbox{and the local martingale} \\
   \widetilde{M}^{\gamma}_t(x):=&\int_0^{t} \int_{\hat{F}} c(z,x) \mathds{1}_{v \leqslant \zeta(z,x) } \mathds{1}_{F_{\gamma}}(z)  \widetilde{N}(ds,dz,dv). \nonumber
 \end{align*}
%
%
%
%
%
%
%
%
%
%
Moreover, for every $\epsilon>0$, we assume that
\begin{align}
\label{hyp:cv_zero_sublinear_jump}
&\mathbb{P}(d\omega)-a.s. \quad
   \left\{
      \begin{aligned}
        & \forall K \in \mathcal{K}_{\mathbb{R}^d}, \quad  \lim_{\gamma \to 0^+} \sup_{x \in K} M^{\gamma,q}_{\gamma}(x,\omega) =0 , \\
        &\exists \gamma_{0,\epsilon}(\omega)>0, \forall \gamma \in (0,\gamma_{0,\epsilon}(\omega) ],\forall x \in \mathbb{R}^d, \sup_{t \in [0,\gamma] } \vert M^{\gamma,q}_{t}(x,\omega) \vert \leqslant \epsilon(1+ \vert x \vert) .
      \end{aligned}
    \right.
\end{align}
with $M^{\gamma,q}_{t}(x)= \mathds{1}_{q \in (0,1/2]} M^{\gamma}_{t} (x) +\mathds{1}_{q \in (1/2,+\infty)}( \mathds{1}_{\pi(F) =+ \infty} \widetilde{M}^{\gamma}_{t} (x)+\mathds{1}_{\pi(F) <+ \infty} M^{\gamma}_{t} (x))$, $t\geqslant 0$.
%
%
\begin{remark}
Assume that there exists $c_0:F \to \mathbb{R}_+$ such that for every $x \in \mathbb{R}^d$ and $z\in F$, $\vert c (z,x) \vert \leqslant C(1+\vert x \vert ) c_{0}(z)$ with $\int_F c_{0}(z)  \pi(dz) < + \infty$. Since $\vert M^{\gamma}_{\gamma}(x) \vert \leqslant \overline{M}_{\gamma}(x) := C(1+\vert x \vert ) \int_0^{\gamma} \int_{F} c_{0}(z) N(ds,dz,[0,\Vert \xi \Vert_{\infty}] )$ with $(\overline{M}_{\gamma}(x) )_{\gamma \geqslant 0}$ a \textit{càdlàg} process starting from zero, then (\ref{hyp:cv_zero_sublinear_jump}) holds when $M^{\gamma,q}_{t}(x)=M^{\gamma}_{t}(x)$, $t \geqslant 0$.  Moreover, since $\vert \tilde{M}^{\gamma}_{\gamma}(x) \vert \leqslant \overline{M}_{\gamma}(x) + C(1+\vert x \vert )\Vert \xi \Vert_{\infty} \gamma \int_{F} c_{0}(z) \pi(dz)$, then (\ref{hyp:cv_zero_sublinear_jump}) holds when $M^{\gamma,q}_{t}(x)=\tilde{M}^{\gamma}_{t}(x)$, $t \geqslant 0$. 
\end{remark}
We now introduce the sequences of independent random variables $(M^n_{\gamma_n}(x_n))_{n \in \mathbb{N}^{\ast}}$, (respectively ($\widetilde{M}^n_{\gamma_n}(x_n))_{n \in \mathbb{N}^{\ast}}$), $(x_n)_{n \in \mathbb{N}^{\ast}}\in( \mathbb{R}^d)^{\otimes \mathbb{N}^{\ast}}$ with $M^n_{\gamma_n}(x_n)$ (resp. $\widetilde{M}^n_{\gamma_n}(x_n))$ distributed under the the same law as $M^{\gamma_n}_{\gamma_n}(x_n)$ (resp. $\widetilde{M}^{\gamma_n}_{\gamma_n}(x_n)$). For every $n\in \mathbb{N}$, we define the Euler scheme by
\begin{align}
\label{eq:PDMP_bounded_Poiss_approx}
\overline{X}_{q,\Gamma_{n+1}}  =\overline{X}_{q,\Gamma_n} + \gamma_{n+1} \tilde{b}_{q}(\overline{X}_{q,\Gamma_{n}}) +&  (\mathds{1}_{q \in (0,1/2]}+\mathds{1}_{\pi(F)<+\infty} \mathds{1}_{q \in (1,+\infty)} ) M^{n+1}_{\gamma_{n+1}}(\overline{X}_{q,\Gamma_n})  \\
+&  (\mathds{1}_{q \in (1/2,1]}+\mathds{1}_{\pi(F)=+\infty} \mathds{1}_{q \in (1,+\infty)} )  \widetilde{M}^{n+1}_{\gamma_{n+1}}(\overline{X}_{q,\Gamma_n}) .\nonumber
\end{align}
%
%
%
%
We denote by $\widetilde{A}_{q}:=(\widetilde{A}_{q,\gamma} )_{ \gamma >0}$ the pseudo-generator and $(\nu_n^{\eta,q})_{n \in \mathbb{N}^{\ast}}$ the sequence of empirical distributions, of $(\overline{X}_{q,t})_{t \geqslant 0}$  respectively defined as in (\ref{eq:def_A_tilde}) and as in (\ref{eq:def_weight_emp_meas}) with $(\overline{X}_{t})_{t \geqslant 0}$ replaced by $(\overline{X}_{q,t})_{t \geqslant 0}$.
%
%
%
%
\begin{remark}
When $\pi(F)=+\infty$, we can assume that the family $(F_{\gamma})_{\gamma >0}$ satisfies $\pi(F_{\gamma})<+\infty$ for every $\gamma>0$. In this case we can simulate the Euler genuine scheme in the following way: Let $(J^{\gamma}_t)_{t \geqslant 0}$ be the Poisson process with intensity $ \Vert \zeta \Vert_{\infty} \pi(F_{\gamma})$. We introduce the sequences of independent random variables (independent from $J^{\gamma}$)
\begin{align*}
Z^n_k \sim  \pi(F_{\gamma_{n}})^{-1}   \mathds{1}_{F_{\gamma_{n}}}(z)\pi(dz), \qquad \mbox{and}  \qquad   V_k \sim \Vert \zeta \Vert_{\infty}^{-1} \mathds{1}_{[0, \Vert \zeta \Vert_{\infty}]}(v) dv.
\end{align*}
For every $n \in \mathbb{N}$ and every $t \in [\Gamma_n, \Gamma_{n+1}]$, (\ref{eq:PDMP_bounded_Poiss_approx}) can be rewritten (in the continuous case) as the Euler genuine scheme:
\begin{align*}
\overline{X}_{q,t}  = \overline{X}_{q,\Gamma_n} +& (t-\Gamma_{n}) \Big( b(\overline{X}_{q,\Gamma_{n}}) + \mathds{1}_{q \in (1/2,+\infty)} \int_{F \setminus F_{\gamma_{n+1}}} c(z,\overline{X}_{q,\Gamma_n} ) \zeta (z,\overline{X}_{q,\Gamma_n} ) \pi(dz) \Big). \\
+& \sum_{k=1}^{J^{\gamma_{n+1}}_{t}} c(Z^{n+1}_k,\overline{X}_{q,\Gamma_n}) \mathds{1}_{V_k \leqslant \zeta(Z^{n+1}_k, \overline{X}_{q,\Gamma_n}) } .\nonumber
\end{align*}
Notice that when $\pi(F)<+\infty$, since $F_{\gamma}=F$ for every $\gamma>0$, we can use this simulation method.
\end{remark}
In order to simplify the writing, we will use the notations:
\begin{align}
\label{def:incr_pcmp_euler}
\Delta \overline{X}^{1}_{q,{n+1}} = &  \gamma_{n+1}\tilde{b}_{q}(\overline{X}_{q,\Gamma_{n}})  , \\
\Delta \overline{X}^{2}_{q,{n+1}} =  & (\mathds{1}_{q \in (0,1/2]}+\mathds{1}_{\pi(F)<+\infty}\mathds{1}_{q \in (1,+\infty)}) M^{n+1}_{\gamma_{n+1}}(\overline{X}_{q,\Gamma_n})  +  ( \mathds{1}_{q \in (1/2,1]}+\mathds{1}_{\pi(F)=+\infty}\mathds{1}_{q \in (1,+\infty)}) \widetilde{M}^{n+1}_{\gamma_{n+1}}(\overline{X}_{q,\Gamma_n}) ,\nonumber 
\end{align}
%
%
%
and $\overline{X}^{i}_{q,\Gamma_{n+1}}=\overline{X}_{q,\Gamma_{n}}+\sum_{k=1}^i \Delta \overline{X}^{k}_{q,{n+1}}$.\\
Now we introduce some hypothesis concerning the parameters. We begin with the jump component. Let $p \geqslant 0$. In the sequel, we will denote
\begin{align*}
\tau_{p,\gamma} (x)  :=   \int_{F_{\gamma}} \vert c(z,x ) \vert^{2p}\zeta(z,x)  \pi(dz)   \quad \mbox{and} \quad  \tau_p  (x)  := \tau_{p,0} (x),
\end{align*}
 Assume that the following finiteness hypothesis holds
\begin{align}
\label{hyp:jump_component_ordre_p}
\mathcal{H}^p \quad \equiv \qquad \forall x \in \mathbb{R}^d,  \quad  \tau_p  (x) < + \infty,
\end{align}
and that
\begin{align}
\label{hyp:struc_param_saut_x_infini}
\forall z \in F, \qquad \limsup_{\vert x \vert \to + \infty} \frac{ \vert c(z,x) \vert}{\vert x \vert }<1.
\end{align}
 Finally, assume the existence of a Lyapunov function $V: \mathbb{R}^d  \to [v_{\ast}, \infty)$, $v_{\ast}>0$, which satisfies $\mbox{L}_V$ (see (\ref{hyp:Lyapunov})) with $E=\mathbb{R}^d $, and
\begin{align}
\label{hyp:Lyapunov_control_saut}
\vert \nabla V \vert^2 \leqslant C_V V, \qquad \Vert D^2 V \Vert_{\infty} < + \infty.
\end{align}
We now consider the mean-reverting property of $V$ for polynomial test functions, $i.e.$ when $\psi(y)=\psi_p(y)=y^p$, $y \geqslant 0$, $p>0$. Let
\begin{align}
\label{def:lambda_psi_saut}
\forall x \in \mathbb{R}^d , \quad \lambda_{\psi}(x):= \lambda_{D^2V(x)+2 \nabla V(x)^{\otimes 2} \psi''\circ V(x) \psi'\circ V(x)^{-1}}  .
\end{align}
We also use the notation $\lambda_p$ instead of $\lambda_{\psi_p}$. Now let $\phi:[v_{\ast}, + \infty) \to \mathbb{R}_+$. We suppose that
\begin{align}
\label{hyp:jump_component_control_coeff_ordre_p_control_stab}
\mathcal{H}^p(\phi,V) \quad \equiv \qquad  \qquad  \qquad   \forall x \in \mathbb{R}^d, \quad \tau_p  (x)   \leqslant  C  \phi \circ V (x)^p ,
\end{align}
and, when $p \geqslant 1$, we also introduce 
\begin{align}
\label{hyp:jump_component_control_coeff_ordre_p_control_stab_cas_inf}
\underline{\mathcal{H}}^p(\phi,V) \Leftrightarrow  \mathcal{H}^p \quad \mbox{and} \quad \mathcal{H}^{p'}(\phi,V), \forall p' \in [1,p) .
\end{align}
%
Moreover, assume that
\begin{align}
\label{hyp:controle_coefficients_saut_p_q}
\mathfrak{B}_{q}(\phi) \quad \equiv \qquad \forall x \in \mathbb{R}^d,  \quad  \vert \tilde{b}_{q} \vert^2  \leqslant C  \phi \circ V (x),
\end{align}
 Let $\beta \in \mathbb{R}$ and $\alpha>0$. We assume that $V$ satisfies the following mean-reverting property:
\begin{align}
\label{hyp:recursive_control_param_saut}
\mathcal{R}_{p,q}(\alpha,\beta,\phi,V) \; \equiv \quad \forall x \in \mathbb{R}^d,  \quad  \Big\langle \nabla V(x), b(x)+ \int_{F} c(z,x ) \zeta (z,x ) \pi(dz) \Big\rangle + \frac{1}{2} \chi_{p,q}(x) \leqslant \beta - \alpha \phi \circ V (x),
\end{align}
with
%
%
%
%
\begin{align}
\label{hyp:recursive_control_param_saut_terme_ordre_sup}
&\forall x \in \mathbb{R}^d,  \quad  \chi_{p,q}(x)= 
   \left\{
      \begin{aligned}
        & 2 p^{-1}V^{1-p}(x) \tilde{\chi}_{p,q} (x) & & \quad \mbox{if } p \leqslant 1\\
        &\Vert \lambda_{1} \Vert_{\infty} \tau_1(x) & & \quad \mbox{if } p =1\\
        & \Vert \lambda_{p} \Vert_{\infty} 2^{(2p-3)_+} \big( \tau_1(x)+[\sqrt{V}]_1^{2p-2} V^{1-p}(x)  \mathfrak{K}_p \tau_{p}(x)\big)&   & \quad \mbox{if } p > 1 ,
      \end{aligned}
    \right.
\end{align}
with $\tilde{\chi}_{p,q} (x) = (\mathds{1}_{q \leqslant 1/2}   [V^p]_{2q}+\mathds{1}_{q\in (1/2,+\infty)} C_{q\wedge 1}  [V^{p-1}\nabla V]_{2 (q\wedge 1) -1} ) \tau_{q\wedge 1}(x)$, and, for every $p>1$,
\begin{align}
\label{eq:constante_K_BDG_succ}
 \mathfrak{K}_p=\mathds{1}_{\pi(F)< +\infty} +\mathds{1}_{\pi(F)= +\infty}p2^{2p}2^{p/(2-2^{1-k_0}) -k_0} C_p \prod_{k=1}^{k_0} C_{p2^{1-k}}, \quad k_0= \inf \{k, 2^k \geqslant p\}= \lceil \log_2(p)\rceil ,
\end{align}
with $C_r$, $r \geqslant 1$, the constant from the BDG inequality defined in (\ref{eq:BDG_inegalite}). \\
For $p>0,a\in (0,1],s\geqslant 1, \rho \in [1,2]$, we consider the following assumption
\begin{align}
\label{hyp:control_step_weight_pol_saut}
 \mathcal{S} \mathcal{W}_{pol}(p,a,s,\rho) \quad \equiv \qquad  \qquad a p \rho /s \leqslant a+p-1 , \qquad \mbox{and} \qquad \rho \leqslant s.
\end{align}
Finally, consider also the hypothesis
\begin{changemargin}{1cm}{1cm} 
\begin{center}
 \begin{align}
 \label{hyp:control_step_weight_pol_saut_control}
 \mathcal{S} \mathcal{W}_{\mbox{Jump}}(p,q,a,s,\rho)
\end{align}
 $\equiv $
 \end{center}
%
%
Assume that $q \geqslant  (\rho/2) \vee p$ and let us define $\epsilon_{\mathcal{I}}(\gamma)=\mathds{1}_{2p>s}\gamma^{\rho/(2(q\vee 1/2)) }+\gamma^{ ( 2\wedge (1/q)) p\rho/s}$ and $\tilde{\epsilon}_{\mathcal{I}}(\gamma)=\gamma^{1 \wedge (\rho/(2q))}$ and let $\phi(y)=y^a$, $y \in[0,+\infty)$. Assume  that $\mathcal{H}^{q}(\phi,V) $ (see (\ref{hyp:jump_component_control_coeff_ordre_p_control_stab})) holds and that when $\pi(F)=+\infty$ and $q>1$ we have $\underline{\mathcal{H}}^{q}(\phi,V) $ (see (\ref{hyp:jump_component_control_coeff_ordre_p_control_stab_cas_inf})). \\
Finally, assume that $\mathcal{S}\mathcal{W}_{\mathcal{I}, \gamma,\eta}(\rho, \epsilon_{\mathcal{I}})$ (see (\ref{hyp:step_weight_I})) and $\mathcal{S}\mathcal{W}_{\mathcal{I}, \gamma,\eta}( 1\vee \tau_{q}^{1 \wedge (\rho/(2q))} ,\rho,\tilde{\epsilon}_{\mathcal{I}}) $ (see (\ref{hyp:step_weight_I_gen_chow})), with $\overline{X}$ replaced by $\overline{X}_{q}$, hold. \\
Notice that when $a(q \vee (\rho /2)) \leqslant p+a-1$, the assumption $\mathcal{S}\mathcal{W}_{\mathcal{I}, \gamma,\eta}(  1\vee\tau_{q}^{1 \wedge (\rho/(2q))} ,\rho,\tilde{\epsilon}_{\mathcal{I}}) $ (see (\ref{hyp:step_weight_I_gen_chow})) can be replaced by $\mathcal{S}\mathcal{W}_{\mathcal{I}, \gamma,\eta}(\rho, \tilde{\epsilon}_{\mathcal{I}})$ (see (\ref{hyp:step_weight_I})) and $\mathcal{H}^{q}(\phi,V) $ (see (\ref{hyp:jump_component_control_coeff_ordre_p_control_stab})).
\end{changemargin} 
\begin{mytheo}
\label{th:cv_was_saut}
Let $p > 0$, $a \in (0,1]$, $s \geqslant 1, \rho \in [1,2]$ and let $\psi_p(y)=y^p$, $\phi(y)=y^a$. Let $q_p\in[p,1]$ if $p\leqslant 1$ and $q_p=p$ if $p\geqslant 1$. Let $\alpha>0$ and $\beta \in \mathbb{R}$.\\

 Assume that $\mbox{L}_{V}$ (see (\ref{hyp:Lyapunov})) holds and that $p/s+a-1 >0$. 
Assume also that $\mathfrak{B}_{q_p}(\phi)$ (see (\ref{hyp:controle_coefficients_saut_p_q})) and  $\mathcal{R}_{p,q_p}(\alpha,\beta,\phi,V)$ (\ref{hyp:recursive_control_param_saut}) hold and that:
 \begin{enumerate}[label=\textbf{\roman*.}]
\item \textbf{Case $p>1$ ($q_p=p$).} If $\pi(F)<+\infty$ assume that $\mathcal{H}^p(\phi,V) $ and $\mathcal{H}^{1}(\phi,V) $ (see (\ref{hyp:jump_component_control_coeff_ordre_p_control_stab})) are satisfied. If $\pi(F)=+\infty$ assume that $\underline{\mathcal{H}}^{p}(\phi,V) $ (see (\ref{hyp:jump_component_control_coeff_ordre_p_control_stab_cas_inf})) holds.
\item \textbf{Case $p \leqslant 1$.} Assume that $\mathcal{H}^{q_p} $ (see (\ref{hyp:jump_component_ordre_p})) holds.
\end{enumerate}
%
%
Suppose that $ \mathcal{S} \mathcal{W}_{\mbox{Jump}}(p,q_p,a,s,\rho)$ (see (\ref{hyp:control_step_weight_pol_saut_control})), $\mathcal{S}\mathcal{W}_{\mathcal{II},\gamma,\eta}(V^{p/s}) $ (see (\ref{hyp:step_weight_I_gen_tens})) with $\overline{X}$ replaced by $\overline{X}_{q_p}$, $\mathcal{S} \mathcal{W}_{pol}(p,a,s,\rho)$ (see (\ref{hyp:control_step_weight_pol_saut})) and (\ref{hyp:accroiss_sw_series_2}) are satisfied. \\

Then $(\nu_n^{\eta,q_p})_{n \in \mathbb{N}^{\ast}}$ is $\mathbb{P}-a.s.$ tight and 
\begin{align*}
\mathbb{P} \mbox{-a.s.} \quad \sup_{n \in \mathbb{N}^{\ast}} \nu_n^{\eta,q_p}( V^{p/s+a-1} ) < + \infty .
\end{align*}
Moreover, assume also that $b+\mathds{1}_{\pi(F) =+ \infty}\kappa_{q_p}$ has sublinear growth and that, when $\pi(F) =+ \infty$, there exists $r \in [0,1/2]$ such that $\mathcal{H}^{r+\mathds{1}_{q_p\in(1/2,+\infty)}/2}$ (see (\ref{hyp:jump_component_ordre_p})), (\ref{hyp:struc_param_saut_x_infini}) and (\ref{hyp:cv_zero_sublinear_jump}) hold and that $\tau_{r+\mathds{1}_{q_p\in(1/2,+\infty)}/2} \leqslant  CV^{p/s+a-1}$. \\

Then, every weak limiting distribution $\nu$ of $(\nu_n^{\eta,q_p})_{n \in \mathbb{N}^{\ast}}$ is an invariant distribution of $(X_t)_{t \geqslant 0}$ and when $\nu$ is unique, we have
\begin{align*}
\mathbb{P} \mbox{-a.s.} \quad  \forall f \in \mathcal{C}_{\tilde{V}_{\psi_p,\phi,s}}(\mathbb{R}^d), \quad \lim\limits_{n \to + \infty} \nu_n^{\eta,q_p}(f)=\nu(f),
\end{align*}
 with $\mathcal{C}_{\tilde{V}_{\psi_p,\phi,s}}(\mathbb{R}^d)$ defined in (\ref{def:espace_test_function_cv}). Notice that when $p/s \leqslant p +a-1$, the assumption $\mathcal{S}\mathcal{W}_{\mathcal{II},\gamma,\eta}(V^{p/s}) $ (see (\ref{hyp:step_weight_I_gen_tens})) can be replaced by $\mathcal{S}\mathcal{W}_{\mathcal{II},\gamma,\eta} $ (see (\ref{hyp:step_weight_II})).
\end{mytheo}
\begin{remark}
Actually, we show that this Theorem holds when $ \mathcal{S} \mathcal{W}_{\mbox{Jump}}(p,q_p,a,s,\rho)$ (see (\ref{hyp:control_step_weight_pol_saut_control}) is replaced by the following weaker assumption (avoided for sake of clarity in the presentation):
\begin{changemargin}{1cm}{1cm} 
\begin{center}
 \begin{align}
 \label{hyp:control_step_weight_pol_saut_control_bis}
\widetilde{\mathcal{S} \mathcal{W}}_{\mbox{Jump}}(p,q,a,s,\rho)
\end{align}
 $\equiv $
 \end{center}
Let us consider $\tilde{q}_1 \geqslant  \rho/2$, $\tilde{q}_2 \geqslant  p$, $\tilde{q}_3>0$ and let us define $\epsilon_{\mathcal{I}}(\gamma)=\mathds{1}_{2p>s}\gamma^{\rho/(2(\tilde{q}_1\vee 1/2)) }+\gamma^{ ( 2\wedge (1/\tilde{q}_2)) p\rho/s}$ and $\tilde{\epsilon}_{\mathcal{I}}(\gamma)=\gamma^{1 \wedge (\rho/(2\tilde{q}_3))}$ and let $\phi(y)=y^a$, $y \in[0,+\infty)$. Assume  that $\mathcal{H}^{\tilde{q}_1}(\phi,V) $ (see (\ref{hyp:jump_component_control_coeff_ordre_p_control_stab})) when $2p>s$ (respectively $\mathcal{H}^{\tilde{q}_2}(\phi,V) $, $\mathcal{H}^{ \tilde{q}_3}$ (see (\ref{hyp:jump_component_ordre_p})) when $p>0$) holds and that when $\pi(F)=+\infty$ and $\tilde{q}_1>1$ (resp. $\tilde{q}_2>1$,$\tilde{q}_3>1$), we have $\underline{\mathcal{H}}^{\tilde{q_1}}(\phi,V) $ (see (\ref{hyp:jump_component_control_coeff_ordre_p_control_stab_cas_inf})) (resp. $\underline{\mathcal{H}}^{\tilde{q_2}}(\phi,V) $, $\underline{\mathcal{H}}^{\tilde{q_3}}(\phi,V) $). \\
Finally, assume that $\mathcal{S}\mathcal{W}_{\mathcal{I}, \gamma,\eta}(\rho, \epsilon_{\mathcal{I}})$ (see (\ref{hyp:step_weight_I})) and $\mathcal{S}\mathcal{W}_{\mathcal{I}, \gamma,\eta}( 1\vee \tau_{\tilde{q}_3}^{1 \wedge (\rho/(2\tilde{q}_3))} ,\rho,\tilde{\epsilon}_{\mathcal{I}}) $ (see (\ref{hyp:step_weight_I_gen_chow})), with $\overline{X}$ replaced by $\overline{X}_{q}$, hold. \\
Notice that when $a(\tilde{q}_3 \vee (\rho /2)) \leqslant p+a-1$, the assumption $\mathcal{S}\mathcal{W}_{\mathcal{I}, \gamma,\eta}(  1\vee\tau_{\tilde{q}_3}^{1 \wedge (\rho/(2\tilde{q}_3))} ,\rho,\tilde{\epsilon}_{\mathcal{I}}) $ (see (\ref{hyp:step_weight_I_gen_chow})) can be replaced by $\mathcal{S}\mathcal{W}_{\mathcal{I}, \gamma,\eta}(\rho, \tilde{\epsilon}_{\mathcal{I}})$ (see (\ref{hyp:step_weight_I})) and $\mathcal{H}^{\tilde{q}_3}(\phi,V) $ (see (\ref{hyp:jump_component_control_coeff_ordre_p_control_stab})).
\end{changemargin}
\end{remark}
\subsubsection{Proof of the recursive mean reverting control}
Before we establish the recursive mean reverting control, we provide some useful results concerning the jump component. 
\begin{lemme}
\label{lemme:control_jump}
Let $t ,\gamma \geqslant 0$. We have the following properties:
\begin{enumerate}[label=\textbf{\Alph*.}]
\item \label{lemme:control_jump_point_1} Let $p > 0$. Assume that $\pi(F)<\infty$, $F_{\gamma}=F, \forall \gamma>0$, and that $\mathcal{H}^p $ (see (\ref{hyp:jump_component_ordre_p})) holds. Then, there exists a locally bounded function $\epsilon : \mathbb{R}_+ \to \mathbb{R}$ such that for every $\overline{t}\geqslant 0$ we have $\epsilon(t) \leqslant C t, \forall t \in [0,\overline{t}]$ with $C>0$, and such that for every $n \in \mathbb{N}$,
\begin{align}
\label{eq:control_big_jump}
\mathbb{E} [\vert   M^{\gamma}_t(x) \big \vert ^{2p} ]\leqslant    t(1+ \epsilon( t) )  \tau_{p,\gamma} (x) =t(1+ \epsilon( t) )  \tau_{p} (x) .
\end{align}
\item \label{lemme:control_jump_point_3_bis} Let $p > 1$ and assume that $\underline{\mathcal{H}}^{p}(\phi,V) $ (see (\ref{hyp:jump_component_control_coeff_ordre_p_control_stab_cas_inf})) holds. Then, there exists $\xi >1$ and a finite constant $\mathfrak{C}_{p}\geqslant 0$, such that 
\begin{align}
\label{eq:control_small_jump_p_sup_un}
\mathbb{E} [  \vert  \widetilde{M}^{\gamma}_t(x) \vert^{2p}] \leqslant      \mathfrak{K}_p t \tau_{p,\gamma} (x)  +\mathfrak{C}_{p} \gamma^{\xi} \phi \circ V(x)^p,
\end{align}
where $\mathfrak{K}_p$ is the constant defined in (\ref{eq:constante_K_BDG_succ}).
\item \label{lemme:control_jump_point_2}  Let $p \in (0,1]$ and assume that $\mathcal{H}^p$ (see (\ref{hyp:jump_component_ordre_p})) holds.Then 
\begin{align} 
\mathbb{E} [  \vert  M^{\gamma}_t(x) \vert^{2p}]&\leqslant t  \tau_{p,\gamma} (x) & & \quad \mbox{if } p \in (0,1/2]  \label{eq:control_small_jump_p_inf_un_demi} \\
\mathbb{E} [  \vert  \widetilde{M}^{\gamma}_t(x) \vert^{2p}]     &\leqslant  C_p t \tau_{p,\gamma} (x) & & \quad \mbox{if } [1/2,1)    \label{eq:control_jump_p<1} \\
 \mathbb{E} [  \vert  \widetilde{M}^{\gamma}_t(x) \vert^{2p}]   &= t \tau_{1,\gamma} (x)&   & \quad \mbox{if } p = 1 , \label{eq:control_jump_p=1}
\end{align}
%
%
with $C_p$ the constant which appears in the BDG inequality (see (\ref{eq:BDG_inegalite})).\\
%
%
%
%
  \end{enumerate}
  Moreover, those results remain true when we replace $\tau_{p,\gamma}$ by $\tau_{p}$.
\end{lemme}
\begin{proof} 
We prove point \ref{lemme:control_jump_point_1} Let $(J_t)_{t \geqslant 0}$, a Poisson process with intensity $\tilde{\pi} :=t\Vert \zeta \Vert_{\infty} \pi(F)$. We introduce the sequences of independent random variables (and independent from $J$)
\begin{align*}
Z_k \sim  \pi(F)^{-1}  \pi(dz), \qquad \mbox{and}  \qquad   V_k \sim \Vert \zeta \Vert_{\infty}^{-1}  (v)  \mathds{1}_{v \leqslant \Vert \zeta \Vert_{\infty}} dv.
\end{align*} 
Now, we rewrite $M^{\gamma}_t(x)= \sum_{k=1}^{J_{t}} c(Z_k,x) \mathds{1}_{V_k \leqslant \zeta(Z_k, x) }$. Therefore, we study 
\begin{align*}
\mathbb{E} \Big[ \big \vert    \sum_{k=1}^{J_t} c(Z_k,x) \mathds{1}_{V_k \leqslant \zeta(Z_k, x) } \big \vert^{2p} \Big]= & \mathbb{E} \Big[   \sum_{k \geqslant 1} \mathds{1}_{J_{t}=k}  \big \vert   \sum_{l=1}^k c(Z_l,x) \mathds{1}_{V_l \leqslant \zeta(Z_l, x) } \big \vert^{2p}  \Big] \\
= &  \sum_{k \geqslant 1} \mathbb{E} \Big[   \big \vert   \sum_{l=1}^k c(Z_l,x) \mathds{1}_{V_l \leqslant \zeta(Z_l, x) } \big \vert^{2p}  \Big] \frac{e^{- \tilde{\pi} t }( \tilde{\pi} t)^k}{k!}.
\end{align*}
We put $\alpha=2p$ in the inequality (\ref{eq:puisance_somme_n_terme}) and it follows that
\begin{align*}
 \mathbb{E} \Big[   \big \vert   \sum_{l=1}^k c(Z_l,x) \mathds{1}_{V_l \leqslant \zeta(Z_l, x) } \big \vert^{2p}  \Big] \leqslant &     k^{(2p-1)_+} \sum_{l=1}^k \mathbb{E} \Big[  \vert    c(Z_l,x) \mathds{1}_{V_l \leqslant \zeta(Z_l, x) } \vert^{2p}  \Big] = k^{1+(2p-1)_+}  \tilde{\pi}^{-1} \tau_p(x).
\end{align*}
Moreover,
\begin{align*}
 e^{- \tilde{\pi} t}  \sum_{k \geqslant 1}  k^{1+(2p-1)_+} \frac{( \tilde{\pi} t)^k}{k!}  =   e^{- \tilde{\pi} t } \tilde{\pi} t \sum_{k \geqslant 0}  (k+1)^{(2p-1)_+} \frac{( \tilde{\pi} t)^k}{k!} 
\end{align*}
Now we are going to use the following result
\begin{lemme}
Let $\overline{\theta} \geqslant 0$, $a\in \mathbb{N}$ and $\theta \in [0,\overline{\theta}]$. Then
\begin{align*}
&\sum_{k \geqslant 0}  (k+1)^{a} \frac{\theta^k}{k!} \leqslant
   \left\{
      \begin{aligned}
        & \exp(\theta)(1 +  \theta),& & \quad \mbox{if } a \leqslant 1\\
        & \exp(\theta)(1+\epsilon(\theta))   & & \quad \mbox{if } a \geqslant 1,
      \end{aligned}
    \right.
\end{align*}
where the function $\epsilon : \mathbb{R}_+ \to \mathbb{R}$ satisfies $\epsilon(\theta) \leqslant C \theta$ for every $\theta \in [0,\overline{\theta}]$.
\end{lemme}
\begin{proof}
When $a \leqslant 1$ we use the following inequality
\begin{align}
\label{eq:puisssance_inf_1}
\forall u,v \in \mathbb{R}_+,\forall \alpha \in (0,1], \qquad (u+v)^{\alpha} \leqslant & u^{\alpha}+v^{\alpha} , 
\end{align}
with $\alpha=a$ and derive
\begin{align*}
 \sum_{k \geqslant 0}  (k+1)^{a} \frac{\theta^k}{k!}   \leqslant  \exp(\theta) +   \sum_{k \geqslant 0}  k^{a} \frac{\theta^k}{k!}  =  \exp(\theta) +  \sum_{k \geqslant 1}  k^{a} \frac{\theta^k}{k!} = & \exp(\theta) +  \sum_{k \geqslant 1}  k^{a-1} \frac{\theta^k}{(k-1)!}    \\
 \leqslant &  \exp(\theta) +  \sum_{k \geqslant 1}  \frac{\theta^k}{(k-1)!}    \\
 = &\exp(\theta) +  \theta \exp(\theta).
\end{align*}
%
%
Assume now that $a \geqslant 1$, we apply the inequality \ref{eq:puisssance_sup_1}
 with $u=k+1$, $v=1$ and $\alpha=a$, and it follows that
\begin{align*}
\sum_{k \geqslant 0}  (k+1)^{a} \frac{\theta^k}{k!}  \leqslant  \exp(\theta)+  a2^{a-1} \sum_{k \geqslant 0} (k+ k^{a}) \frac{\theta^k}{k!} =& \exp(\theta)+  a2^{a-1} \sum_{k \geqslant 1} (k+ k^{a}) \frac{\theta^k}{k!} \\
 =& \exp(\theta)+  a2^{a-1} \sum_{k \geqslant 1} (1+ k^{a-1}) \frac{\theta^k}{(k-1)!} \\
=& \exp(\theta)+  a2^{a-1} \theta \exp(\theta) + a2^{a-1} \theta  \sum_{k \geqslant 0} (k+1)^{a-1} \frac{\theta^k}{k!}, 
\end{align*}
and a recursive approach yields the result
\end{proof}
We apply this Lemma with $\theta=\tilde{\pi} t$ and $a=(2p-1)_+$ and (\ref{eq:control_big_jump}) follows. \\
We now focus on the proof of point \ref{lemme:control_jump_point_3_bis}. For any $k \in \mathbb{N}^{\ast}$, $  \widetilde{M}^{\gamma,k}_{t}:=\sum_{s \leqslant t}\vert \Delta \widetilde{M}^{\gamma}_{s} \vert^{2^k} - t \tau_{2^{k-1}, \gamma}(x) $, $t\geqslant 0$,  is a martingale (with notation $\Delta  \widetilde{M}^{\gamma}_{t}= \widetilde{M}^{\gamma}_{t} - \widetilde{M}^{\gamma}_{t^{-}}, \forall t \geqslant 0$).  Using (\ref{eq:puisance_somme_n_terme}) for $(\widetilde{M}^{\gamma,k}_{1,t})_{ t \geqslant 0}$, it follows that
%
%
%
\begin{align*}
\mathbb{E} \big[  \vert \sum_{s \leqslant \gamma}  \vert \Delta \widetilde{M}^{\gamma}_{s} \vert^{2^k} \vert^{p/2^{k-1}}   \big] = &  \mathbb{E} \big[  \vert \widetilde{M}^{\gamma,k}_{t}  +  t \tau_{2^{k-1},\gamma}(x) \vert^{p/2^{k-1}}   \big]  \\
\leqslant & 2^{(p/2^{k-1}-1)_+}\mathbb{E} \big[  \vert \widetilde{M}^{\gamma,k}_{t}  \vert^{p/2^{k-1}} \big]  +    2^{(p/2^{k-1}-1)_+}\vert t \tau_{2^{k-1},\gamma}(x) \vert^{p/2^{k-1}}  \\
\leqslant & 2^{(p/2^{k-1}-1)_+}C_{p/2^{k-1}} \mathbb{E} \big[  \vert \sum_{s \leqslant t} \vert \Delta \widetilde{M}^{\gamma}_{s}  \vert^{2^{k+1}} \vert^{p/2^{k}} \big]  +    2^{(p/2^{k-1}-1)_+}\vert t \tau_{2^{k-1},\gamma}(x) \vert^{p/2^{k-1}}  .
\end{align*}
Now, let $k_0= \inf \{k \in \mathbb{N}^{\ast};   2^{k} \geqslant p  \} $. Using (\ref{eq:puisssance_inf_1}), and $\mathcal{H}^p$, we have 
\begin{align*}
 \mathbb{E} \big[  \vert \sum_{s \leqslant t} \vert \Delta \widetilde{M}^{\gamma}_{s}  \vert^{2^{k_0+1}} \vert^{p/2^{k_0}} \big] \leqslant  \mathbb{E} \big[  \sum_{s \leqslant t} \vert \Delta \widetilde{M}^{\gamma}_{s}  \vert^{2p} \big] =t \tau_{p,\gamma}(x) .
\end{align*}
Since $2^k <p$ for every $k < k_0$, it follows that 
\begin{align*}
\mathbb{E}[\vert \widetilde{M}^{\gamma}_t \vert^{2p} ] \leqslant \frac{ \mathfrak{K}_p}{p2^{2p}} t  \tau_{p,\gamma}(x) +C \sum_{k=1}^{k_0}  \vert t \tau_{2^{k-1},\gamma}(x) \vert^{p/2^{k-1}  } \leqslant \frac{ \mathfrak{K}_p}{p2^{2p}}  t \tau_{p}(x) +  C t^{p/2^{k_0-1}} \sup_{k \in \{1,\ldots,k_0\}} \vert  \tau_{2^{k-1}}(x) \vert^{p/2^{k-1}  }  
 \end{align*}
 with $\mathfrak{K}_p$ the constant defined in (\ref{eq:constante_K_BDG_succ}). Since we have $\mathcal{H}^{p'} (\phi,V)$ for every $p' \in [1,p)$, it follows that there exists $\xi >1$ and $\mathfrak{C}_p \geqslant 0$, such that
 \begin{align*}
\mathbb{E}[\vert \widetilde{M}^{\gamma}_t \vert^{2p} ]  \leqslant & \frac{ \mathfrak{K}_p}{p2^{2p}} t \tau_{p,\gamma}(x) +\mathfrak{C}_p t^{\xi} \phi \circ V (x)^{p} ,
 \end{align*}
that is (\ref{eq:control_small_jump_p_sup_un}). Finally, we consider the proof of point \ref{lemme:control_jump_point_2}  First we treat the case $p=1$. In this case, the process $(\mathcal{M}^{\gamma}_t)_{t \geqslant 0}$ such that $\mathcal{M}^{\gamma}_t:= \vert \widetilde{M}^{\gamma}_{t} \vert^2- t\tau_{1,\gamma}(x)$, for every $t \geqslant 0$, is a martingale and then, for every $t \geqslant 0$, we have
  \begin{align*}
\mathbb{E}[\vert \widetilde{M}^{\gamma}_t \vert^2] =t \tau_{1,\gamma}(x).
\end{align*}
Let $p \in (0,1/2]$. We apply the inequality (\ref{eq:puisance_somme_n_terme}) and the compensation formula, and (\ref{eq:control_small_jump_p_inf_un_demi}) follows from
\begin{align*}
\mathbb{E} [ \vert M^{\gamma}_t(x) \vert^{2p} ] \leqslant & \mathbb{E} \big[  \sum_{s \leqslant t} \vert \Delta  M^{\gamma}_{s}(x)  \vert^{2p} \big] =t \tau_{p,\gamma}(x)  .
\end{align*}
Finally, let $p \in [1/2,1)$. Using the BDG inequality (see (\ref{eq:BDG_inegalite})), (\ref{eq:puisssance_inf_1}) and the compensation formula, we derive
\begin{align*}
\mathbb{E}[\vert  \widetilde{M}^{\gamma}_{t} \vert^{2p} ] \leqslant  C_p \mathbb{E} \big[  \vert \sum_{s \leqslant t} \vert \Delta  M^{\gamma}_{s} \vert^2 \vert^p   \big] \leqslant &  C_p \mathbb{E} \big[  \sum_{s \leqslant t} \vert \Delta  M^{\gamma}_{s} \vert^{2p}   \big] = C_p  t \tau_{p,\gamma}(x) ,
\end{align*}
and the proof is completed.
\end{proof}
\begin{lemme}
\label{lemme:control_seut_hyp_q}
Let $\gamma \geqslant 0$, $x \in \mathbb{R}^d$ $p \in (0,1]$ and $q\in[p,1]$. We assume that $\mathcal{H}^{q} $ (see (\ref{hyp:jump_component_ordre_p})) and (\ref{hyp:Lyapunov_control_saut}) hold. Then, for every $x_0 \in \mathbb{R}^d$, we have
\begin{align}
\mathbb{E}[V^{p} ( x_0+\mathds{1}_{q \in (0,1/2]}M^{\gamma}_{\gamma} (x) + \mathds{1}_{q \in (1/2,1]}\widetilde{M}^{\gamma}_{\gamma} (x))-V^{p}(x_0 )]  \leqslant & \gamma \frac{1}{2} \tilde{\chi}_{p,q} (x) \nonumber 
\end{align}
where $\tilde{\chi}_{p,q}$ is defined in (\ref{hyp:recursive_control_param_saut_terme_ordre_sup}).
\end{lemme}
\begin{proof}
 Assume first that $q \leqslant 1/2$. Since we have (\ref{hyp:Lyapunov_control_saut}), the function $V^p$ is $\alpha$-H\"older for every $\alpha \in [2p,1]$ (see \cite{Panloup_2008}, Lemma 3). If follows from Lemma \ref{lemme:control_jump} point \ref{lemme:control_jump_point_2} (see (\ref{eq:control_small_jump_p_inf_un_demi})) that for $q \in [p,1/2]$ we have
\begin{align*}
\mathbb{E}[\vert V^{p} ( x_0+M^{\gamma}_{\gamma} (x) )-V^{p}(x_0 ) \vert]  \leqslant & [V^p]_{2q}  \mathbb{E}[ \vert M^{\gamma}_{\gamma}(x)\vert^{2q} ]  \leqslant  [V^p]_{2q} \gamma \tau_q(x) .
\end{align*}
Assume now that $q \in (1/2,1]$. Since we have (\ref{hyp:Lyapunov_control_saut}), the function $x \mapsto V^{p-1}(x)\nabla V(x)$ is $2q-1$-H\"older in this case (see \cite{Panloup_2008}, Lemma 3) and since $\widetilde{M}^{\gamma}_{\gamma}(x)$ is centered, it follows from Lemma \ref{lemme:control_jump} point \ref{lemme:control_jump_point_2} (see (\ref{eq:control_jump_p<1}) and (\ref{eq:control_jump_p=1})), that
\begin{align*}
\mathbb{E}[V^{p} ( x_0 +\widetilde{M}^{\gamma}_{\gamma} (x))-V^{p}(x_0 )] \leqslant & \mathbb{E} [V^{p-1} (x_0 ) \nabla V (x_0)  \widetilde{M}^{\gamma}_{\gamma}(x) + [V^{p-1}\nabla V]_{2 q -1} \vert \widetilde{M}_{\gamma}(x)  \vert^{2 q }  ]\\
 \leqslant & C_q  [V^{p-1}\nabla V]_{2 q -1}  \gamma \tau_q(x) ,
\end{align*}
 which concludes the proof.
%
%
%
%
%
%
%
%
 \end{proof}
Now, we are able to present the weakly mean reverting recursive control result for test functions with polynomial growth.
\begin{myprop}
\label{prop:recursive_control_saut}
Let $v_{\ast}>0,p\geqslant 0$ and let $\phi:[v_{\ast},\infty )\to \mathbb{R}_+$ be a continuous function such that $C_{\phi}:= \sup_{y \in [v_{\ast},\infty )}\phi(y)/y< +\infty$ and let $\psi_p(y)=y^p$. Let $q_p\in[p,1]$ if $p\leqslant 1$ and $q_p=p$ if $p\geqslant 1$. Let $\alpha>0$ and $\beta \in \mathbb{R}$.\\

Assume that $\mathfrak{B}_{q_p}(\phi)$ (see (\ref{hyp:controle_coefficients_saut_p_q})) and  $\mathcal{R}_{p,q_p}(\alpha,\beta,\phi,V)$ (\ref{hyp:recursive_control_param_saut}) hold and that the following assumptions are satisfied
 \begin{enumerate}[label=\textbf{\roman*.}]
\item \textbf{Case $p>1$ ($q_p=p$).} If $\pi(F)<+\infty$ assume that $\mathcal{H}^p(\phi,V) $ and $\mathcal{H}^{1}(\phi,V) $ (see (\ref{hyp:jump_component_control_coeff_ordre_p_control_stab})) are satisfied. If $\pi(F)=+\infty$ assume that $\underline{\mathcal{H}}^{p}(\phi,V) $ (see (\ref{hyp:jump_component_control_coeff_ordre_p_control_stab_cas_inf})) holds.
\item \textbf{Case $p \leqslant 1$.} Assume that $\mathcal{H}^{q_p} $ (see (\ref{hyp:jump_component_ordre_p})) holds.
\end{enumerate}
Then, for every $\tilde{\alpha} \in (0,\alpha)$, there exists $n_0 \in \mathbb{N}^{\ast}$, such that
\begin{align}
\label{eq:recursive_control_saut_fonction_pol_p_sup_un}
 \forall n   \geqslant n_0, \forall x \in \mathbb{R}^d,  \quad\widetilde{A}_{q_p,\gamma_n} \psi_p \circ V(x)\leqslant \frac{ \psi_p \circ V(x)}{V(x)}p \big(\beta - \tilde{\alpha} \phi\circ V(x) \big). 
\end{align}
Then $\mathcal{RC}_{Q,V}(\psi_p,\phi,p\tilde{\alpha},p\beta)$ (see (\ref{hyp:incr_sg_Lyapunov})) holds for every $\tilde{\alpha}\in (0, \alpha)$ such that $\liminf\limits_{y \to + \infty} \phi(y) > \beta / \tilde{\alpha}$. Moreover, when $\phi=Id$,
\begin{align}
\label{eq:mom_pol_saut_p_sup_1}
\sup_{n \in \mathbb{N}} \mathbb{E}[\psi_p \circ V (\overline{X}_{q_p,\Gamma_{n}})] < + \infty.
\end{align}
\end{myprop}
\begin{proof}
From the second order Taylor expansion and the definition of $\lambda_{\psi_p}=\lambda_p$ (see (\ref{def:lambda_psi_saut})), we derive
\begin{align}
\label{eq:taylor_preuve_RC_pol_saut}
\psi_{p} \circ V(\overline{X}_{q_p,\Gamma_{n+1}})=& \psi_{p} \circ V(\overline{X}_{q_p,\Gamma_n})+ \langle \overline{X}_{q_p,\Gamma_{n+1}}-\overline{X}_{q_p,\Gamma_n}, \nabla_x V(\overline{X}_{q_p,\Gamma_n}) \rangle \psi_{p}'\circ V(\overline{X}_{q_p,\Gamma_n}) \nonumber \\
&+ \frac{1}{2} ( D^2 V(\Upsilon_{n+1} ) \psi_{p}' \circ  V(\Upsilon_{n+1} )+\nabla V (\Upsilon_{n+1} )^{\otimes 2} \psi_{p}'' \circ V(\Upsilon_{n+1} ) )
 ( \overline{X}_{q_p,\Gamma_{n+1}}-\overline{X}_{q_p,\Gamma_n} )^{\otimes 2} \nonumber \\
 \leqslant & \psi_{p} \circ V(\overline{X}_{q_p,\Gamma_n})+ \langle \overline{X}_{q_p,\Gamma_{n+1}}-\overline{X}_{q_p,\Gamma_n}, \nabla V(\overline{X}_{q_p,\Gamma_n}) \rangle \psi_{p}'\circ V(\overline{X}_{q_p,\Gamma_n})  \nonumber \\
&+ \frac{1}{2} \lambda_{p} (\Upsilon_{n+1} ) \psi_{p}'\circ V(\Upsilon_{n+1} ) \vert   \overline{X}_{q_p,\Gamma_{n+1}}-\overline{X}_{q_p,\Gamma_n}  \vert^{2}. 
\end{align} 
with $\Upsilon_{n+1}  \in (\overline{X}_{q_p,\Gamma_n}, \overline{X}_{q_p,\Gamma_{n+1}})$. First, from (\ref{hyp:Lyapunov_control_saut}), we have $\sup_{x \in \mathbb{R}^d} \lambda_{p} (x)  < + \infty$. With notation (\ref{def:incr_pcmp_euler}), we compute
 \begin{align*}
 & \mathbb{E}[\overline{X}_{q_p,\Gamma_{n+1}}-\overline{X}_{q_p,\Gamma_n} \vert \overline{X}_{q_p,\Gamma_n} ]= \gamma_{n+1}\Big(  b(\overline{X}_{q_p,\Gamma_n} ) +\int_{F} (f(\overline{X}_{q_p,\Gamma_n}+c(z,\overline{X}_{q_p,\Gamma_n}))-f(\overline{X}_{q_p,\Gamma_n})) \zeta(z,\overline{X}_{q_p,\Gamma_n})\pi(dz) \Big) \\
& \mathbb{E}[ \vert \overline{X}_{q_p,\Gamma_{n+1}}-\overline{X}_{q_p,\Gamma_n} \vert^{ 2} \vert \overline{X}_{q_p,\Gamma_n} ]=  \mathbb{E}[ \vert  \Delta \overline{X}_{q_p,n+1}^{2} \vert^2 \vert \overline{X}_{q_p,\Gamma_{n}} ]  + \gamma_{n+1}^2  \vert b_{q_p}(\overline{X}_{q_p,\Gamma_n} )\vert^{2}+ \gamma_{n+1} \langle b_{q_p}(\overline{X}_{q_p,\Gamma_n} ) ,  \Delta \overline{X}_{q_p,n+1}^{2}\rangle .
\end{align*}
\paragraph{Case $p=1$.}
Let $p=1$ so that $q_p=1$. Since $\mathcal{H}^1 $ (see (\ref{hyp:jump_component_ordre_p})) holds, we derive from Lemma \ref{lemme:control_jump} point \ref{lemme:control_jump_point_2} (see (\ref{eq:control_jump_p=1})) that
\begin{align*}
\mathbb{E}[\vert  \Delta \overline{X}^{2}_{1,{n+1}} \vert^2\vert \overline{X}_{1,\Gamma_{n}} ]\leqslant  \gamma_{n+1} \tau_{1} ( \overline{X}_{1,\Gamma_{n}}  ).
\end{align*}
Using $\mathfrak{B}_{1}(\phi)$ (see (\ref{hyp:controle_coefficients_saut_p_q})), for every $\tilde{\alpha} \in (0, \alpha)$, there exists $n_0(\tilde{\alpha})$ such that for every $n\geqslant n_0(\tilde{\alpha})$,
\begin{align}
\label{eq:recursive_control_remaider_Id_saut}
\Vert \lambda_{1} \Vert_{\infty}  \gamma_{n+1}^2 \vert \tilde{b}_1(\overline{X}_{1,\Gamma_n} ) \vert^{ 2}  \leqslant \gamma_{n+1}(\alpha- \tilde{\alpha})\phi \circ V(\overline{X}_{1,\Gamma_n} ).
\end{align} 
We gather all the terms of (\ref{eq:taylor_preuve_RC_pol_saut}) together and using $\mathcal{R}_{1,1}(\alpha,\beta,\phi,V)$ (see (\ref{hyp:recursive_control_param_saut})), the proof is completed when $p=1$.
\paragraph{Case $p>1$. }Assume now that $p>1$ so that $q_p=p$ and $\psi_{p}'(y)=p y^{p-1}$. Since $\vert \nabla V \vert^2 \leqslant C_V V$ (see (\ref{hyp:Lyapunov_control_saut})), then $\sqrt{V}$ is Lipschitz. Using (\ref{eq:puisance_somme_n_terme}), it follows that
\begin{align*}
V^{p-1} (\Upsilon_{n+1} )  \leqslant & \big( \sqrt{V}(\overline{X}_{q_p,\Gamma_n})+[\sqrt{V}]_1 \vert \overline{X}_{q_p,\Gamma_{n+1}}-\overline{X}_{q_p,\Gamma_n} \vert \big)^{2p-2} \\
\leqslant & 2^{(2p-3)_+} (V^{p-1}(\overline{X}_{q_p,\Gamma_n}) + [\sqrt{V}]_1^{2p-2} \vert \overline{X}_{q_p,\Gamma_{n+1}}-\overline{X}_{q_p,\Gamma_n} \vert^{2p-2})
\end{align*}
To study the `remainder' of (\ref{eq:taylor_preuve_RC_pol_saut}), we multiply the above inequality by $\vert \overline{X}_{q_p,\Gamma_{n+1}}-\overline{X}_{q_p,\Gamma_n} \vert^{2}$. First, we study the second term which appears in the $r.h.s.$ and using $\mathfrak{B}_{p}(\phi)$ (see (\ref{hyp:controle_coefficients_saut_p_q})), for every $p \geqslant 1$, we have
\begin{align*}
\vert \Delta \overline{X}^{1}_{q_p,{n+1}} \vert^{2p} \leqslant C\gamma_{n+1}^{2p} \phi \circ V(\overline{X}_{q_p,\Gamma_{n}})^p, 
\end{align*}
with notations introduced in (\ref{def:incr_pcmp_euler}). Now we study $\mathbb{E} \big[ \big \vert   \Delta \overline{X}^{2}_{q_p,{n+1}} \big \vert ^{2p} \vert \overline{X}_{q_p,\Gamma_n}\big]$. We distinguish two cases: $\pi(F)<+\infty$ and $\pi(F)=+\infty$. First, let $\pi(F)<+\infty$. Using Lemma \ref{lemme:control_jump} point \ref{lemme:control_jump_point_1} (see (\ref{eq:control_big_jump})), $\mathcal{H}^{q_p}(\phi,V) $ (see (\ref{hyp:jump_component_control_coeff_ordre_p_control_stab})) and $q_p=p$, we deduce that
\begin{align*}
\mathbb{E} \big[ \big \vert   \Delta \overline{X}^{2}_{q_p,{n+1}} \big \vert ^{2p} \vert \overline{X}_{q_p,\Gamma_n}\big] \leqslant \gamma_{n+1} \tau_{q_p}(\overline{X}_{q_p,\Gamma_n}) (1 + \epsilon( \gamma_{n+1})) \leqslant \gamma_{n+1} \tau_{p}(\overline{X}_{q_p,\Gamma_n})+ C \gamma_{n+1}^2 \phi \circ V (\overline{X}_{q_p,\Gamma_n})^p
\end{align*}
Now let $\pi(F)=+\infty$. Using Lemma \ref{lemme:control_jump} point \ref{lemme:control_jump_point_3_bis} (see (\ref{eq:control_small_jump_p_sup_un})) since $\underline{\mathcal{H}}^{p}(\phi,V) $ (see (\ref{hyp:jump_component_control_coeff_ordre_p_control_stab_cas_inf})) holds, and $q_p=p$, we derive that there exists $\xi>1$ and $\mathfrak{C}_p\geqslant 0$ such that
\begin{align*}
\mathbb{E} \big[ \big \vert   \Delta \overline{X}^{2}_{q_p,{n+1}} \big \vert ^{2p} \vert \overline{X}_{q_p,\Gamma_n}\big] \leqslant  \gamma_{n+1}  \mathfrak{K}_p\tau_{p}(\overline{X}_{q_p,\Gamma_n})+ \mathfrak{C}_p \gamma_{n+1}^{\xi} \phi \circ V (\overline{X}_{q_p,\Gamma_n})^p
\end{align*}
It follows that in both cases ($\pi(F)<+\infty$ and $\pi(F)=+\infty$), there exists $\xi>1$ and $C\geqslant 0$ such that
\begin{align*}
\mathbb{E} \big[ \big \vert   \Delta \overline{X}^{2}_{q_p,{n+1}} \big \vert ^{2p} \vert \overline{X}_{q_p,\Gamma_n}\big] \leqslant  \gamma_{n+1}  \mathfrak{K}_p\tau_{p}(\overline{X}_{q_p,\Gamma_n})+ C \gamma_{n+1}^{\xi} \phi \circ V (\overline{X}_{q_p,\Gamma_n})^p
\end{align*}
Now, let $p':= 1-1/(2p)$. Using the Jensen's inequality, (\ref{eq:puisssance_inf_1}) and $\mathcal{H}^{p}(\phi,V) $ (see (\ref{hyp:jump_component_control_coeff_ordre_p_control_stab})), we have
\begin{align*}
\mathbb{E} \big[ \vert  \Delta \overline{X}^{2}_{q_p,{n+1}} \vert  ^{2p-1} \vert \overline{X}_{q_p,\Gamma_n} \big]  \leqslant  \mathbb{E} \big[ \vert   \Delta \overline{X}^{2}_{q_p,{n+1}}   \vert ^{2p} \vert \overline{X}_{q_p,\Gamma_n}\big]^{p'}   \leqslant &   \gamma_{n+1}^{p'} \mathfrak{K}_p^{p'} \tau_{p} (\overline{X}_{q_p,\Gamma_{n}} )^{p'}   +C \gamma_{n+1}^{\xi p'} \phi \circ V(\overline{X}_{q_p,\Gamma_{n}})^{pp'} \\
  \leqslant &C \gamma_{n+1}^{1-1/(2p)} \phi \circ V(\overline{X}_{q_p,\Gamma_{n}})^{p-1/2} .
\end{align*}
%
%
 Applying the inequality (\ref{eq:puisssance_sup_1}) 
with $u=\vert  \Delta \overline{X}^{1}_{q_p,{n+1}} +\Delta \overline{X}^{2}_{q_p,{n+1}}  \vert$, $v= \vert  \Delta \overline{X}^{2}_{q_p,{n+1}} \vert$ and $\alpha=2p$ and also $\vert u-v\vert \leqslant \vert  \Delta \overline{X}^{1}_{q_p,{n+1}} \vert$, it follows that
 \begin{align*}
 \mathbb{E}\big[ \vert \overline{X}_{q_p,\Gamma_{n+1}}- \overline{X}_{q_p,\Gamma_n} \vert^{2p} \vert \overline{X}_{q_p,\Gamma_n} \big] \leqslant & \gamma_{n+1} \mathfrak{K}_p \tau_{p} (\overline{X}_{q_p,\Gamma_{n}} )  +C \gamma_{n+1}^{\xi} \phi \circ V(\overline{X}_{q_p,\Gamma_{n}} )^p \\
 & + 2p2^{2p-1}  \big(  C  \gamma_{n+1}  \phi \circ V(\overline{X}_{q_p,\Gamma_{n}})^{1/2} \times  C \gamma_{n+1}^{p-1/2} \phi \circ V(\overline{X}_{q_p,\Gamma_{n}})^{1-1/(2p)}  \\
 & + C \gamma_{n+1}^{2p} \phi \circ V (\overline{X}_{q_p,\Gamma_{n}})^p \big)
\end{align*}
%
 Let $\hat{\alpha} \in (0,\alpha)$. We deduce that there exists $n_0(\hat{\alpha}) \in \mathbb{N}$ such that for any $n \geqslant n_0(\hat{\alpha})$, we have
\begin{align*}
 \mathbb{E}[ \vert \overline{X}_{q_p,\Gamma_{n+1}}- \overline{X}_{q_p,\Gamma_n} \vert^{2p} \vert \overline{X}_{q_p,\Gamma_n} ] \leqslant \gamma_{n+1} \mathfrak{K}_p \tau_{p} (\overline{X}_{q_p,\Gamma_{n}} ) + \gamma_{n+1} \phi \circ V (\overline{X}_{q_p,\Gamma_{n}} )^{p} \frac{\alpha- \hat{\alpha} }{C_{\phi}^{p-1} \Vert \lambda_{p} \Vert_{\infty} 2^{(2p-3)_+} [\sqrt{V}]_1^{2p-2} } .
\end{align*}
To treat the other term of the `remainder' of (\ref{eq:taylor_preuve_RC_pol_saut}) when $\pi(F)=+\infty$, we proceed as in (\ref{eq:recursive_control_remaider_Id_saut}) with $\Vert \lambda_{1} \Vert_{\infty}\tilde{b}_1$ replaced by $\Vert \lambda_{p} \Vert_{\infty} 2^{2p-3} [\sqrt{V}]_1^{2p-2}\tilde{b}_p  $, $\alpha$ replaced by $ \hat{\alpha}$ and $\tilde{\alpha} \in (0, \hat{\alpha})$. When $\pi(F)<+\infty$, the approach is similar using Lemma \ref{lemme:control_jump} (see (\ref{eq:control_big_jump})) for $q=1$ since $\mathcal{H}^{1}(\phi,V) $ (see (\ref{hyp:jump_component_control_coeff_ordre_p_control_stab})) holds. We gather all the terms of (\ref{eq:taylor_preuve_RC_pol_saut}) together and using $\mathcal{R}_{p,q_p}(\alpha,\beta,\phi,V)$ (see (\ref{hyp:recursive_control_param_saut})), for every $n \geqslant n_0(\tilde{\alpha}) \vee n_0(\hat{\alpha}) $, we obtain
\begin{align*}
\mathbb{E}[V^p (\overline{X}_{q_p,\Gamma_{n+1}})- & V^p(\overline{X}_{q_p,\Gamma_{n}}) \vert \overline{X}_{q_p,\Gamma_{n}}]  \\
\leqslant &  \gamma_{n+1}p V^{p-1}(\overline{X}_{q_p,\Gamma_{n}})( \beta - \alpha \phi \circ V (\overline{X}_{q_p,\Gamma_{n}})  ) \\
 & +\gamma_{n+1}p V^{p-1}(\overline{X}_{q_p,\Gamma_{n}}) \Big(  \phi \circ V (\overline{X}_{q_p,\Gamma_{n}}) (\hat{\alpha} -\tilde{\alpha}  ) + (\alpha-\hat{\alpha})  \frac{ V^{1-p}(\overline{X}_{q_p,\Gamma_{n}})  \phi \circ V (\overline{X}_{q_p,\Gamma_{n}})^{p} }{C_{\phi}^{p-1}}    \Big) \\
\leqslant&\gamma_{n+1} V^{p-1}(\overline{X}_{q_p,\Gamma_{n}}) \big( \beta p- \tilde{\alpha} p  \phi \circ V (\overline{X}_{q_p,\Gamma_{n}}  ) \big ),
\end{align*}
which is exactly the recursive control for $p>1$, that is (\ref{eq:recursive_control_saut_fonction_pol_p_sup_un}). The proof of (\ref{eq:mom_pol_saut_p_sup_1}) is an immediate application of Lemma \ref{lemme:mom_psi_V} as soon as we notice that the increments of the Euler scheme (\ref{eq:PDMP_bounded_Poiss_approx}) have finite polynomial moments which implies (\ref{eq:mom_psi_V}).
\paragraph{Case $p \in(0,1].$}
 Let $p \in (0,1]$ so that the function defined on $[v_{\ast}, \infty)$ by  $y \mapsto y^p$ is concave. Using then the Taylor expansion at order 2 of the function $V$, for every $x,y \in \mathbb{R}^d$, there exists $\xi \in [0,1]$ such that
\begin{align*}
V^p(y) - V^p(x) \leqslant & pV^{p-1}(x)(V(y)-V(x)) = &pV^{p-1}(x) \big( \langle \nabla V(x),y-x \rangle +\frac{1}{2}\Tr[ D^2V(\xi x+(1- \xi)y) (y-x)^{\otimes 2} ] \big)
\end{align*}
and then, 
\begin{align*}
V^p(y) - V^p(x) \leqslant & pV^{p-1}(x) \Big( \langle \nabla V(x),y-x \rangle \\
&+\frac{1}{2} \Vert D^2V \Vert_{\infty} \vert y - x \vert^2 \Big).
\end{align*}
We apply this inequality, and with the notation (\ref{hyp:recursive_control_param_saut_terme_ordre_un}), it follows that
\begin{align*}
\mathbb{E}[ V^p(\overline{X}^{1}_{q_p,\Gamma_{n+1}} ) - V^p(\overline{X}_{q_p,\Gamma_n} ) \vert \overline{X}_{q_p,\Gamma_n}  ] \leqslant & pV^{p-1} (  \overline{X}_{q_p,\Gamma_n}  ) (\gamma_{n+1}
 \langle \nabla V ( \overline{X}_{q_p,\Gamma_n} ),b_{q_p}(\overline{X}_{q_p,\Gamma_n}) \rangle \\
& + \frac{1}{2} \Vert D^2V \Vert_{\infty} \mathbb{E}[  \vert \overline{X}^{1}_{q_p,\Gamma_{n+1}}   - \overline{X}_{q_p,\Gamma_n} \vert^2 \vert \overline{X}_{q_p,\Gamma_n} ]  ).
\end{align*}
As in the proof of the case $p \geqslant 1$, it follows from $\mathfrak{B}_{q_p}(\phi) $ (see (\ref{hyp:controle_coefficients_saut_p_q})), that there exists $\widehat{\alpha} \in (0, \alpha)$ and $n_0(\widehat{\alpha}) \in \mathbb{N}^{\ast}$ such that for every $n \geqslant n_0(\widehat{\alpha})$, we have 
\begin{align*}
\mathbb{E}[  \vert \overline{X}^{1}_{q_p,\Gamma_{n+1}}- \overline{X}_{q_p,\Gamma_n} \vert^2 \vert \overline{X}_{q_p,\Gamma_n} ]  \leqslant  C \gamma_{n+1}^2  \phi \circ V ( \overline{X}_{q_p,\Gamma_n} ) \leqslant \gamma_{n+1} (\alpha - \widehat{\alpha}) \phi \circ V ( \overline{X}_{q_p,\Gamma_n} )
\end{align*}
Finally, since $\overline{X}_{q_p,\Gamma_{n+1}}=\overline{X}^{1}_{q_p,\Gamma_{n+1}}+ \Delta \overline{X}_{q_p,n+1}^{2}$, we use Lemma \ref{lemme:control_seut_hyp_q} together with $\mathcal{H}^{q_p} $ (see (\ref{hyp:jump_component_ordre_p})) and we obtain
\begin{align*}
\mathbb{E}[ V^p(\overline{X}_{q_p,\Gamma_{n+1}} )  - V^p(\overline{X}^{1}_{q_p,\Gamma_{n+1}}) \vert \overline{X}_{q_p,\Gamma_n}  ] \leqslant & \gamma_{n+1} \frac{1}{2} \tilde{\chi}_{p,q_p}(\overline{X}_{q_p,\Gamma_n}) .
\end{align*}
Gathering all the terms together and using $\mathcal{R}_{p,q_p}(\alpha,\beta,\phi,V)$ (see (\ref{hyp:recursive_control_param_saut})) yields the recursive control (\ref{eq:recursive_control_saut_fonction_pol_p_sup_un}). The proof of (\ref{eq:mom_pol_saut_p_sup_1}) is an immediate application of Lemma \ref{lemme:mom_psi_V} as soon as we notice that the increments of the Euler scheme (\ref{eq:PDMP_bounded_Poiss_approx}) have finite polynomial moments which implies (\ref{eq:mom_psi_V}).
\end{proof}
\subsubsection{Proof of the infinitesimal estimation}
%
%
%
%
%
%
%
%
%
\begin{myprop}
\label{prop:saut_infinitesimal_approx}
Let $q>0$. \\

Assume that $b+\mathds{1}_{\pi(F) =+ \infty}\kappa_{q}$ has sublinear growth, and that (\ref{hyp:struc_param_saut_x_infini}) and (\ref{hyp:cv_zero_sublinear_jump}) hold. Moreover, when $\pi(F) =+ \infty$, assume that there exists $r \in [0,1/2]$ such that $\mathcal{H}^{r+\mathds{1}_{q\in(1/2,+\infty)}/2}$ (see (\ref{hyp:jump_component_ordre_p})) holds and that $\sup_{n \in \mathbb{N}^{\ast}} \nu_n^{\eta,q}( \tau_{r+\mathds{1}_{q\in(1/2,+\infty)}/2})< + \infty, \; a.s. $\\

 Then, $\mathcal{E}(\widetilde{A}_{q},A,\mathcal{C}^2_K(\mathbb{R}^d)) $ (see (\ref{hyp:erreur_tems_cours_fonction_test_reg})) is fulfilled, with $A$ defined in (\ref{eq:PDMP_generator}).
\end{myprop}
\begin{proof}
Let $f \!\in \mathcal{C}^2_K(\mathbb{R}^d)$. In this proof we will use the function $\omega_{b,q, \gamma}: \mathbb{R}^d \to \mathbb{R}^d$ such that $\omega_{b,q, \gamma}(x,v)=x+\gamma( b(x) +\mathds{1}_{\pi(F) =+ \infty}\kappa_{q}(x))$. Focusing on the jump component, we study
%
\begin{align*}
\gamma^{-1}\big (\mathbb{E}[f(\omega_{b,q,\gamma}(x) + M^{\gamma,q}_{\gamma})]-f(\omega_{b,q, \gamma}(x) \big)
\leqslant  A_2f(x)+R^1_{f,A_2}(x,\gamma)+R^2_{f,A_2}(x,\gamma)
\end{align*}
with $M^{\gamma,q}_{t}(x)= \mathds{1}_{q \in (0,1/2]} M^{\gamma}_{t} (x) +\mathds{1}_{q \in (1/2,+\infty)}( \mathds{1}_{\pi(F) =+ \infty} \widetilde{M}^{\gamma}_{t} (x)+\mathds{1}_{\pi(F) <+ \infty} M^{\gamma}_{t} (x)$, $t\geqslant 0$, and using the following representation (which follows from Remark \ref{rmk:representation_mesure_infinie}),
\begin{align*}
 A_2f(x)=&  \int_{F} (f(x+c(z,x))-f(x)-\mathds{1}_{\pi(F) =+ \infty}\mathds{1}_{q \in (1/2,+\infty)}  \langle \nabla f(x), c(z,x) \rangle )\zeta(z,x)\pi(dz)  \\
 R^1_{f,A_2}(x,\gamma)=& g_q(x) \int_{z \in F}\tilde{\mathbb{E}}[\tilde{\Lambda}^1_{f,A_2}(x,\gamma,\omega,z) ] \pi(dz)  ,\\
  R^2_{f,A_2}(x,\gamma)=& g_q(x) \int_{F } \tilde{\Lambda}^2_{f,A_2}(x,\gamma,z) \pi(dz) 
\end{align*}
%
%
%
where $g_q(x)=\mathds{1}_{\pi(F)<+\infty}+\mathds{1}_{\pi(F)=+\infty}(1\vee \tau_{r+\mathds{1}_{q\in(1/2,+\infty)}}(x))$ and $\tilde{\Lambda}^1_{f,A_2} (x,t,\omega,z)=\tilde{\mathcal{R}}^1_{f,A_2} (x,\gamma,M^{\gamma,q}_{\Theta(\tilde{\omega}) \gamma}(x,\tilde{\omega}),z)$ with $M^{\gamma,q}_{t}(x,\tilde{\omega})$, $t \geqslant 0$, following the same law under $\tilde{\mathbb{P}}$ as $M^{\gamma,q}_{t}(x,\tilde{\omega})$, $t \geqslant 0$, under $\mathbb{P}$, $\Theta \sim \mathcal{U}_{[0,1]}$ under $\tilde{\mathbb{P}}$,
\begin{align*}
\begin{array}{crcl}
\tilde{\mathcal{R}}^1_{f,A_2} & :  \mathbb{R}^d  \times \mathbb{R}_+  \times \mathbb{R}^d \times F   & \to & \mathbb{R}_+ \\
 &( x,\gamma,u,z) & \mapsto &   \frac{\mathds{1}_{F_{\gamma}}(z)}{g_q(x)} \big\vert f(\omega_{b_q,t}(x) +u+c(z,x))-f(\omega_{b, \gamma}(x) +u )\\
 &&& -  (f(x +c(z,x))-f(x ) )\\
 &&&-\mathds{1}_{\pi(F) =+ \infty}\mathds{1}_{q \in (1/2,+\infty)} \langle\nabla f(\omega_{b_q,\gamma}(x) )- \nabla f(x), c(z,x) \rangle \big\vert \zeta(z,x),
\end{array}
\end{align*}
and
\begin{align*}
\begin{array}{crcl}
\tilde{\Lambda}^2_{f,A_2} & :  \mathbb{R}^d  \times \mathbb{R}_+ \times F & \to & \mathbb{R}_+ \\
 &( x,\gamma,z) & \mapsto & \frac{\mathds{1}_F (z)-\mathds{1}_{F_{\gamma}}(z)}{g_q(x)} \\
 & & & \times \big\vert  f(x+c(z,x))-f(x)-\mathds{1}_{\pi(F) =+ \infty}\mathds{1}_{q \in (1/2,+\infty)} \langle \nabla f(x), c(z,x) \rangle  \big\vert  \zeta(z,x).
 \end{array}
\end{align*}
We show that $\mathcal{E}(\widetilde{A}_q,A,\mathcal{C}^2_K(\mathbb{R}^d))$ \ref{hyp:erreur_tems_cours_fonction_test_reg_Lambda_representation_1} holds for $ ( \tilde{\Lambda}^1_{f,A_2} ,g_q)$.\\

First we notice that from (\ref{hyp:cv_zero_sublinear_jump}), for every $\epsilon>0$, we have $\gamma_{0,\epsilon}(\tilde{\omega}):=1 \wedge \sup \{ \gamma > 0,\forall x \in \mathbb{R}^d,\sup_{t \in [0,\gamma] } \vert M^{\gamma,q}_t(x,\tilde{\omega}) \vert \leqslant \epsilon(1+ \vert x \vert)\}>0, \mathbb{P}(d\tilde{\omega})-a.s.$ (with convention $\sup \emptyset=-\infty$). We fix $z \in F$, $\theta \in [0,1]$ and since $b+\mathds{1}_{\pi(F) =+ \infty}\kappa_{q}$ has sublinear growth and (\ref{hyp:struc_param_saut_x_infini}) holds, it follows that $\tilde{\mathbb{P}}(d\tilde{\omega})-a.s.$, there exists $\gamma_0(z,\theta,\tilde{\omega})$ such that $ \lim\limits_{\vert x \vert \to +\infty}  \sup_{\gamma \in (0,\gamma_0(z,\theta,\tilde{\omega})]}   \tilde{\mathcal{R}}^1_{f,A_2} (x, \gamma, M^{\gamma,q}_{\gamma \theta}(x,\tilde{\omega}), z)=0  $, which yields $\mathcal{E}(\widetilde{A}_q,A,\mathcal{C}^2_K(\mathbb{R}^d))$ \ref{hyp:erreur_tems_cours_fonction_test_reg_Lambda_representation_1} ii) (see (\ref{hyp:erreur_temps_cours_fonction_test_reg_Lambda_representation_2_1})). \\
Moreover, we recall that $f$ is continuous with compact support, so it is uniformly continuous and (\ref{hyp:cv_zero_sublinear_jump}) holds. Therefore, for every compact subset $K$ of $\mathbb{R}^d$, we have
 \begin{align*}
 \forall (z,\theta) \in F \times[0,1], \quad  \lim\limits_{\gamma \to 0} \sup_{x \in K}  \tilde{\mathcal{R}}^1_{f,A_2} (x, \gamma, M^{\gamma,q}_{\gamma \theta}(x), z)=0    \quad \tilde{\mathbb{P}}(d\tilde{\omega})-a.s.
 \end{align*}
 Consequently $\mathcal{E}(\widetilde{A}_q,A,\mathcal{C}^2_K(\mathbb{R}^d))$ \ref{hyp:erreur_tems_cours_fonction_test_reg_Lambda_representation_1} i) (see (\ref{hyp:erreur_temps_cours_fonction_test_reg_Lambda_representation_2_1})) holds. 
 
 Now, using a similar approach, it follows from (\ref{hyp:struc_param_saut_x_infini}), $\cup_{\gamma>0}F_{\gamma}=F$, and the fact that $f$ has a compact support, that $\mathcal{E}(\widetilde{A}_q,A,\mathcal{C}^2_K(\mathbb{R}^d))$ \ref{hyp:erreur_temps_cours_fonction_test_reg_Lambda_representation_2} (see \ref{hyp:erreur_temps_cours_fonction_test_reg_Lambda_representation_2_2})) holds for $ ( \tilde{\Lambda}^2_{A_2} ,g_q)$.\\
 
  Finally, using Taylor expansions at order one and two, for every $r\in [0,1/2]$, we derive that for every $x,y \in \mathbb{R}^d$, $z \in F$,
  \begin{align*}
 f(y+c(z,x))-f(y)- & \mathds{1}_{\pi(F) =+ \infty}\mathds{1}_{q \in (1/2,+\infty)} \langle \nabla f(y), c(z,x) \rangle \\
  \leqslant & \mathds{1}_{\pi(F) <+ \infty}2\Vert f \Vert_{\infty} \\
&+\mathds{1}_{\pi(F) =+ \infty}\mathds{1}_{q \in (0,1/2]}(2 \Vert f \Vert_{\infty}) \vee (\Vert \nabla f \Vert_{\infty}  \vert c(z,x) \vert) \\
&+\mathds{1}_{\pi(F) =+ \infty}\mathds{1}_{q \in (1/2,+\infty)}  (2 \Vert \nabla f \Vert_{\infty}\vert c(z,x) \vert) \vee \big(\frac{1}{2}\Vert D^2 f \Vert_{\infty}  \vert c(z,x) \vert^2 \big)\\
\leqslant &  C (\mathds{1}_{\pi(F) <+ \infty}+\mathds{1}_{\pi(F) =+ \infty}( \mathds{1}_{q \in (0,1/2]}  \vert c(z,x) \vert^{2r} + \mathds{1}_{q \in (1/2,+\infty)}  \vert c(z,x) \vert^{1+2r}).
 \end{align*}
 Applying this estimation together with $\mathcal{H}^{r+\mathds{1}_{q\in(1/2,+\infty)}/2}$ (see (\ref{hyp:jump_component_ordre_p})), it follows that
 \begin{align*}
\int_F \tilde{\mathbb{E}}[ \sup_{x \in \mathbb{R}^d} \sup_{\gamma \in \mathbb{R}_+} \tilde{\Lambda}^1_{f,A_2}(x,\gamma,\tilde{\omega},z)  ] \pi(dz)+\int_F \sup_{x \in \mathbb{R}^d} \sup_{\gamma \in \mathbb{R}_+} \tilde{\Lambda}^2_{f,A_2}(x,\gamma,z)   \pi(dz) <+\infty.
 \end{align*}
 We gather all the terms together noticing that $\tilde{\Lambda}^i_{f,A_2}= \tilde{\Lambda}^i_{-f,A_2}$, $i \in \{1,2\}$, and the proof the infinitesimal control for the jump part is completed.
To complete the proof, it remains to study: $\gamma^{-1} (f(\omega_{b,q,\gamma}(x) - f(x) ).$ This proof appears as a simplified version of the proof of Proposition \ref{prop:milstein_infinitesimal_approx}, so we invite the reader to refer to this part of the paper for more details.
\end{proof}
\subsubsection{Proof of Growth control and Step Weight assumptions}
\begin{lemme}
\label{lemme:incr_lyapunov_X_saut}
 Let $q,\tilde{q}, p>0, a \in (0,1]$, $s \geqslant 1$, $\rho \in [1,2]$ and let $\psi_p(y)=y^p$ and $\phi(y)=y^a$. Suppose that (\ref{hyp:Lyapunov_control_saut}) holds. We have the following properties:
 \begin{enumerate}[label=\textbf{\Alph*.}]
 \item \label{lemme:incr_lyapunov_X_saut_point_A} Assume that $\mathcal{H}^{ \tilde{q}}$ (see (\ref{hyp:jump_component_ordre_p})) is satisfied and that, when $\pi(F)=+\infty$ and $\tilde{q}>1$, $\underline{\mathcal{H}}^{\tilde{q}}(\phi,V) $ (see (\ref{hyp:jump_component_control_coeff_ordre_p_control_stab_cas_inf})) holds. Then, for every $n \in \mathbb{N}$ we have: for every $f \in \DomA_0$,  
\begin{align}
\label{eq:incr_lyapunov_X_saut_f_DomA}
 \mathbb{E}[  \vert f(\overline{X}_{q,\Gamma_{n+1}})- f( \overline{X}^{1}_{q,\Gamma_{n+1}}) \vert^{\rho}\vert \overline{X}_{q,\Gamma_{n}}  ]    \leqslant  & C_f \gamma_{n+1}^{1 \wedge (\rho/(2\tilde{q}))}  \big( 1\vee \tau_{\tilde{q}}^{1 \wedge (\rho/(2\tilde{q}))}(\overline{X}_{q,\Gamma_n} ) \\
& \qquad \qquad \qquad +\mathds{1}_{\pi(F)=+\infty} \mathds{1}_{\tilde{q}>1} \phi \circ V(\overline{X}_{q,\Gamma_{n}})^{\rho/2}) . \nonumber
\end{align}
with $\DomA_0 =\mathcal{C}^2_K (\mathbb{R}^d )$ and notations (\ref{def:incr_pcmp_euler}). In other words, we have $\mathcal{GC}_{Q}(\DomA_0,  1\vee\tau_{\tilde{q}}^{1 \wedge (\rho/(2\tilde{q}))} ,\rho,\epsilon_{\mathcal{I}})  $ (see (\ref{hyp:incr_X_Lyapunov})) with $\epsilon_{\mathcal{I}}(\gamma)=\gamma^{1 \wedge (\rho/(2q))}$ for every $\gamma \in \mathbb{R}_+$. \\
%
%
%
\item \label{lemme:incr_lyapunov_X_saut_point_B} Let $\hat{q}\geqslant p$. Assume that $\mathcal{S} \mathcal{W}_{pol}(p,a,s,\rho)$ (see (\ref{hyp:control_step_weight_pol_saut})), (\ref{hyp:Lyapunov_control_saut}) and $\mathfrak{B}_{q}(\phi)$ (see (\ref{hyp:controle_coefficients_saut_p_q})) hold. Assume also that $\mathcal{H}^{\hat{q}}(\phi,V) $ (see (\ref{hyp:jump_component_control_coeff_ordre_p_control_stab})) is satisfied and that, when $\pi(F)=+\infty$ and $\hat{q}>1$, $\underline{\mathcal{H}}^{\hat{q}}(\phi,V) $ (see (\ref{hyp:jump_component_control_coeff_ordre_p_control_stab_cas_inf})) holds. Moreover when $2p>s$ assume also that $\tilde{q} \geqslant \rho/2$, that $\mathcal{H}^{\tilde{q}}(\phi,V) $ (see (\ref{hyp:jump_component_control_coeff_ordre_p_control_stab}))) is satisfied and that, when $\pi(F)=+\infty$ and $\tilde{q}>1$, $\underline{\mathcal{H}}^{\tilde{q}}(\phi,V) $ (see (\ref{hyp:jump_component_control_coeff_ordre_p_control_stab_cas_inf})) holds.
%
%
%
%
%
Then, for every $n \in \mathbb{N}$, we have 
\begin{align}
\label{eq:incr_lyapunov_X_f_tens_saut}
 \mathbb{E}[\vert  V^{p/s}(\overline{X}_{q,\Gamma_{n+1}})-V^{p/s}(\overline{X}_{q,\Gamma_{n}}) \vert^{\rho}\vert \overline{X}_{q,\Gamma_{n}}]  
\leqslant &  C V^{p+a-1}(\overline{X}_{q,\Gamma_{n}}) (\mathds{1}_{2p>s} \gamma_{n+1}^{\rho/(2(\tilde{q} \vee 1/2)) } \\
& \qquad  \qquad \qquad \qquad +\gamma_{n+1}^{ ( 2\wedge (1/\hat{q})) p\rho/s}) , \nonumber
\end{align}
In other words, we have $\mathcal{GC}_{Q}(V^{p/s},V^{p+a-1},\rho,\epsilon_{\mathcal{I}}) $ (see (\ref{hyp:incr_X_Lyapunov})) with $\epsilon_{\mathcal{I}}(\gamma)=\mathds{1}_{2p>s} \gamma^{\rho/(2(\tilde{q} \vee 1/2)) }+\gamma^{ ( 2\wedge (1/\hat{q})) p\rho/s}$ for every $\gamma \in \mathbb{R}_+$. \\
\end{enumerate}
\end{lemme}
\begin{proof}
We prove point \ref{lemme:incr_lyapunov_X_saut_point_A} Let $f \in \DomA$. We study $ \Delta \overline{X}^{2}_{q,{n+1}} $. We distinguish two cases: $\rho/2 \leqslant \tilde{q}$ and $\tilde{q}< \rho/2$. First, let $\rho/2 \leqslant \tilde{q}$. Using the Cauchy Schwartz inequality and Lemma \ref{lemme:control_jump} point \ref{lemme:control_jump_point_1} (see (\ref{eq:control_big_jump})), point \ref{lemme:control_jump_point_3_bis} (see (\ref{eq:control_small_jump_p_sup_un})) and point \ref{lemme:control_jump_point_2} (see (\ref{eq:control_jump_p<1}) and (\ref{eq:control_jump_p=1})) since we have $\mathcal{H}^{\tilde{q}}$ and when $\pi(F)=+\infty$ and $\tilde{q}>1$, we have $\underline{\mathcal{H}}^{\tilde{q}}(\phi,V) $ (see (\ref{hyp:jump_component_control_coeff_ordre_p_control_stab_cas_inf})), it follows that
\begin{align*}
\mathbb{E} \big[  \vert \Delta \overline{X}^{2}_{q,n+1}\vert ^{ \rho}  \vert \overline{X}_{q,\Gamma_n}  ] \leqslant \mathbb{E} \big[  \vert \Delta \overline{X}^{2}_{q,n+1}\vert ^{2 \tilde{q}}  ]^{\rho /(2\tilde{q})} \leqslant  C \gamma_{n+1}^{\rho/(2\tilde{q})} \big( \tau_{\tilde{q}}^{\rho/(2\tilde{q})}( \overline{X}_{q,\Gamma_n})  +\mathds{1}_{\pi(F)=+\infty} \mathds{1}_{\hat{q}>1} \phi \circ V(\overline{X}_{q,\Gamma_{n}})^{\rho/2} \big) ,
\end{align*}
and the result follows from the fact that $f$ is Lipschitz.\\
Now if $\tilde{q}\leqslant \rho/2$, then since $f$ is Lipschitz and defined on a compact set, it is also $2q/ \rho$-H\"older, and then for every $x_0 \in \mathbb{R}^d$, it follows from Lemma \ref{lemme:control_jump} point \ref{lemme:control_jump_point_2} (see (\ref{eq:control_small_jump_p_inf_un_demi}), (\ref{eq:control_jump_p<1}) and (\ref{eq:control_jump_p=1})) (since we have $\mathcal{H}^{\tilde{q}}$) that
\begin{align*}
\mathbb{E} [ \vert f(x_0+ \Delta \overline{X}^{2}_{q,n+1})-f(x_0 )\vert ^{\rho} \vert \overline{X}_{q,\Gamma_n} ]  \leqslant [ f ]_{2\tilde{q} / \rho}^{\rho}  \mathbb{E} [\vert  \Delta \overline{X}^{2}_{q,n+1}   \vert ^{2\tilde{q}} \vert \overline{X}_{q,\Gamma_n} ] \leqslant C_f \gamma_{n+1} \tau_{\tilde{q}}( \overline{X}_{q,\Gamma_n}),
\end{align*}
and gathering all the terms together yields  (\ref{eq:incr_lyapunov_X_saut_f_DomA}).\\
 We focus now on the proof of point \ref{lemme:incr_lyapunov_X_saut_point_B} (see (\ref{eq:incr_lyapunov_X_f_tens_saut})).
\paragraph{Case $2p\leqslant s$.}
 First, we assume that $2p/s \leqslant 1$ and we notice that  $V^{p/s}$ is $\alpha$-H\"older for any $\alpha \in [2p/s,1]$ (see Lemma 3. in \cite{Panloup_2008}) and then $V^{p/s}$ is $2p/s$-H\"older. We deduce from (\ref{eq:puisance_somme_n_terme}), the Cauchy-Schwartz inequality (since $\rho \leqslant s$ from $\mathcal{S} \mathcal{W}_{pol}(p,a,s,\rho)$ (see (\ref{hyp:control_step_weight_pol_saut}) and $p\leqslant \hat{q}$), Lemma \ref{lemme:control_jump} point \ref{lemme:control_jump_point_1}, point \ref{lemme:control_jump_point_3_bis} and point \ref{lemme:control_jump_point_2} (since we have $\mathcal{H}^{\hat{q}}$ and when $\pi(F)=+\infty$ and $\hat{q}>1$, $\underline{\mathcal{H}}^{\hat{q}}(\phi,V) $ (see (\ref{hyp:jump_component_control_coeff_ordre_p_control_stab_cas_inf})) holds) and $\mathfrak{B}_{q}(\phi)$ (see (\ref{hyp:controle_coefficients_saut_p_q})), that
\begin{align*}
\mathbb{E} [ \vert V^{p/s}( \overline{X}_{q,\Gamma_{n+1}} )-& V^{p/s}(\overline{X}_{q,\Gamma_n}) \vert^{\rho} \vert \overline{X}_{q,\Gamma_n} ] \\
\leqslant &2^{(2p\rho/s-1)_+}[V^{p/s}]^{\rho}_{2p/s} \Big(  \mathbb{E} \big[ \vert \Delta \overline{X}^{2}_{q,{n+1}}  \vert^{2 p \rho/s } \vert \overline{X}_{q,\Gamma_n} ] + \mathbb{E}[\vert   \Delta\overline{X}^{1}_{q,n+1}\vert^{2p \rho/s} \vert \overline{X}_{q,\Gamma_n}   \big)]  \Big) \\
\leqslant & C \Big( \mathbb{E} [ \vert \Delta \overline{X}^{2}_{q,n+1} \vert^{2 \hat{q} } \vert \overline{X}_{q,\Gamma_n} ]^{p\rho/(\hat{q}s)}  + \mathbb{E}[\vert   \Delta\overline{X}^{1}_{q,n+1}\vert^{2p \rho/s} \vert \overline{X}_{q,\Gamma_n}   \big)]  \Big)\\
\leqslant & C \gamma_{n+1}^{p\rho/(\hat{q}s)} \big( \tau_{\hat{q}}^{p\rho/(\hat{q}s)}( \overline{X}_{q,\Gamma_{n}} ) +\mathds{1}_{\pi(F)=+\infty} \mathds{1}_{\hat{q}>1} \phi \circ V(\overline{X}_{q,\Gamma_{n}})^{p\rho/s} \big)+C \gamma_{n+1}^{2 p\rho/s}  V^{a p\rho /s}(\overline{X}_{q,\Gamma_{n}}).
\end{align*}
In order to obtain (\ref{eq:incr_lyapunov_X_f_tens_saut}), it remains to use $\mathcal{H}^{\hat{q}}(\phi,V) $ (see (\ref{hyp:jump_component_control_coeff_ordre_p_control_stab})) and $a\,p \rho /s \leqslant a+p-1$.
\paragraph{Case 2p>s.}
Assuming now that $2p>s$ and using  (\ref{eq:puisssance_sup_1}) with $\alpha=2p/s$, it follows that
\begin{align*}
\big\vert V^{p/s}( \overline{X}_{q,\Gamma_{n+1}} )-V^{p/s}(\overline{X}_{q,\Gamma_n} ) \big\vert \leqslant &  2^{2p/s}p/s ( V^{p/s-1/2}(\overline{X}_{q,\Gamma_n} ) \vert \sqrt{V}( \overline{X}_{q,\Gamma_{n+1}})  - \sqrt{V} (\overline{X}_{q,\Gamma_n}) \vert \\
&+ \vert  \sqrt{V}( \overline{X}_{q,\Gamma_{n+1}})  - \sqrt{V}(\overline{X}_{q,\Gamma_n}) \vert^{2p/s}  ) \\
\leqslant &  2^{2p/s}p/s ( [ \sqrt{V}]_1 V^{p/s-1/2}(\overline{X}_{q,\Gamma_n} ) \vert  \overline{X}_{q,\Gamma_{n+1}}- \overline{X}_{q,\Gamma_n} \vert \\
&+ [ \sqrt{V}]_1^{2p/s}  \vert   \overline{X}_{q,\Gamma_{n+1}} -\overline{X}_{q,\Gamma_n}\vert^{2p/s}   ).
\end{align*}
Now, we study $\Delta \overline{X}^{2}_{q,n+1} $. We recall that $\rho \leqslant s$ from $\mathcal{S} \mathcal{W}_{pol}(p,a,s,\rho)$ (see (\ref{hyp:control_step_weight_pol_saut})) and $2\tilde{q} \geqslant \rho$. Using once again the Cauchy Schwartz inequality and Lemma \ref{lemme:control_jump} point \ref{lemme:control_jump_point_1}, point \ref{lemme:control_jump_point_3_bis} and point \ref{lemme:control_jump_point_2} (since we have $\mathcal{H}^{\hat{q}}$ (respectively $\mathcal{H}^{\tilde{q}}$) and when $\pi(F)=+\infty$ and $\hat{q}>1$ (resp. $\tilde{q}>1$), $\underline{\mathcal{H}}^{\hat{q}}(\phi,V) $ (see (\ref{hyp:jump_component_control_coeff_ordre_p_control_stab_cas_inf})) (resp. $\underline{\mathcal{H}}^{\tilde{q}}(\phi,V) $) holds), we derive
\begin{align*}
&\mathbb{E} [ \vert \Delta \overline{X}^{2}_{q,n+1} \vert^{ \rho} \vert \overline{X}_{q,\Gamma_n} ]\leqslant C \gamma_{n+1}^{\rho/(2 \tilde{q})} \big( \tau_{\tilde{q}}^{\rho/(2 \tilde{q})}( \overline{X}_{q,\Gamma_{n}} ) + \mathds{1}_{\pi(F)=+\infty} \mathds{1}_{\tilde{q}>1} \phi \circ V(\overline{X}_{q,\Gamma_{n}})^{p\rho/s} \big) , \quad \mbox{and} \\
&\mathbb{E} [ \vert \Delta \overline{X}^{2}_{q,n+1}   \vert^{2 p \rho/s } \vert \overline{X}_{q,\Gamma_n} ]\leqslant C \gamma_{n+1}^{p\rho/(\hat{q}s)} \big( \tau_{\hat{q}}^{p\rho/(\hat{q}s)}( \overline{X}_{q,\Gamma_{n}} ) + \mathds{1}_{\pi(F)=+\infty} \mathds{1}_{\hat{q}>1} \phi \circ V(\overline{X}_{q,\Gamma_{n}})^{p\rho/s} \big).
\end{align*}
%
%
%
%
Using $\mathfrak{B}_{q}(\phi)$ (see (\ref{hyp:controle_coefficients_saut_p_q})), $\mathcal{H}^{ \tilde{q}}(\phi,V) $ and $\mathcal{H}^{\hat{q}}(\phi,V) $ (see (\ref{hyp:jump_component_control_coeff_ordre_p_control_stab})), we obtain
\begin{align*}
&\mathbb{E}[\vert   \overline{X}_{q,\Gamma_{n+1}} -\overline{X}_{q,\Gamma_n}\vert^{\rho} \vert \overline{X}_{q,\Gamma_n}=x   ]  \leqslant  C \gamma_{n+1}^{\rho/(2(\tilde{q} \vee 1/2))} V^{ a \rho /2} (x), \quad \mbox{and} \\
&\mathbb{E}[\vert   \overline{X}_{q,\Gamma_{n+1}} -\overline{X}_{q,\Gamma_n}\vert^{2p\rho/s } \vert \overline{X}_{q,\Gamma_n}=x   ] \leqslant  C \gamma_{n+1}^{( 2\wedge (1/\hat{q}))p\rho/s}  V^{ a p \rho/s}(x).
\end{align*}
In order to obtain (\ref{eq:incr_lyapunov_X_f_tens_saut}), we observe that $a\,p \rho /s \leqslant a+p-1$.
\end{proof}
\subsubsection{Proof of Theorem \ref{th:cv_was_saut}}
We prove Theorem \ref{th:cv_was_saut} under $\widetilde{\mathcal{S} \mathcal{W}}_{\mbox{Jump}}(p,q,a,s,\rho)$ (see (\ref{hyp:control_step_weight_pol_saut_control_bis})) instead of $\mathcal{S} \mathcal{W}_{\mbox{Jump}}(p,q,a,s,\rho)$ (see (\ref{hyp:control_step_weight_pol_saut_control})) which is more general. The proof of Theorem \ref{th:cv_was_saut}, follows directly from Theorem \ref{th:tightness}, Theorem \ref{th:identification_limit}. The hypothesis of those theorems are given by Proposition \ref{prop:recursive_control_saut}, Proposition \ref{prop:saut_infinitesimal_approx}, Lemma \ref{lemme:incr_lyapunov_X_saut} and by Proposition \ref{prop:control_mesure_emprique_saut_pol} together with Remark \ref{rmk:control_mesure_emprique_saut_pol} which are given below.
\begin{myprop}
\label{prop:control_mesure_emprique_saut_pol}
Let $p > 0,a \in (0,1]$, $s \geqslant 1, \rho \in [1,2]$ and, $\psi_p(y)=y^p$, $\phi(y)=y^a$. Let $q_p\in[p,1]$ if $p\leqslant 1$ and $q_p=p$ if $p\geqslant 1$. Let $\alpha>0$ and $\beta \in \mathbb{R}$. Then, we have the following properties, 
\begin{enumerate}[label=\textbf{\Alph*.}]
\item \label{prop:control_mesure_emprique_saut_pol_point_1} Assume that $\mathfrak{B}_{q_p}(\phi)$ (see (\ref{hyp:controle_coefficients_saut_p_q})) and  $\mathcal{R}_{p,q_p}(\alpha,\beta,\phi,V)$ (\ref{hyp:recursive_control_param_saut}) hold. Assume also that:
 \begin{enumerate}[label=\textbf{\roman*.}]
\item \textbf{Case $p>1$ ($q_p=p$).} If $\pi(F)<+\infty$ assume that $\mathcal{H}^p(\phi,V) $ and $\mathcal{H}^{1}(\phi,V) $ (see (\ref{hyp:jump_component_control_coeff_ordre_p_control_stab})) are satisfied. If $\pi(F)=+\infty$ assume that $\underline{\mathcal{H}}^{p}(\phi,V) $ (see (\ref{hyp:jump_component_control_coeff_ordre_p_control_stab_cas_inf})).
\item \textbf{Case $p \leqslant 1$.} Assume that $\mathcal{H}^{q_p} $ (see (\ref{hyp:jump_component_ordre_p})) holds.\\
\end{enumerate}

Now let $\tilde{q}_1 \geqslant  \rho/2$, $\tilde{q}_2 \geqslant  p$ and $\epsilon_{\mathcal{I}}(\gamma)=\mathds{1}_{2p>s}\gamma^{\rho/(2(\tilde{q}_1\vee 1/2)) }+\gamma^{ ( 2\wedge (1/\tilde{q}_2)) p\rho/s}$. \\
Assume $\mathcal{S}\mathcal{W}_{\mathcal{I}, \gamma,\eta}(\rho, \epsilon_{\mathcal{I}})$ (see (\ref{hyp:step_weight_I})) and that $\mathcal{H}^{\tilde{q}_1}(\phi,V) $ (see (\ref{hyp:jump_component_control_coeff_ordre_p_control_stab})) when $2p>s$ (respectively $\mathcal{H}^{\tilde{q}_2}(\phi,V) $ for every $p>0$) holds and that when $\pi(F)=+\infty$ and $\tilde{q}_1>1$ (resp. $\tilde{q}_2>1$), we have $\underline{\mathcal{H}}^{\tilde{q_1}}(\phi,V) $ (see (\ref{hyp:jump_component_control_coeff_ordre_p_control_stab_cas_inf})) (resp. $\underline{\mathcal{H}}^{\tilde{q_2}}(\phi,V) $). \\

Then  $\mathcal{S}\mathcal{W}_{\mathcal{I}, \gamma,\eta}( V^{p +a-1},\rho,\epsilon_{\mathcal{I}}) $ (see (\ref{hyp:step_weight_I_gen_chow})) holds with $\overline{X}$ replaced by $\overline{X}_{q_p}$ and we have the following properties: \\
If, in addition, $\mathcal{S}\mathcal{W}_{\mathcal{II},\gamma,\eta}(V^{p/s}) $ (see (\ref{hyp:step_weight_I_gen_tens})) with $\overline{X}$ replaced by $\overline{X}_{q_p}$ and $\mathcal{S} \mathcal{W}_{pol}(p,a,s,\rho)$ (see (\ref{hyp:control_step_weight_pol_saut})) are satisfied, then
\begin{equation}
\label{eq:invariance_mes_emp_Lyap_gen_saut_pol}
\mathbb{P} \mbox{-a.s.} \quad  \sup_{n\in \mathbb{N}^{\ast} } - \frac{1}{H_n} \sum_{k=1}^n \eta_k \widetilde{A}_{q_p,\gamma_k} (\psi \circ V)^s (\overline{X}_{q_p,\Gamma_{k-1}})< + \infty,
\end{equation}
and we also have, 
\begin{align}
 \label{eq:tightness_saut}
\mathbb{P} \mbox{-a.s.} \quad \sup_{n \in \mathbb{N}^{\ast}} \nu_n^{\eta,q_p}( V^{p/s+a-1} ) < + \infty .
\end{align}
Moreover, when $p/s \leqslant p +a-1$, the assumption $\mathcal{S}\mathcal{W}_{\mathcal{II},\gamma,\eta}(V^{p/s}) $ (see (\ref{hyp:step_weight_I_gen_tens})) can be replaced by $\mathcal{S}\mathcal{W}_{\mathcal{II},\gamma,\eta} $ (see (\ref{hyp:step_weight_II})). Finally, if we also suppose that $\mbox{L}_{V}$ (see (\ref{hyp:Lyapunov})) holds and that $p/s+a-1 >0$, then $(\nu_n^{\eta})_{n \in \mathbb{N}^{\ast}}$ is $\mathbb{P}-a.s.$ tight.
\item \label{prop:control_mesure_emprique_saut_pol_point_2}
Let $\tilde{q}_3>0$ and let $\tilde{\epsilon}_{\mathcal{I}}(\gamma)=\gamma^{1 \wedge (\rho/(2\tilde{q}_3))}$. Assume that (\ref{hyp:accroiss_sw_series_2}) and $\mathcal{H}^{ \tilde{q}_3}$ (see (\ref{hyp:jump_component_ordre_p})) hold and that when $\pi(F)=+\infty$ and $\tilde{q}_3>1$, $\underline{\mathcal{H}}^{\tilde{q_3}}(\phi,V) $ (see (\ref{hyp:jump_component_control_coeff_ordre_p_control_stab_cas_inf})) holds. Assume also that $\mathcal{S}\mathcal{W}_{\mathcal{I}, \gamma,\eta}(  1\vee \tau_{\tilde{q}_3}^{1 \wedge (\rho/(2 \tilde{q}_3))} ,\rho,\tilde{\epsilon}_{\mathcal{I}}) $ (see (\ref{hyp:step_weight_I_gen_chow})) with $\overline{X}$ replaced by $\overline{X}_{q_p}$ holds. Then
\begin{equation}
\label{eq:invariance_mes_emp_Lyap_fonction_test_dom_saut_pol}
\mathbb{P} \mbox{-a.s.} \quad \forall f \in \DomA_0, \quad \lim\limits_{n\to +\infty } \frac{1}{H_n} \sum_{k=1}^n \eta_k \widetilde{A}_{q_p\gamma_k}f (\overline{X}_{q_p,\Gamma_{k-1}})=0
\end{equation}
\end{enumerate}
\end{myprop}
\begin{remark}
\label{rmk:control_mesure_emprique_saut_pol}
The reader may notice that (\ref{eq:invariance_mes_emp_Lyap_fonction_test_dom_saut_pol}) remains true if we replace $\mathcal{H}^{ \tilde{q}_3}$ by $\mathcal{H}^{ \tilde{q}_3}(\phi,V)$ and if we also replace $\mathcal{S}\mathcal{W}_{\mathcal{I}, \gamma,\eta}( 1\vee\tau_{\tilde{q}_3}^{1 \wedge \rho/(2\tilde{q}_3)}  ,\rho,\tilde{\epsilon}_{\mathcal{I}}) $ by $\mathcal{S}\mathcal{W}_{\mathcal{I}, \gamma,\eta}( V^{a (\tilde{q}_3 \wedge( \rho/2))},\rho,\tilde{\epsilon}_{\mathcal{I}}) $. A solution to obtain $\mathcal{S}\mathcal{W}_{\mathcal{I}, \gamma,\eta}( V^{a (\tilde{q}_3\wedge( \rho/2))},\rho,\tilde{\epsilon}_{\mathcal{I}}) $ when $a (\tilde{q}_3 \wedge( \rho/2)) \leqslant a+p-1$ is provided by point \ref{prop:control_mesure_emprique_saut_pol_point_1} that is $\mathcal{S}\mathcal{W}_{\mathcal{I}, \gamma,\eta}( V^{p +a-1},\rho,\tilde{\epsilon}_{\mathcal{I}}) $ which follows as soon as we suppose that  $\mathcal{S}\mathcal{W}_{\mathcal{I}, \gamma,\eta}(\rho, \tilde{\epsilon}_{\mathcal{I}})$ (see (\ref{hyp:step_weight_I})) holds. When $a (\tilde{q}_3 \wedge( \rho/2))>a+p-1$, a possible solution consists in replacing $p$ by $p_0$ in \ref{prop:control_mesure_emprique_saut_pol_point_1} with $p_0$ satisfying $a (\tilde{q}_3 \wedge( \rho/2)) \leqslant a+p_0-1$.
\end{remark}
\begin{proof}
The result is an immediate consequence of  Theorem \ref{th:tightness} and Theorem \ref{th:identification_limit}. It remains to check the assumptions of those Theorems. \\

We focus on the proof of (\ref{eq:invariance_mes_emp_Lyap_gen_saut_pol}) and (\ref{eq:tightness_saut}). First, we show $\mathcal{S}\mathcal{W}_{\mathcal{I}, \gamma,\eta}( V^{p +a-1},\rho,\epsilon_{\mathcal{I}}) $ (see (\ref{hyp:step_weight_I_gen_chow})). Since (\ref{hyp:Lyapunov_control_saut}), $\mathfrak{B}_{q_p}(\phi)$ (see (\ref{hyp:controle_coefficients_saut_p_q})) and $\mathcal{R}_{p,q_p}(\alpha,\beta,\phi,V)$ (see (\ref{hyp:recursive_control_param_saut})) hold, it follows (using the hypothesis from point \ref{prop:control_mesure_emprique_saut_pol_point_1})  from Proposition \ref{prop:recursive_control_saut} that $\mathcal{RC}_{Q,V}(\psi_p,\phi,p\tilde{\alpha},p\beta)$ (see (\ref{hyp:incr_sg_Lyapunov})) is satisfied for every $\tilde{\alpha} \in (0,\alpha)$ since $\liminf\limits_{y \to +\infty} \phi(y) > \beta / \tilde{\alpha}$. Then, using $\mathcal{S}\mathcal{W}_{\mathcal{I}, \gamma,\eta}(\rho, \epsilon_{\mathcal{I}})$ (see (\ref{hyp:step_weight_I})) with Lemma \ref{lemme:mom_V}  gives $\mathcal{S}\mathcal{W}_{\mathcal{I}, \gamma,\eta}( V^{p+a-1},\rho,\epsilon_{\mathcal{I}}) $ (see (\ref{hyp:step_weight_I_gen_chow})). In the same way, for $p/s \leqslant a+p-1$, we deduce from $\mathcal{S}\mathcal{W}_{\mathcal{II},\gamma,\eta} $ (see (\ref{hyp:step_weight_II})) and  Lemma \ref{lemme:mom_V} that $\mathcal{S}\mathcal{W}_{\mathcal{II},\gamma,\eta}(V^{p/s}) $ (see (\ref{hyp:step_weight_I_gen_tens})) holds. \\

 Now,we are going to prove $\mathcal{GC}_{Q}(V^{p/s},V^{a +p  -1},\rho,\epsilon_{\mathcal{I}}) $ (see (\ref{hyp:incr_X_Lyapunov})) and the proof of (\ref{eq:invariance_mes_emp_Lyap_gen_saut_pol}) will be completed. Notice that (\ref{eq:tightness_saut}) will follow from $\mathcal{RC}_{Q,V}(\psi_p,\phi,p\tilde{\alpha},p\beta)$ (see (\ref{hyp:incr_sg_Lyapunov})) and Theorem  \ref{th:tightness}. The proof is a consequence of Lemma \ref{lemme:incr_lyapunov_X_saut}. We notice indeed that Lemma \ref{lemme:incr_lyapunov_X_saut} (see (\ref{eq:incr_lyapunov_X_f_tens_saut})) implies assumption $\mathcal{GC}_{Q}(V^{p/s},V^{a +p -1},\rho,\epsilon_{\mathcal{I}}) $ (see (\ref{hyp:incr_X_Lyapunov})) and the proof of (\ref{eq:invariance_mes_emp_Lyap_gen_saut_pol}) and (\ref{eq:tightness_saut}) is completed.\\
 We complete the proof of the Proposition by noticing that (\ref{eq:invariance_mes_emp_Lyap_fonction_test_dom_saut_pol}) follows directly from Lemma \ref{lemme:incr_lyapunov_X_saut} (see (\ref{eq:incr_lyapunov_X_saut_f_DomA})).
\end{proof}

\bibliography{Biblio_these}

\def\cprime{$'$} \def\cprime{$'$}
\begin{thebibliography}{10}

\bibitem{Basak_Hu_Wei_1997}
G.K. Basak, I.~Hu, and C-Z Wei.
\newblock Weak convergence of recursions.
\newblock {\em Stochastic Processes and their Applications}, 68(1):65 -- 82,
  1997.

\bibitem{Bhattacharya_1982}
R.~N. Bhattacharya.
\newblock On the functional central limit theorem and the law of the iterated
  logarithm for markov processes.
\newblock {\em Zeitschrift f{\"u}r Wahrscheinlichkeitstheorie und Verwandte
  Gebiete}, 60(2):185--201, 1982.

\bibitem{DFMS_2004}
R.~Douc, G.~Fort, E.~Moulines, and P.~Soulier.
\newblock Practical drift conditions for subgeometric rates of convergence.
\newblock {\em Ann. Appl. Probab.}, 14(3):1353--1377, 08 2004.

\bibitem{Durmus_Moulines_2015}
A.~{Durmus} and E.~{Moulines}.
\newblock {Non-asymptotic convergence analysis for the Unadjusted Langevin
  Algorithm}.
\newblock {\em ArXiv e-prints}, July 2015.

\bibitem{Ethier_Kurtz_1986}
S.~N. Ethier and T.~G. Kurtz.
\newblock {\em Markov processes}.
\newblock Wiley Series in Probability and Mathematical Statistics: Probability
  and Mathematical Statistics. John Wiley \& Sons, Inc., New York, 1986.
\newblock Characterization and convergence.

\bibitem{Feller_1952}
W.~Feller.
\newblock The parabolic differential equations and the associated semi-groups
  of transformations.
\newblock {\em Annals of Mathematics}, 55(3):468--519, 1952.

\bibitem{Fournier_2002}
N.~Fournier.
\newblock Jumping sdes: absolute continuity using monotonicity.
\newblock {\em Stochastic Processes and their Applications}, 98(2):317--330,
  2002.

\bibitem{Ganidis_Roynette_Simonot_1999}
H.~Ganidis, B.~Roynette, and F.~Simonot.
\newblock Convergence rate of some semi-groups to their invariant probability.
\newblock {\em Stochastic Processes and their Applications}, 79(2):243--263,
  1999.

\bibitem{Giles_Szpruch_2014}
M.~B. Giles and L.~Szpruch.
\newblock Antithetic multilevel monte carlo estimation for multi-dimensional
  sdes without lévy area simulation.
\newblock {\em Ann. Appl. Probab.}, 24(4):1585--1620, 08 2014.

\bibitem{Hasminskii_1980}
R.J. Has'minskii.
\newblock {\em Stochastic stability of differential equations}, volume~7 of
  {\em Monographs and Textbooks on Mechanics of Solids and Fluids : Mechanics
  and Analysis}.
\newblock Sijthoff \& Noordhoff, Alphen aan den Rijn, 1980.

\bibitem{Ikeda_Watanabe_1989}
N.~Ikeda and S.~Watanabe.
\newblock {\em Stochastic differential equations and diffusion processes}.
\newblock Kodansha scientific books. North-Holland, 1989.

\bibitem{Lamberton_Pages_2002}
D.~Lamberton and G.~Pagès.
\newblock Recursive computation of the invariant distribution of a diffusion.
\newblock {\em Bernoulli}, 8(3):367--405, 04 2002.

\bibitem{Lamberton_Pages_2003}
D.~Lamberton and G.~Pagès.
\newblock Recursive computation of the invariant distrbution of a diffusion:
  The case of a weakly mean reverting drift.
\newblock {\em Stochastics and Dynamics}, 03(04):435--451, 2003.

\bibitem{Lemaire_thesis_2005}
V.~Lemaire.
\newblock {\em Estimation récursive de la mesure invariante d'un processus de
  diffusion}.
\newblock PhD thesis, 2005.
\newblock Thèse de doctorat dirigée par Lamberton, Damien et Pagès, Gilles
  Mathématiques appliquées Université de Marne-la-Vallée 2005.

\bibitem{Lemaire_2007}
Vincent Lemaire.
\newblock An adaptive scheme for the approximation of dissipative systems.
\newblock {\em Stochastic Processes and their Applications}, 117(10):1491 --
  1518, 2007.

\bibitem{Mei_Yin_2015}
H.~Mei and G.~Yin.
\newblock Convergence and convergence rates for approximating ergodic means of
  functions of solutions to stochastic differential equations with markov
  switching.
\newblock {\em Stochastic Processes and their Applications}, 125(8):3104 --
  3125, 2015.

\bibitem{Milstein_1987}
G.N. Milstein.
\newblock Weak approximation of solutions of systems of stochastic differential
  equations.
\newblock In {\em Numerical Integration of Stochastic Differential Equations},
  volume 313 of {\em Mathematics and Its Applications}, pages 101--134.
  Springer Netherlands, 1995.

\bibitem{Pages_2001_ergo}
G.~Pag{\`e}s.
\newblock Sur quelques algorithmes r\'ecursifs pour les probabilit\'es
  num\'eriques.
\newblock {\em ESAIM Probab. Statist.}, 5:141--170 (electronic), 2001.

\bibitem{Pages_Panloup_2009}
G.~Pagès and F.~Panloup.
\newblock Approximation of the distribution of a stationary markov process with
  application to option pricing.
\newblock {\em Bernoulli}, 15(1):146--177, 02 2009.

\bibitem{Pages_Panloup_2012}
G.~Pagès and F.~Panloup.
\newblock Ergodic approximation of the distribution of a stationary diffusion:
  Rate of convergence.
\newblock {\em Ann. Appl. Probab.}, 22(3):1059--1100, 06 2012.

\bibitem{Pages_Rey_2017}
G.~Pagès and C.~Rey.
\newblock Recursive computation of invariant distributions of feller processes.
\newblock 2017.

\bibitem{Panloup_2008_rate}
F.~Panloup.
\newblock Computation of the invariant measure for a lévy driven sde: Rate of
  convergence.
\newblock {\em Stochastic Processes and their Applications}, 118(8):1351 --
  1384, 2008.

\bibitem{Panloup_2008}
F.~Panloup.
\newblock Recursive computation of the invariant measure of a stochastic
  differential equation driven by a lévy process.
\newblock {\em Ann. Appl. Probab.}, 18(2):379--426, 04 2008.

\bibitem{Pazy_1992}
A.~Pazy.
\newblock {\em Semigroups of Linear Operators and Applications to Partial
  Differential Equations}.
\newblock Applied Mathematical Sciences. Springer New York, 1992.

\bibitem{Rabiet_2015}
V.~Rabiet.
\newblock {\em A stochastic equation with censored jumps related to multi-scale
  Piecewise Deterministic Markov Processes}.
\newblock PhD thesis, Universit\'e Paris Est, Marne-la-Vallée, 2015.

\bibitem{Soize_1994}
C.~Soize.
\newblock {\em The Fokker-Planck Equation for Stochastic Dynamical Systems and
  Its Explicit Steady State Solutions}.
\newblock Advanced Series on Fluid Mechanics. World Scientific, 1994.

\bibitem{Talay_1990}
D.~Talay.
\newblock Second-order discretization schemes of stochastic differential
  systems for the computation of the invariant law.
\newblock {\em Stochastics and Stochastic Reports}, 29(1):13--36, 1990.

\end{thebibliography}
\bibliographystyle{plain}

\end{document}